 \documentclass[a4paper,11pt]{article}
\usepackage{pdfsync}

 \usepackage[plainpages=false]{hyperref}
 \usepackage{amsfonts,amsmath,amssymb,amsthm}
 \usepackage{latexsym,lscape,rawfonts,mathrsfs}
  \usepackage{mathtools}

 \usepackage[dvips]{color}
 \usepackage{multicol}

 \usepackage[all]{xy}
 \usepackage{eufrak}
 \usepackage{makeidx}
 \usepackage{graphicx,psfrag}
 \usepackage{pstool}
 \usepackage{float}

 \usepackage{array,tabularx}

 \usepackage{setspace}

 \usepackage{appendix}
 
\usepackage{txfonts}

 \newcommand{\D}[2]{\ensuremath{ \frac{\partial{#1}}{\partial{#2}}}}

 \newcommand{\R}{\ensuremath{\mathbb{R}}}
 
 \newcommand{\CP}{\ensuremath{\mathbb{CP}}}

 \newcommand{\ba}{\begin{align*}}
 \newcommand{\ea}{\end{align*}}

 \DeclareMathOperator{\Vol}{Vol}
 \DeclareMathOperator{\diam}{diam}

 \newcommand{\norm}[2]{{ \ensuremath{\left\|} #1 \ensuremath{\right\|}}_{#2}}

 \makeatletter
 \def\ExtendSymbol#1#2#3#4#5{\ext@arrow 0099{\arrowfill@#1#2#3}{#4}{#5}}
 
 \makeatother

 \makeatletter
 \def\ExtendSymbol#1#2#3#4#5{\ext@arrow 0099{\arrowfill@#1#2#3}{#4}{#5}}
 \newcommand\longright[2][]{\ExtendSymbol{-}{-}{\rightarrow}{#1}{#2}}
 \makeatother

 \definecolor{orange}{rgb}{1,0.5,0}
 \definecolor{brown}{rgb}{0.48,0.33,0.19}
 \definecolor{magenta}{rgb}{1,0,1}
 \definecolor{miao}{cmyk}{0.5,0,0.2,0.2}
 \definecolor{qiao}{gray}{0.96}

\newtheorem{prop}{Proposition}[section]

\newtheorem{proposition}[prop]{Proposition}

\newtheorem{theorem}[prop]{Theorem}

\newtheorem{lemma}[prop]{Lemma}
\newtheorem{claim}[prop]{Claim}

\newtheorem{corollary}[prop]{Corollary}

\newtheorem{remark}[prop]{Remark}

\newtheorem{definition}[prop]{Definition}

\newtheorem{conjecture}[prop]{Conjecture}

\numberwithin{equation}{section}

\title{The local entropy along Ricci flow\\ \large---Part A: the no-local-collapsing theorems}
  
\author{Bing Wang} 
\date{}

\begin{document}
\maketitle

\begin{abstract}
 We localize the entropy functionals of G. Perelman and generalize his no-local-collapsing theorem and pseudo-locality theorem.
 Our generalization is technically inspired by further development of Li-Yau estimate along the Ricci flow. 
 It can be used to show the Gromov-Hausdorff convergence of  the K\"ahler Ricci flow on each minimal projective manifold of general type. 
\end{abstract}

\tableofcontents

\section{Introduction}

A Ricci flow solution $\{(M^m, g(t)), t \in I \subset \R\}$ is a smooth family of metrics satisfying the evolution equation
\begin{align}
    \frac{\partial}{\partial t} g = -2 Rc,  \label{eqn:ML27_1}
\end{align}
where $M^m$ is a complete manifold of dimension $m$.    
For simplicity of our discussion, we also assume that $\sup_{M} |Rm|_{g(t)} <\infty$ 
for each time $t \in I$. This condition holds automatically if $M$ is a closed manifold. 
It is very often to put an extra term on the right hand side of (\ref{eqn:ML27_1}) to obtain the following rescaled Ricci flow
\begin{align}
    \frac{\partial}{\partial t} g= -2 \left\{ Rc + \lambda(t)g \right\},   \label{eqn:ML27_2} 
\end{align}
where $\lambda(t)$ is a function depending only on time.  Typically, $\lambda(t)$ is chosen as the average of the scalar curvature, i.e., $\frac{1}{m} \fint R dv$ or some fixed constant independent of time. 
In the case that $M$ is closed and $\lambda(t)=\frac{1}{m} \fint R dv$, the flow (\ref{eqn:ML27_2}) is also called the normalized Ricci flow.  

The Ricci flow equations (\ref{eqn:ML27_1}) and (\ref{eqn:ML27_2}) were introduced by R. Hamilton in his seminal paper~\cite{Ha82}.
Starting from a positive Ricci curvature metric on a 3-manifold, he showed that the normalized Ricci flow exists forever and converges to a space form metric.  
Hamilton developed the maximum principle for tensors to study the Ricci flow initiated from some metric with positive curvature conditions.
Along this direction, there are various convergence theorems of the flow (\ref{eqn:ML27_2}),  proved by G. Huisken~\cite{Hu85}, R. Hamilton~\cite{Ha86}, Bohm-Wilking~\cite{BoWi}, etc.  
Such developments finally lead to the sphere theorem of Brendle-Schoen~\cite{BrSc},  which asserts that starting from a manifold whose Riemannian curvature is quater-pinched, 
the normalized Ricci flow (\ref{eqn:ML27_2}) converges to a round metric. 
For metrics without positive curvature condition,  the study of Ricci flow was profoundly affected by the celebrated work of G. Perelman~\cite{Pe1}. 
He introduced new tools, i.e., the entropy functionals $\boldsymbol{\mu}$, $\boldsymbol{\nu}$, the reduced distance and the reduced volume, to investigate the behavior of the Ricci flow. 
Perelman's new input enabled him to revive Hamilton's program of Ricci flow with surgery, leading to solutions of the Poincar\'{e} conjecture and Thurston's geometrization conjecture(c.f.~\cite{Pe1},~\cite{Pe2},~\cite{Pe3}).

In the general theory of the Ricci flow developed by Perelman in~\cite{Pe1},  the entropy functionals $\boldsymbol{\mu}$ and $\boldsymbol{\nu}$ are of essential importance.
Perelman discovered the monotonicity of such functionals and applied them to prove the no-local-collapsing theorem(c.f. Theorem 4.1 of~\cite{Pe1}), which removes the stumbling block for Hamilton's program 
of Ricci flow with surgery.  By delicately using such monotonicity, he further proved the pseudo-locality theorem(c.f. Theorem 10.1 and Theorem 10.3 of~\cite{Pe1}), which claims that the Ricci flow can not quickly turn an almost Euclidean region 
into a very curved one, no matter what happens far away.
Besides the functionals, Perelman also introduced the reduced distance and reduced volume.   In terms of them, the Ricci flow space-time admits a remarkable comparison geometry picture(c.f. Section 6 and Section 7 of~\cite{Pe1}),
which is the foundation of his ``local"-version of the no-local-collapsing theorem(c.f. Theorem 8.2 of~\cite{Pe1}). 
Each of the tools has its own advantages and shortcomings. 
The functionals $\boldsymbol{\mu}$ and $\boldsymbol{\nu}$ have the advantage that their definitions only require the information for each time slice $(M, g(t))$ of the flow.
However,  they are global invariants of the underlying manifold $(M, g(t))$. It is not convenient to apply them to study the local behavior around a given point $x$. 
Correspondingly, the reduced volume and the reduced distance reflect the natural comparison geometry picture of the space-time. 
Around a base point $(x,t)$, the reduced volume and the reduced distance are closely related to the ``local" geometry of $(x, t)$. 
Unfortunately,  it is the space-time ``local", rather than the Riemannian geometry ``local" that is concerned by the reduced volume and reduced geodesic.
In order to apply them, some extra conditions of the space-time neighborhood of $(x, t)$ are usually required.
However, such strong requirement of space-time is hard to fulfill.
Therefore, it is desirable to have some new tools to balance the advantages of the reduced volume, the reduced distance and the entropy functionals.
In this paper, we localize the functionals $\boldsymbol{\mu}$ and $\boldsymbol{\nu}$ for this purpose(c.f. Section~\ref{sec:localization}). 
On one hand, our local functionals enjoy similar geometric pictures of the reduced distance and the reduced volume.
On the other hand, they only require local information of single time-slices of the underlying flow.
It turns out that the localized functionals are convenient tools.  We shall apply them to generalize the no-local-collapsing theorem and the pseudo-locality theorem of Perelman~\cite{Pe1} in this paper and the forthcoming paper~\cite{BW2}.

Our study is motivated by the comparison geometry picture of the Ricci flow space-time. 
Let $(M^m, g)$ be a complete Ricci-flat manifold, $x_0$ is a point on $M$ such that $d(x_0, x)<A$.    
Suppose the ball $B(x_0, r_0)$ is $A^{-1}-$non-collapsed,  i.e., $r_0^{-m}|B(x_0, r_0)| \geq A^{-1}$, can we obtain uniform non-collapsing for the ball $B(x,r)$,  whenever $0<r<r_0$ and $d(x,x_0)<Ar_0$?
This question can be answered easily by applying triangle inequalities and Bishop-Gromov volume comparison theorems.
In particular, there exists a $\kappa=\kappa(m, A) \geq 3^{-m} A^{-m-1}$(c.f. Remark~\ref{rmk:CF18_1})
such that $B(x,r)$ is $\kappa$-non-collapsed, i.e., 
$r^{-m}|B(x,r)| \geq \kappa$.  Consequently, there is an estimate of propagation speed of non-collapsing constant on the manifold $M$.
This is easily illustrated by Figure~\ref{fig:comparison1}.

 \begin{figure}[H]
 \begin{center}
 \psfrag{A}[c][c]{\color{brown}{$B(x_0, r_0)$}}
 \psfrag{B}[c][c]{\color{red}{$B(x,r)$}}
 \includegraphics[width=0.4 \columnwidth]{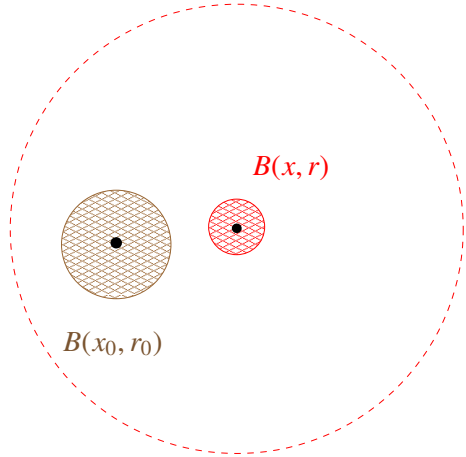}
 \caption{Propagation of non-collapsing on Ricci-flat manifolds}
 \label{fig:comparison1}
 \end{center}
 \end{figure}

Similar to the discussions in Section 2.7 of Chen-Wang~\cite{CW2},  we now regard $(M, g)$ as a trivial space-time $\{(M, g(t)), -\infty < t <\infty\}$ such that $g(t) \equiv g$. 
Clearly, $g(t)$ is a static Ricci flow solution by the Ricci-flatness of $g$.  
Then the above estimate can be explained as the propagation of volume non-collapsing constant on the space-time(c.f. Figure~\ref{fig:comparison2}).
However, in a more intrinsic way, it can also be interpreted as the propagation of non-collapsing constant  of  Perelman's reduced volume(c.f.~Section 7 of Perelman~\cite{Pe1} or Section~\ref{sec:reduced} of the current paper for a brief discussion of the reduced volume and the reduced distance). Recall that on the Ricci flat space-time, Perelman's reduced volume(c.f. equations (2.85) of Chen-Wang~\cite{CW2}) has a special formula

\begin{align*}
  \mathcal{V}( (x, t), r^2)=(4\pi)^{-\frac{m}{2}} r^{-m} \int_{M} e^{-\frac{d^2(y, x)}{4r^2}} dv_y,
\end{align*}
which is almost the volume ratio of $B_{g(t)}(x,r)$. 
On a general Ricci flow solution, the reduced volume is also well-defined and has monotonicity with respect to the parameter $r^2$, if one replace $\frac{d^2(y, x)}{4r^2}$ in the above formula by the reduced distance $l((x,t),(y,t-r^2))$. 
Therefore, via the comparison geometry of Bishop-Gromov type, one can regard a Ricci-flow as  an ``intrinsic-Ricci-flat" space-time.  
Consequently, the above reduced volume interpretation of non-collapsing propagation can be easily generalized to general Ricci flows, as done by Perelman in Section 8 of~\cite{Pe1}.

\begin{figure}[H]
 \begin{center}
 \psfrag{A}[c][c]{\color{brown}{$B(x_0, r_0)$}}
 \psfrag{B}[c][c]{\color{red}{$B(x,r)$}}
 \psfrag{C}[c][c]{$M$}
 \psfrag{D}[c][c]{$t$}
 \psfrag{E}[c][c]{$x_0$}
 \psfrag{F}[c][c]{$x$}
 \psfrag{G}[c][c]{$t=r_0^2$}
 \psfrag{H}[c][c]{$t=0$}
 \includegraphics[width=0.5 \columnwidth]{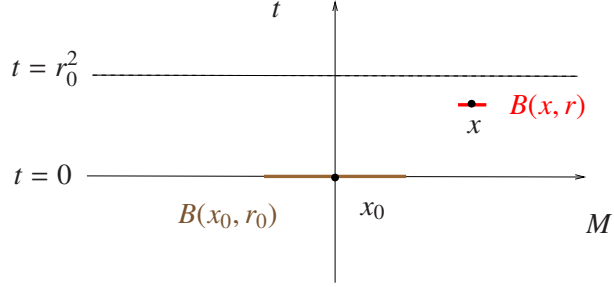}
 \caption{Propagation of non-collapsing on Ricci-flat space-time}
 \label{fig:comparison2}
 \end{center}
 \end{figure}

 However, the disadvantage of the reduced volume explanation(c.f. Theorem~\ref{thm:Pe8.2}) is also clear: it requires the curvature estimate in a whole space-time neighborhood around the point $(x,t)$, 
 rather than the scalar curvature estimate of a single time slice $t$.  We shall show that such strong requirement of space-time geometry is not necessary by the following version of  no-local-collapsing theorem(c.f. Figure~\ref{fig:step1}).

\begin{theorem}[\textbf{Improved version of no-local-collapsing}] 
  For every $A>1$ there exists $\kappa=\kappa(m,A)>0$ with the following property. 
  Suppose $\{(M^{m}, g(t)), 0 \leq t \leq  r_0^2\}$ is a solution of Ricci flow (\ref{eqn:ML27_1}) such that
  \begin{align}
   r_0^2 |Rm|(x,t) \leq m^{-1}, \quad \forall \; x \in B_{g(0)}(x_0, r_0),  \; 0 \leq t \leq r_0^2;   \quad   r_0^{-m} \left| B_{g(0)}(x_0, r_0) \right|_{dv_{g(0)}} \geq A^{-1}.  \label{eqn:ML27_3}
  \end{align}
  Then we have
  \begin{align}
    r^{-m} \left|B_{g(t)}(x,r) \right|_{dv_{g(t)}} \geq \kappa      \label{eqn:ML30_1}
  \end{align}
  whenever $A^{-1}r_0^2 \leq t \leq r_0^2$, $0<r \leq r_0$, and $B_{g(t)}(x,r) \subset B_{g(t)}(x_0,Ar_0)$ is a geodesic ball satisfying $r^2R(\cdot, t) \leq 1$. 
\label{thmin:ML14_1}
\end{theorem}

In Theorem~\ref{thmin:ML14_1}, we replace the requirement of space-time condition by a time-slice condition around $(x,t)$. 
Namely, we only need $R(\cdot, t) \leq r^{-2}$ in the ball $B_{g(t)}(x,r)$ to conclude the non-collapsing of $B_{g(t)}(x,r)$ whenever $(x,t)$ is not very far away from $(x_0, 0)$. 
However, we still require space-time condition (\ref{eqn:ML27_3}) around the base $(x_0, 0)$, which looks artificial. 
A condition depending only on the initial metric $g(0)$ should be more natural.  The quest of such a natural condition leads us to develop the pseudo-locality theorems, 
which unite and improve the  pseudo-locality theorems(c.f. Theorem 10.1 and Corollary 10.3 of~\cite{Pe1}) of Perelman 
and a similar pseudo-locality theorem(c.f. Proposition 3.1 and Theorem 3.1 of~\cite{TiWa}) of  Tian and the author. 
The details of the pseudo-locality theorems will be discussed in the second paper of this series~\cite{BW2}. 

  \begin{figure}[H]
 \begin{center}
 \psfrag{A}[c][c]{\color{blue}{Bounded geometry around the base}}
 \psfrag{B}[c][c]{\color{red}{Non-collapsing time-slice off the base}}
 \psfrag{C}[c][c]{$M$}
 \psfrag{D}[c][c]{$t$}
 \psfrag{E}[c][c]{$t=r_0^2$}
 \psfrag{F}[c][c]{$t=A^{-1}r_0^2$}
 \psfrag{G}[c][c]{$(x_0, 0)$}
 \psfrag{H}[c][c]{$(x, t)$}
 \includegraphics[width=0.5 \columnwidth]{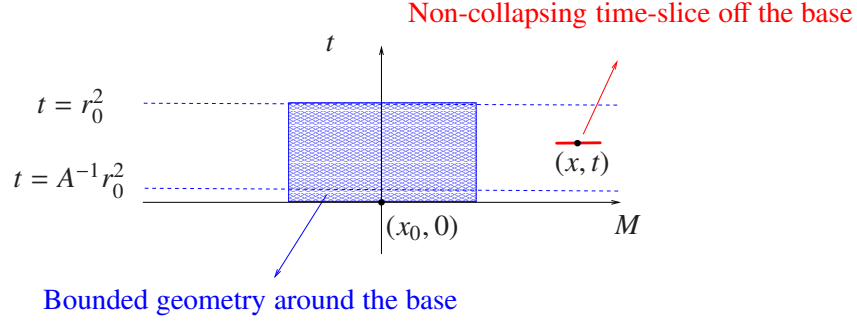}
 \caption{Propagation of non-collapsing when base neighborhood has bounded geometry}
 \label{fig:step1}
 \end{center}
 \end{figure}

In order to prove Theorem~\ref{thmin:ML14_1},  we need to revisit the work of Perelman~\cite{Pe1} and reorganize it from its starting point: 
the functionals $\boldsymbol{\mu}$ and $\boldsymbol{\nu}$.  We start from localizing them to each bounded domain $\Omega \subset M$ and study the properties of the minimizer functions.
Recall that on a Riemannian manifold $(M, g)$, for each positive number $\tau$, the functional $\boldsymbol{\mu}(g, \tau)$ is defined as the infimum of 
$$\mathcal{W}(g, \varphi,\tau)=-m -\frac{m}{2} \log (4\pi \tau) + \int_{\Omega} \left\{ \tau\left(R \varphi^2+\left|\nabla \varphi \right|^2 \right)-2\varphi^2 \log \varphi \right\}dv$$
among all smooth nonnegative functions $\varphi$ satisfying $\int_M \varphi^2 dv=1$. 
Now we let the infimum be achieved among all the $\varphi$'s satisfying an extra condition $\varphi \in C_c^{\infty}(\Omega)$. 
The infimum is denoted by $\boldsymbol{\mu}(\Omega, g, \tau)$, which is a functional of the domain $\Omega$.  Then we set $\displaystyle \boldsymbol{\nu}(\Omega, g, \tau)=\inf_{0<s \leq \tau} \boldsymbol{\mu}(\Omega, g, s)$. 
We call $\boldsymbol{\mu}(\Omega, g, \tau)$ and $\boldsymbol{\nu}(\Omega, g, \tau)$ as the local functionals, or the functionals localized on $\Omega$. 
Although not written down explicitly, it is implied by the work of Perelman(c.f.~\cite{Pe1}) that both $\boldsymbol{\mu}(\Omega, g, \tau)$ and $\boldsymbol{\nu}(\Omega, g, \tau)$ reveal information of the underlying geometry of $(\Omega, g)$.
In particular, if $\Omega$ is a geodesic ball of radius $r$ and the scalar curvature satisfies $R \leq r^{-2}$ in $\Omega$, then the uniform lower bound of $\boldsymbol{\nu}(\Omega, g, r^2)$ implies a uniform lower bound
of the volume ratio of $\Omega$(c.f. Theorem~\ref{thm:CF21_3}, Remark~\ref{rmk:MJ16_3} and Remark~\ref{rmk:MJ16_2}). 
Our new observation is that the minimizer function $\varphi$ of $\boldsymbol{\mu}(\Omega, g, \tau)$ carries more complete information than that of the value $\boldsymbol{\mu}(\Omega, g, \tau)$.
We could study the local geometry via the study of the the local minimizer functions, in particular under the Ricci flow evolution. 
For example, in the Ricci flow space-time, we focus our attention on a bounded domain $\Omega$ and time $T$, and choose $\varphi(T)$ as the minimizer function of  $\boldsymbol{\mu}(\Omega, g(T), \tau_T)$ for some $\tau_T>0$.
Let $u(T)=\varphi^2(T)$ and $u$ be the solution of the conjugate heat equation $\square^*u=(\partial_t - \Delta + R)u=0$.  
Similar to the Harnack inequality of Perelman, we have (c.f. Theorem~\ref{thm:CA02_2} for full details)
 \begin{align}
    v=\{(\tau_T+T-t) (R+2\Delta f - |\nabla f|^2) +f-m-\boldsymbol{\mu}\} u \leq 0,     \label{eqn:ML19_1}
 \end{align}
 where $\boldsymbol{\mu}$ is the value of $\boldsymbol{\mu}(\Omega, g(T), \tau_T)$ and $f=-\log u-\frac{m}{2} \log \{4\pi (\tau_T+T-t)\}$.
 The proof of (\ref{eqn:ML19_1}) follows from Section 9 of Perelman~\cite{Pe1} intuitively. 
 However, we need to deal with extra technical difficulties and regularity issues caused by $\partial \Omega$. 
 By careful heat kernel estimates, we show that the intuitive argument can actually be made rigorous.
 It is not hard to see(c.f. Remark~\ref{rmk:MJ13_3}) that (\ref{eqn:ML19_1}) is a generalization of Perelman's Harnack inequality.

 Note that (\ref{eqn:ML19_1}) provides a bridge between different domains at different time slices, via the study of the evolution of $u$. 
 In fact, for each domain $\Omega_0 \subset M$, we can study the restriction of $u$ on $\Omega_0$ to obtain the relationship between
  $\boldsymbol{\mu}(\Omega_0, g(0), \tau_T+T)$ and $\boldsymbol{\mu}(\Omega, g(T), \tau_T)$.
 Suppose $\Omega_0' \Subset \Omega_0$ and $h$ is a cutoff function which vanishes outside $\Omega_0$ and equals $1$ inside $\Omega_0'$,  
  the relationship can be calculated explicitly(c.f. (\ref{eqn:CF04_3}) in Theorem~\ref{thm:CF04_1}):
  \begin{align}
   \boldsymbol{\mu}(\Omega_0, g(0), \tau_T+T)- \boldsymbol{\mu}(\Omega, g(T), \tau_T) \leq  \left( 4\tau_0 C_h + e^{-1}\right)  \cdot \frac{\int_{\Omega_0 \backslash \Omega_0'} u}{\int_{\Omega_0'} u} 
 \label{eqn:CF14_1}  
 \end{align}
 where $C_h=\sup_{\Omega_0} \left|\nabla \sqrt{h} \right|^2$.
 It can be chosen as  $4r^{-2}$ if $\Omega_0$ is a ball of radius $2r$. 
 As $\int_{\Omega_0} u \leq 1$, the only nontrivial estimate in the above inequality is the lower bound of $\int_{\Omega_0'} u$.   
 Once we obtain uniform lower bound of $\int_{\Omega_0'} u$ which we denote by $c_u$,  (\ref{eqn:CF14_1}) implies (c.f. (\ref{eqn:CF05_1}) in Theorem~\ref{thm:CF07_1}) that
 \begin{align}
  \boldsymbol{\nu}(\Omega_0, g(0), \tau_T+T)- \boldsymbol{\nu}(\Omega, g(T), \tau_T) \leq  \left( 4\tau_0 C_h + e^{-1}\right) \left( c_u^{-1} -1 \right). 
 \label{eqn:CF14_2} 
 \end{align}
 Therefore,  $\boldsymbol{\nu}(\Omega, g(T), \tau_T)$ can be bounded from below by some number determined by $c_u$ and $ \boldsymbol{\nu}(\Omega_0, g(0), \tau_T+T)$.

 Now  the proof of Theorem~\ref{thmin:ML14_1} is clear.  Without loss of generality, we set $t=r_0=1$ and $r \in (0,1)$ and $\tau_1=r^2$.
 We set $\Omega=B_{g(1)}(x,r)$ and $\Omega_0 =B_{g(0)}(x_0, 0.1)$. 
 Based on reduced distance estimate similar to that in Section 8 of~\cite{Pe1} and the fact that $u \geq (4\pi [T-t])^{-\frac{m}{2}} e^{-l}$, we obtain a uniform lower bound of $c_u$.
 On the other hand, the uniformly bounded local geometry around $x_0$ provides a uniform lower bound of $\boldsymbol{\nu}(\Omega_0, g(0), \tau_1+1)$. 
 Therefore, we have a uniform lower bound of $\boldsymbol{\nu}(\Omega, g(1), \tau_1)$.
 However, the lower bound of $\boldsymbol{\nu}(\Omega, g(1), \tau_1)$ explicitly implies a non-collapsing constant(c.f. Theorem~\ref{thm:CF21_3}) if $R \leq r^{-2}$ inside $\Omega$.
 This finishes the proof. 
 
 Using the same idea, we indeed have a formula much more precise than the one stated in Theorem~\ref{thmin:ML14_1},  under much weaker conditions(c.f. (\ref{eqn:CE24_2}) in Theorem~\ref{thm:CE26_1}).
 Translating the information contained in $\boldsymbol{\nu}$ to volume ratios, we obtain an explicit formula(c.f. (\ref{eqn:CF09_4}) in Theorem~\ref{thm:CF03_1} and Remark~\ref{rmk:CF18_1}) of the propagation speed of the non-collapsing constant.
 We believe such precise formulas will be useful in the further study of the Ricci flow(collapsing case in particular), although the rough estimate in Theorem~\ref{thmin:ML14_1} is enough for many of our applications. 
 
 If one would like to sacrifice preciseness, there is an alternative shorter proof of Theorem~\ref{thmin:ML14_1}. 
 Actually, it follows from (\ref{eqn:ML19_1}) that
  \begin{align}
      \boldsymbol{\mu}   \geq (\tau_T+T-t) (R+2\Delta f - |\nabla f|^2) +f-m.   
  \label{eqn:CF14_3}   
  \end{align}
  The uniformly bounded geometry around $(x_0, 0)$ and the uniform lower bound of $u$ implies a uniform two sided bound of $u$ around $(x_0, 0.1)$,
   by the mean value formula 
  and the Harnack inequality of positive heat solutions(c.f. Theorem 10 on p.391 of L.C. Evans~\cite{Evans}).
  Also, note that all higher curvature derivatives around $(x_0, 0.1)$ are bounded by Shi's estimate(c.f.~\cite{Shi}, or chapter 6 of~\cite{CLN}).  Then the relationship $u=(4\pi [\tau_T+T-t])^{-\frac{m}{2}}e^{-f}$ implies that the right hand side of (\ref{eqn:CF14_3})
  is uniformly bounded from below. 

Although not natural in general Riemannian setting, the conditions (\ref{eqn:ML27_3}) in Theorem~\ref{thmin:ML14_1} are available whenever we study specific types of K\"ahler Ricci flow, up to an elementary parabolic rescaling. 
In fact, (\ref{eqn:ML27_3}) can often be obtained by regularity theory of parabolic Monge-Amp\`{e}re equation, which is deeply affected the fundamental work of Yau~\cite{Yau}. 
Therefore, it seems reasonable to believe that Theorem~\ref{thmin:ML14_1} will be useful in the study of general K\"ahler Ricci flow.  
As an evidence, we apply Theorem~\ref{thmin:ML14_1} to show the following convergence theorem.

\begin{theorem}[\textbf{Convergence of the K\"ahler Ricci flow}]
Let $X$ be a minimal projective manifold of general type.  Starting from a K\"ahler metric $g_0$, the flow solution of 
\begin{align*}
\displaystyle \partial_t g=-\left\{ Rc + g\right\}
\end{align*}
converges to the unique singular K\"ahler-Einstein metric $\omega_{KE}$ on the canonical model $X_{can}$ in the Gromov-Hausdorff  topology as $t \rightarrow \infty$.
\label{thmin:ML28_1}
\end{theorem}

Theorem~\ref{thmin:ML28_1} confirms a long-standing conjecture(c.f. Conjecture~\ref{cje:GB03_1}),  whose low dimensional cases were confirmed by Guo-Song-Weinkove~\cite{GSW} in dimension 2, and Tian-Zhang~\cite{TiZhL} in dimension $\leq 3$ by
different methods. The key for the proof of Theorem~\ref{thmin:ML28_1} is to develop a uniform $\kappa$-non-collapsing estimate and a uniform diameter bound along the flow.   
Similar estimates along the Fano K\"ahler Ricci flow were discovered by Perelman(c.f. Remark~\ref{rmk:CF25_1}). 
It is interesting to observe that the statement of Theorem~\ref{thmin:ML28_1} mirrors that of the Fano K\"ahler Ricci flow(c.f. Remark~\ref{rmk:ML26_1}).
We provide the proof of Theorem~\ref{thmin:ML28_1} and necessary background and references in Section~\ref{sec:krf}.  

 This paper is organized as follows. 
 In Section~\ref{sec:localization}, we localize Perelman's functionals $\boldsymbol{\mu}$ and $\boldsymbol{\nu}$, together with other 
 closely related functionals $\bar{\boldsymbol{\mu}}$, $\bar{\boldsymbol{\nu}}$. We discuss the basic properties of the localized functionals and the minimizer functions. 
 In Section~\ref{sec:vratios}, we study the relationships among $\boldsymbol{\nu}$, $\bar{\boldsymbol{\nu}}$ and the volume ratios.  
 In Section~\ref{sec:LYHP}, we generalize the Harnack inequality of Perelman and provide an alternative approach to understand the meaning of
 Perelman's reduced distance, via Li-Yau's Harnack estimate.
 In Section~\ref{sec:almostmon}, we derive effective monotonicity formulas for local $\boldsymbol{\mu}$ and $\boldsymbol{\nu}$-functionals and consequently deduce one version of no-local-collapsing theorem.
 In Section~\ref{sec:reduced},  we generalize the reduced distance  and the reduced volume density functions to be defined from a probability measure and develop effective lower bound of the reduced volume density function. 
 In Section~\ref{sec:alter}, by combining the generalized reduced volume density function estimate with the generalized Harnack inequality, we can estimate the propagation speed of the local $\boldsymbol{\nu}$-functionals.
 Such an estimate in turn implies a strong version of the no-local-collapsing theorem, i.e.  Theorem~\ref{thmin:ML14_1}.   
 Finally, in Section~\ref{sec:krf}, we show the uniform $\kappa$-non-collapsing estimate and the uniform diameter bound along each K\"ahler Ricci flow
  on a minimal projective manifold of general type and consequently prove Theorem~\ref{thmin:ML28_1}. \\

{\bf Acknowledgements}:  
This paper is partially supported by NSF grant DMS-1510401.  The author would also like to acknowledge the invitation to MSRI Berkeley in spring 2016 supported by NSF grant DMS-1440140.
Part of this work was done while the author was visiting AMSS(Academy of Mathematics and Systems Science) in Beijing and USTC(University of Science and Technology of China) in Hefei, during the summer of 2016. 
 He wishes to thank AMSS and USTC  for their hospitality. He would also like to thank Mikhail Feldman, Jeff Viaclovsky, Lu Wang and Shaosai Huang for helpful discussions.

\section{Localization of Perelman's functionals}
\label{sec:localization}

Let $(M, g)$ be a complete Riemannian manifold of dimension $m$, and $\Omega$ be a connected, open subset of $M$ with smooth boundary.
Then we can regard $(\Omega, \partial \Omega, g)$ as a smooth manifold with boundary.  Let $\boldsymbol{a}$ be a smooth function on $\bar{\Omega}$, and $\tau$ be a positive constant.
Then we define
\begin{align}
  &\mathscr{S}(\Omega) \coloneqq \left\{ \varphi \left| \varphi \in W_0^{1,2}(\Omega), \quad \varphi \geq 0,  \quad \int_{\Omega} \varphi^2 dv=1 \right. \right\}, \label{eqn:MJ16_a}\\
  &\mathcal{W}^{(\boldsymbol{a})}(\Omega, g, \varphi, \tau) \coloneqq -m -\frac{m}{2} \log (4\pi \tau) + \int_{\Omega} \left\{ \tau\left(\boldsymbol{a}\varphi^2+4\left|\nabla \varphi \right|^2 \right)-2\varphi^2 \log \varphi \right\}dv, \label{eqn:MJ16_b} \\
  &\boldsymbol{\mu}^{(\boldsymbol{a})} \left( \Omega, g, \tau \right) \coloneqq \inf_{\varphi \in \mathscr{S}(\Omega)} \mathcal{W}^{(\boldsymbol{a})}(\Omega, g, \varphi, \tau), \label{eqn:MJ16_c}\\
  &\boldsymbol{\nu}^{(\boldsymbol{a})} \left( \Omega, g, \tau \right) \coloneqq \inf_{s \in (0, \tau]} \boldsymbol{\mu}^{(\boldsymbol{a})} \left( \Omega, g, s \right),  \label{eqn:MJ16_d}\\
  &\boldsymbol{\nu}^{(\boldsymbol{a})} \left( \Omega, g\right) \coloneqq \inf_{\tau \in (0, \infty)} \boldsymbol{\mu}^{(\boldsymbol{a})} \left( \Omega, g, \tau \right).  \label{eqn:MJ16_e}
\end{align} 
By the result of O. Rothaus(c.f.~\cite{ROS}), we know that for each smooth function $\boldsymbol{a}$ and positive number $\tau>0$, 
$\boldsymbol{\mu}^{(\boldsymbol{a})}(\Omega, g, \tau)$ is achieved by a function $\varphi \in W_{0}^{1,2}(\Omega)$ whenever $\Omega$ is bounded.    
Moreover, $\varphi$ is positive and smooth in $\Omega$, and $\varphi$ satisfies the following Euler-Lagrangian equation
\begin{align}
  -4\tau \Delta \varphi  + \tau \boldsymbol{a} \varphi -2\varphi \log \varphi - \left( \boldsymbol{\mu}^{(\boldsymbol{a})} +m+\frac{m}{2}\log (4\pi \tau) \right) \varphi=0.
  \label{eqn:GH01_1}  
\end{align}
We call $\varphi$ as the minimizer function of $\boldsymbol{\mu}^{(\boldsymbol{a})}(\Omega, g, \tau)$.   Since in our setting, $\partial \Omega$ is smooth, we can say more about the boundary behavior of $\varphi$. 
Note that if $\boldsymbol{a}=R$, $\Omega=M$ and let $\varphi^2=(4\pi \tau)^{-\frac{m}{2}} e^{-f}$, then we have
\begin{align*}
  \mathcal{W}^{(\boldsymbol{a})}(\Omega, g, \varphi, \tau)=\int_{M} \left\{ \tau(R+|\nabla f|^2) +f -m \right\} (4\pi \tau)^{-\frac{m}{2}} e^{-f} dv,
\end{align*}
which is the functional introduced by Perelman(c.f. (3.1) of~\cite{Pe1}).  If $\boldsymbol{a}=0$, $\Omega=M$, the corresponding functionals are the ones studied by L. Ni(c.f. (1.2) and (1.7) of~\cite{Ni}).

The functional $\boldsymbol{\mu}^{(\boldsymbol{a})}(\Omega, g, \tau)$ reveals the information of the Riemannian geometry of $(\Omega, \partial \Omega, g)$, by choosing different $\boldsymbol{a}$.
If $\boldsymbol{a}=0$, then $\boldsymbol{\mu}^{(0)}(\Omega, g, \tau)$ is exactly the classical Logarithmic Sobolev constant(c.f.~\cite{Gross},~\cite{Da}).
In~\cite{Pe1}, Perelman choose $\boldsymbol{a}$ as the scalar curvature function $R$ and discovered the monotonicity of the functional $\boldsymbol{\mu}^{(R)}(M, g(t), T-t)$, 
which plays a foundational role in his celebrated resolution of Poincar\'{e} conjecture and the geometrization conjecture(c.f.~\cite{Pe1},~\cite{Pe2},~\cite{Pe3}). 

The cases $\boldsymbol{a}=0$ and $R$ are the most important cases for the application of $\boldsymbol{\mu}^{(\boldsymbol{a})}$ in the study of the Ricci flow.
However, most discussion of $\boldsymbol{\mu}^{(\boldsymbol{a})}(\Omega, g, \tau)$  in the current literature of Ricci flow focuses on the closed manifold case.
In this paper, we shall pay our attention to manifolds with boundary.  For simplicity of notation, we define
\begin{align}
   &\boldsymbol{\mu} \coloneqq \boldsymbol{\mu}^{(R)}, \quad \boldsymbol{\nu} \coloneqq  \boldsymbol{\nu}^{(R)};   \label{eqn:MJ16_1}\\
   &\bar{\boldsymbol{\mu}} \coloneqq  \boldsymbol{\mu}^{(0)}, \quad \bar{\boldsymbol{\nu}} \coloneqq  \boldsymbol{\nu}^{(0)}.  \label{eqn:MJ16_2}
\end{align}

We list some elementary properties of the functionals $\boldsymbol{\mu}^{(\boldsymbol{a})}, \boldsymbol{\nu}^{(\boldsymbol{a})}$. 
For simplicity, the metric $g$ will not appear explicitly when it is clear in the context.

\begin{proposition}[\textbf{Monotonicity induced by inclusion}]
  Suppose $\Omega_1, \Omega_2$ are bounded domains of $M$ satisfying $\Omega_1 \subsetneq \Omega_2$.  Then we have
  \begin{align}
   &\boldsymbol{\mu}^{(\boldsymbol{a})}(\Omega_1, \tau) > \boldsymbol{\mu}^{(\boldsymbol{a})}(\Omega_2, \tau),  \label{eqn:CA01_1}\\
   &\boldsymbol{\nu}^{(\boldsymbol{a})}(\Omega_1, \tau) \geq \boldsymbol{\nu}^{(\boldsymbol{a})}(\Omega_2, \tau).  \label{eqn:CA01_2}
  \end{align}
\label{prn:CA01_2}  
\end{proposition}

\begin{proof}
 Let $\varphi_1$ be the minimizer function of $\boldsymbol{\mu}^{(\boldsymbol{a})}(\Omega_1, \tau)$. 
 Then $\varphi_1$ is positive in $\Omega_1$ and $\varphi_1 \in W_{0}^{1,2}(\Omega_1)$(c.f. Theorem on page 116 of Rothaus~\cite{ROS}).  
 It follows from the definition that
 \begin{align*}
   \boldsymbol{\mu}^{(\boldsymbol{a})}(\Omega_1, \tau)=\mathcal{W}^{(\boldsymbol{a})}(\Omega_1, \varphi_1, \tau).
 \end{align*}
 On the other hand, every minimizer function of $\boldsymbol{\mu}^{(\boldsymbol{a})}(\Omega_2, \tau)$ is positive on $\Omega_2$.  Since $\varphi_1$ is supported on $\Omega_1$, a strict subdomain of $\Omega_2$, 
 we know $\varphi_1$ cannot be a minimizer function of $\boldsymbol{\mu}^{(\boldsymbol{a})}(\Omega_2, \tau)$.  Therefore, we have
 \begin{align*}
   \mathcal{W}^{(\boldsymbol{a})}(\Omega_1, \varphi_1, \tau)> \boldsymbol{\mu}^{(\boldsymbol{a})}(\Omega_2, \tau).
 \end{align*}
 Therefore, (\ref{eqn:CA01_1}) follows from the combination of the above two inequalities. 
 Clearly, (\ref{eqn:CA01_2}) is implied by (\ref{eqn:CA01_1}) for each $s \in (0, \tau]$ and the definition equation (\ref{eqn:MJ16_d}).  
\end{proof}

\begin{proposition}[\textbf{Non-positivity of $\boldsymbol{\nu}^{(\boldsymbol{a})}$}]
For each bounded domain $\Omega \subset M$, we have
\begin{align}
  \boldsymbol{\nu}^{(\boldsymbol{a})}(\Omega, \tau) \leq 0.     \label{eqn:CA01_3}
\end{align}
\label{prn:CA01_3}
\end{proposition}

\begin{proof}
Recall that $\displaystyle \boldsymbol{\nu}^{(\boldsymbol{a})}(\Omega, \tau)=\inf_{s \in (0, \tau]} \boldsymbol{\mu}^{(\boldsymbol{a})}(\Omega, \tau)$.
For (\ref{eqn:CA01_3}), it suffices to show that
 \begin{align*}
     \lim_{s \to 0^{+}} \boldsymbol{\mu}^{(\boldsymbol{a})}(\Omega, \tau) \leq 0. 
 \end{align*}
 Fix $x_0$ as an interior point of $\Omega$. We choose $\epsilon>0$ small enough such that $d(x_0, \partial \Omega)>2\epsilon$. 
 Let $\eta$ be a cutoff function such that $\eta \equiv 1$ in $B(x_0, \epsilon)$ and vanishes outside $B(x_0, 2\epsilon)$. 
 Motivated by the standard heat kernel expression on Euclidean space, we define
 \begin{align*}
    \varphi_s \coloneqq a_s \cdot \eta \cdot (4\pi s)^{-\frac{m}{4}} e^{-\frac{d^2}{8s}},
 \end{align*} 
 where $a_s$ is a normalization constant such that $\int_{\Omega} \varphi_s^2 dv=1$.  Clearly, we have $a_s \to 1$ as $s \to 0^{+}$.  Then it follows from the definition and direct calculation that
 \begin{align*}
   \lim_{s \to 0^{+}} \boldsymbol{\mu}^{(\boldsymbol{a})}(\Omega, \tau) \leq \lim_{s \to 0^{+}}  \mathcal{W}^{(\boldsymbol{a})}(\Omega, g, \varphi_s, s)=0, 
 \end{align*}
 which implies (\ref{eqn:CA01_3}). 
\end{proof}

In contrast to (\ref{eqn:CA01_3}) of Proposition~\ref{prn:CA01_3},  $\boldsymbol{\mu}^{(\boldsymbol{a})}$ could be positive.  For example, we can let $\boldsymbol{a} \equiv 0$ and $(\Omega, g)$ be the unit ball in the standard Euclidean space $(\R^{m}, g_{Euc})$. 
 For each $\tau>0$, it follows from Proposition~\ref{prn:CA01_2} that
 \begin{align*}
   \boldsymbol{\mu}^{(\boldsymbol{a})}(\Omega, \tau)>\boldsymbol{\mu}^{(\boldsymbol{a})}(\R^m, \tau)=0.
 \end{align*} 
 Since the above inequality holds for each positive $\tau$, the above inequality implies the non-negativity of $\boldsymbol{\nu}^{(\boldsymbol{a})}(\Omega, \tau)$, which together with (\ref{eqn:CA01_3}) yields that
 \begin{align*}
    \boldsymbol{\nu}^{(\boldsymbol{a})}(\Omega, \tau)=0
 \end{align*}
 for each bounded domain $\Omega$ and positive number $\tau$.  In particular, letting $\Omega_1$ and $\Omega_2$ be the balls in $\R^m$ centered at the origin and have radii $1$ and $2$ respectively, we have
 \begin{align*}
    \boldsymbol{\nu}^{(\boldsymbol{a})}(\Omega_1, \tau)=\boldsymbol{\nu}^{(\boldsymbol{a})}(\Omega_2, \tau)=0. 
 \end{align*}
 Therefore, the inequality (\ref{eqn:CA01_2}) in Proposition~\ref{prn:CA01_2} cannot be improved to a strict one in general. 
 This example also shows that there may exist many $\Omega$'s such that $\boldsymbol{\nu}^{(\boldsymbol{a})}(\Omega, \tau)=0$. 
 However, if $\Omega$ is allowed to be a complete manifold and $\boldsymbol{a}=R$, then there are more rigidities(c.f. Proposition~\ref{prn:CA05_1}).

 For unbounded domains, we have the following result.
 
 \begin{proposition}[\textbf{Continuity of $\boldsymbol{\mu}^{(\boldsymbol{a})}$}]
 Suppose $D$ is a possibly unbounded domain of $M$ with an exhaustion $D=\cup_{i=1}^{\infty} \Omega_i$ by bounded domains. 
 In other words, we have 
 \begin{align*}
    \Omega_1 \subset \Omega_2 \subset \cdots \subset \Omega_k \subset \cdots \subset D
 \end{align*}
 and each $\Omega_i$ is a bounded domain.  Then for each $\tau>0$ we have 
\begin{align}
  \boldsymbol{\mu}^{(\boldsymbol{a})}(D, \tau)= \lim_{i \to \infty} \boldsymbol{\mu}^{(\boldsymbol{a})}(\Omega_i, \tau).   \label{eqn:CA02_3}
\end{align}
\label{prn:CA02_1}
\end{proposition}

\begin{proof}
  Since $\Omega_i \subset D$ for each $i$, it follows from the definition that $\boldsymbol{\mu}^{(\boldsymbol{a})}(D, \tau) \leq \boldsymbol{\mu}^{(\boldsymbol{a})}(\Omega_i, \tau)$ for each $i$.
  Moreover, $\left\{ \boldsymbol{\mu}^{(\boldsymbol{a})}(\Omega_i, \tau) \right\}_{i=1}^{\infty}$ is a decreasing sequence.  Consequently, we have
  \begin{align}
      \boldsymbol{\mu}^{(\boldsymbol{a})}(D, \tau) \leq  \lim_{i \to \infty} \boldsymbol{\mu}^{(\boldsymbol{a})}(\Omega_i, \tau).  \label{eqn:CA02_4}
  \end{align}
  On the other hand, we can find a sequence of smooth functions $\varphi_i$ with compact support in $D$ such that
  \begin{align*}
     \boldsymbol{\mu}^{(\boldsymbol{a})}(D, \tau)=\lim_{i \to \infty} \mathcal{W}^{(\boldsymbol{a})}(D, \varphi_i, \tau).
  \end{align*}
  As $D=\cup_{i=1}^{\infty} \Omega_i$, we can assume the support of $\varphi_i$ is contained in $\Omega_{k_i}$ for some $k_i$.  Therefore, we have
  \begin{align*}
     \mathcal{W}^{(\boldsymbol{a})}(D, \varphi_i, \tau)=\mathcal{W}^{(\boldsymbol{a})}(\Omega_{k_i}, \varphi_i, \tau) \geq  \boldsymbol{\mu}^{(\boldsymbol{a})}(\Omega_{k_i}, \tau).
  \end{align*}
  It follows from the combination of the previous two steps that
  \begin{align}
     \boldsymbol{\mu}^{(\boldsymbol{a})}(D, \tau) \geq \lim_{i \to \infty} \boldsymbol{\mu}^{(\boldsymbol{a})}(\Omega_{k_i}, \tau) = \lim_{i \to \infty} \boldsymbol{\mu}^{(\boldsymbol{a})}(\Omega_i, \tau). 
  \label{eqn:CA02_5}   
  \end{align}
  Combining (\ref{eqn:CA02_4}) and (\ref{eqn:CA02_5}), we obtain (\ref{eqn:CA02_3}).   
\end{proof}

 Proposition~\ref{prn:CA02_1} is particularly interesting in the case that $D$ is the whole manifold $M$.  Then we have
 \begin{align}
      \boldsymbol{\mu}^{(\boldsymbol{a})}(M, \tau)= \lim_{i \to \infty} \boldsymbol{\mu}^{(\boldsymbol{a})}( B(x_0, r_i), \tau)  \label{eqn:CA02_6}
 \end{align}
 for a fixed point $x_0 \in M$ and radii $r_i \to \infty$.  However, unless $M$ satisfies some well-behaved geometry condition, the minimizer of $ \boldsymbol{\mu}^{(\boldsymbol{a})}(M, \tau)$ does not exist in general(c.f. Q. Zhang~\cite{QZh}). 
 
Now we discuss some fundamental properties of the minimizer functions. 

\begin{proposition}[\textbf{Boundary regularity of minimizers}]
Suppose $(M^{m}, g)$  is a smooth Riemannian manifold, $\Omega$ is a bounded open set in $M$ such that $\partial \Omega$ is smooth. 
Suppose $\boldsymbol{a}$ is a smooth function on $\bar{\Omega}$ and $\tau$ is a positive number.
Suppose $\varphi$ is a minimizer for the functional $\boldsymbol{\mu}^{(\boldsymbol{a})}(\Omega, g, \tau)$.   
Then we have $\varphi \in C^{2,\alpha}(\bar{\Omega})$ for each $\alpha \in (0,1)$. 
In particular, we have
\begin{align}
  \varphi(x) +|\nabla \varphi(x)| d(x)  \leq C d(x),  \quad \forall \; x \in \Omega 
\label{eqn:MJ11_1}  
\end{align}
where $d(x)=d(x, \partial \Omega)$, $C$ depends on $\Omega$ and $\boldsymbol{a}$. 
\label{prn:CA01_1}
\end{proposition}

\begin{proof}
 As a minimizer,  $\varphi$ satisfies the following Euler-Lagrange equation
\begin{align}
  \left(-4\tau \Delta + \tau \boldsymbol{a}- \boldsymbol{\mu}^{(\boldsymbol{a})}  -m-\frac{m}{2} \log (4\pi \tau) \right) \varphi = 2\varphi \log \varphi.  \label{eqn:MJ14_1}
\end{align}
Recall that $\varphi$ is a positive function in $\Omega$ and $\varphi \equiv 0$ on $\partial \Omega$. Also we have $\varphi \in W_0^{1,2}(\Omega)$(c.f. Rothaus~\cite{ROS}). 
Using Moser iteration, it is not hard to obtain that $\varphi$ is bounded(c.f.  inequality (25) of Tian-Wang~\cite{TiWa}), which implies that  $2\varphi \log \varphi$ is a bounded function on $\Omega$. 
Rewriting (\ref{eqn:MJ14_1}) as
\begin{align*}
  4\tau \Delta \varphi= \left(\tau \boldsymbol{a}- \boldsymbol{\mu}^{(\boldsymbol{a})}  -m-\frac{m}{2} \log (4\pi \tau) \right) \varphi -2\varphi \log \varphi.
\end{align*}
Let $L$ be $4\tau \Delta$ and $h$ be the right hand side of the above equation. Then $\varphi$ satisfies the equation $L \varphi=h$ and $\varphi|_{\partial \Omega}=0$.
Since $L$ is uniformly elliptic and $h$ is bounded,  it follows(c.f. Theorem 8.34 of Gilbarg-Trudinger~\cite{GT}) from the smoothness of  $\partial \Omega$ that 
$\varphi \in C^{1,\frac{1}{2}}(\bar{\Omega})$.  Note that the function $2x \log x$ is in $C^{\alpha}$ for each $\alpha \in (0,1)$. 
Therefore, $2\varphi \log \varphi \in C^{\alpha}(\bar{\Omega})$ for each $\alpha \in (0,1)$.   Applying standard elliptic theory(c.f. Theorem 6.19 of~\cite{GT}) to (\ref{eqn:MJ14_1}) again, we obtain $\varphi \in C^{2,\alpha}(\bar{\Omega})$. 
\end{proof}

Proposition~\ref{prn:CA01_1} is useful in the study of convergence of $\boldsymbol{\mu}$ along the $C^{\infty}$-Cheeger-Gromov convergence. 
Suppose $(M_i^m, x_i, g_i)$ is a sequence of pointed smooth Riemannian manifolds and $(M_{\infty}^m, x_{\infty}, g_{\infty})$ is also a pointed smooth Riemannian manifold.
We say that 
 \begin{align}
  (M_i^m, x_i, g_i)  \longright{C^{\infty}-Cheeger-Gromov} (M_{\infty}, x_{\infty}, g_{\infty})    \label{eqn:CB19_5}
 \end{align}
 if there exists an exhaustion $\cup_{k=1}^{\infty} K_k$ of $M_{\infty}$ by compact sets $K_k \ni x_{\infty}$ such that for each $k$,  there exist diffeomorphisms $\psi_{i,k}$ from $K_k$ to their images in $M_i$ such that
 \begin{align*}
     \psi_{i,k}^*(g_i)  \longright{C^{\infty}} g_{\infty}, \quad \textrm{on} \; K_k.
 \end{align*}
 Let $\Omega_{\infty}$ be a bounded set in $M_{\infty}$, we say that $\Omega_i$ is a sequence of sets in $M_i$ converging to $\Omega_{\infty}$ along the convergence (\ref{eqn:CB19_5}) if we can find a big $k$ such that
 $\Omega_{\infty} \subset K_k$ and $\Omega_i=\psi_{i,k}(\Omega_{\infty})$. 
 With these terminologies, we can discuss the following continuity property of the local-$\boldsymbol{\mu}$-functional. 

\begin{corollary}[\textbf{Continuity of local-$\boldsymbol{\mu}$-functional under $C^{\infty}$-Cheeger-Gromov convergence}]
 Suppose $(M_i^m, x_i, g_i)$ is a sequence of pointed Riemannian manifolds such that
 \begin{align}
  (M_i^m, x_i, g_i)  \longright{C^{\infty}-Cheeger-Gromov} (M_{\infty}^m, x_{\infty}, g_{\infty}).     \label{eqn:CB19_4}
 \end{align}
 Suppose $\Omega_{\infty}$ is a bounded open set in $M_{\infty}$ with smooth boundary, $\Omega_i \subset M_i$ are the sets converging to $\Omega_{\infty}$
 along the convergence (\ref{eqn:CB19_4}).  Then for each $\tau>0$ we have
 \begin{align}
    \lim_{i \to \infty} \boldsymbol{\mu}(\Omega_i, g_i, \tau)=\boldsymbol{\mu}(\Omega_{\infty}, g_{\infty}, \tau).  \label{eqn:CB19_6}
 \end{align}
\label{cly:CB19_2} 
\end{corollary}

\begin{proof}
 By taking further subsequence if necessary, we can assume that there exist diffeomorphisms $\psi_i$ from a compact set $K \subset M_{\infty}$ to its image on $M_i$ such that
 \begin{align}
          \psi_i^* g_i  \longright{C^{\infty}}  g_{\infty}, \quad \textrm{on} \;    K \supset \Omega_{\infty}.    \label{eqn:CB19_2}
 \end{align} 
 Then $\Omega_i$ is $\psi_i(\Omega_{\infty})$. Let $\varphi_i$ be a minimizer of the functional $\boldsymbol{\mu}(\Omega_i, g_i, \tau)$.   Clearly, $\varphi_i \circ \psi_i$ 
 is a minimizer of $\boldsymbol{\mu}\left(\Omega_{\infty}, \psi_i^*(g_i), \tau \right)$. 
 For simplicity of notation, we denote $\psi_i^* g_i$ and $\varphi_i \circ \psi_i$ by $\tilde{g}_i$ and $\tilde{\varphi}_i$ respectively.   
 Recall that $\tilde{\varphi}_i$ satisfies Euler-Lagrange equation (\ref{eqn:MJ14_1}) for $\boldsymbol{a}=R(\tilde{g}_i)$. 
 Then it follows from Proposition~\ref{prn:CA01_1} and the uniform equivalence condition (\ref{eqn:CB19_2}) that $\tilde{\varphi}_i$ have uniform $C^{2,\alpha}(\bar{\Omega}_{\infty})$ bounds, with respect to the metric $g_{\infty}$. 
 It follows that
  \begin{align}
          \tilde{\varphi}_i  \longright{C^{2,\alpha'}}  \tilde{\varphi}_{\infty}, \quad \textrm{on} \;    \bar{\Omega}_{\infty},   \label{eqn:CB19_3}
 \end{align} 
 for some $\alpha' \in (0, \alpha)$.  It is also clear that (\ref{eqn:CB19_2}) implies that $\boldsymbol{\mu}_i=\boldsymbol{\mu}\left(\Omega_{\infty}, \psi_i^*(g_i), \tau \right)$ are uniformly bounded. 
 By taking subsequence if necessary, we assume that $\boldsymbol{\mu}_i$ converges to $\boldsymbol{\mu}_{\infty}$. 
 Recall that (\ref{eqn:CB19_3}) guarantees the convergence of the Euler-Lagrange equation (\ref{eqn:MJ14_1}). So we have
 \begin{align*}
     \left(-4\tau \Delta + \tau R - \boldsymbol{\mu}_{\infty}  -m-\frac{m}{2} \log (4\pi \tau) \right) \tilde{\varphi}_{\infty} = 2\tilde{\varphi}_{\infty} \log \tilde{\varphi}_{\infty}. 
 \end{align*}
 It is also clear from (\ref{eqn:CB19_3}) that $\tilde{\varphi}_{\infty}$ satisfies the normalization condition $\int_{\Omega_{\infty}} \tilde{\varphi}_{\infty}^2 dv=1$ and locates in $W_0^{1,2}(\Omega_{\infty}, g_{\infty})$. 
 Therefore, it follows from the definition that
 \begin{align}
   \boldsymbol{\mu}(\Omega_{\infty}, g_{\infty}, \tau) \leq \mathcal{W}^{(R)}(\Omega_{\infty}, g_{\infty}, \tilde{\varphi}_{\infty}, \tau) = \boldsymbol{\mu}_{\infty}.   \label{eqn:CB19_7}
 \end{align}
 On the other hand, let $\varphi$ be a minimizer of $\boldsymbol{\mu}(\Omega_{\infty}, g_{\infty}, \tau)$.  
 In view of (\ref{eqn:CB19_2}) and (\ref{eqn:CB19_3}), we know $\lambda_i \varphi$ satisfies the normalization condition $\int_{\Omega_{\infty}} (\lambda_i \varphi)^2 dv_{\tilde{g}_i}=1$ for a sequence of constants $\lambda_i \to 1$. 
 It follows that
 \begin{align*}
   \boldsymbol{\mu}_i=\boldsymbol{\mu}\left(\Omega_{\infty}, \tilde{g}_i, \tau \right) \leq \mathcal{W}^{(R)}(\Omega_{\infty}, \tilde{g}_i, \lambda_i \varphi, \tau), 
 \end{align*}
 whose limit reads as
 \begin{align}
      \boldsymbol{\mu}_{\infty} \leq \mathcal{W}^{(R)}(\Omega_{\infty}, g_{\infty}, \varphi, \tau)=\boldsymbol{\mu}(\Omega_{\infty}, g_{\infty}, \tau).  \label{eqn:CB19_8}
 \end{align}
 Consequently, the combination of (\ref{eqn:CB19_7}) and (\ref{eqn:CB19_8}) implies that $\boldsymbol{\mu}(\Omega_{\infty}, g_{\infty}, \tau)= \boldsymbol{\mu}_{\infty}$, which is nothing but (\ref{eqn:CB19_6}).  
\end{proof}

In applications, the smooth boundary condition cannot always be satisfied. Therefore, we often need to slightly perturb the domain in study to have better boundary regularity. 
Such perturbation is guaranteed by the following lemma. 

\begin{lemma}[\textbf{Approximate smooth boundary condition}]
Suppose $\Omega$ is a connected, bounded open set of $M$. 
For each small $\epsilon$,  there exists a set $\Omega_{\epsilon}$ satisfying 
\begin{itemize}
\item $\Omega \subset \Omega_{\epsilon} \subset$ the $\epsilon$-neighborhood of $\Omega$;
\item $\partial \Omega_{\epsilon}$ is smooth. 
\end{itemize}
\label{lma:CA01_1}
\end{lemma}

\begin{proof}
 Fix $\epsilon>0$ and denote the $\epsilon$-neighborhood of $\Omega$ by $\Omega'$. Then $\{ M \backslash \bar{\Omega}, \Omega'\}$ is a covering of $M$. 
 By partition of unity, there exists a smooth function $u$ supported on $\Omega'$ and $u \equiv 1$ on $\bar{\Omega}$. 
 By Sard theorem(c.f. the book of J. Lee~\cite{JLee}), the function $u$ has a regular value $\lambda \in (0, 1)$ such that $u^{-1}(\lambda)$ is a smooth manifold of dimension $m-1$. 
 Clearly, we have $u^{-1}([\lambda,1]) \supset \bar{\Omega}$.  There exists exactly one connected component of $u^{-1}([\lambda,1])$ that contains $\bar{\Omega}$.
 We denote this component by $\Omega_{\epsilon}$. It is clear that 
 \begin{align*}
   \bar{\Omega} \subset  \Omega_{\epsilon} \subset u^{-1}([\lambda, 1]) \subset \Omega'. 
 \end{align*}
 Moreover, $\partial \Omega_{\epsilon}=u^{-1}(\lambda)$ is a smooth manifold.   Therefore, $\Omega_{\epsilon}$ satisfies all the desired properties and the proof of Lemma~\ref{lma:CA01_1} is complete. 
\end{proof}

\begin{remark}
The estimate (\ref{eqn:MJ11_1}) in Proposition~\ref{prn:CA01_1} is needed to estimate the lower bound of $\boldsymbol{\mu}(\Omega, \tau)$ along the Ricci flow whenever $\Omega$ has a smooth boundary.
However, smooth boundary condition is not satisfied by many domains, e.g., geodesic balls. 
Lemma~\ref{lma:CA01_1} is used to drop the smooth boundary condition of $\Omega$.  Actually, by Proposition~\ref{prn:CA01_2}, we have
 \begin{align*}
     \boldsymbol{\mu}(\Omega, \tau) \geq \boldsymbol{\mu}(\Omega_{\epsilon}, \tau).  
 \end{align*}
Therefore, the lower bound of $\boldsymbol{\mu}(\Omega, \tau)$ can be derived from the lower bound of $\boldsymbol{\mu}(\Omega_{\epsilon}, \tau)$ for $\epsilon$ sufficiently small. 
Therefore, for simplicity, we may always assume that $\partial \Omega$ is smooth whenever  we want to develop the lower bound of $\boldsymbol{\mu}(\Omega, \tau)$. 
An alternative way to achieve this purpose is to approximate $\Omega$ by smaller sets with smooth boundary, say $\Omega_{-\epsilon}$'s, and then apply Proposition~\ref{prn:CA02_1}. 
\label{rmk:MJ12_1} 
\end{remark}

\begin{remark}
Similar to the discussion in this section, one can also localize the steady soliton functional $\boldsymbol{\lambda}$ of Perelman~\cite{Pe1}, 
and the expanding soliton functional $\boldsymbol{\mu}^{+}$ and $\boldsymbol{\nu}^{+}$ of Feldman-Ilmanen-Ni~\cite{FIN}. 
\label{rmk:CA02_1}
\end{remark}

\section{The local functionals and the volume ratios}
\label{sec:vratios}

For each $\Omega \subset M$ bounded domain with smooth boundary, it is not hard to see that $\bar{\boldsymbol{\nu}}(\Omega)$ is nothing but the optimal uniform logarithmic Sobolev constant(c.f.~\cite{Gross},~\cite{Da}).
It is known in the literature  that the bound of $\bar{\boldsymbol{\nu}}(\Omega)$ is equivalent to the bound of many other quantities, like Sobolev constant bound, Faber-Krahn constant, Nash constant, 
heat kernel on-diagonal upper bound, heat kernel off-diagonal Gaussian upper bound, etc(e.g. see Section 6.1 of~\cite{AGrig} for a survey). 
It is also known that $\bar{\boldsymbol{\nu}}(\Omega, \tau)$ is enough to bound many quantities(e.g., see p. 320 of~\cite{Dav1}). 
In this section, we shall only investigate some elementary estimates of $\bar{\boldsymbol{\nu}}(\Omega, \tau)$ and its relationship with $\boldsymbol{\nu}(\Omega, \tau)$ and the volume ratios, whenever some scalar or Ricci curvature 
conditions are satisfied.  For simplicity of notations, the following conditions are assumed by default in all the discussion in this section.
\begin{align}
  -\underline{\Lambda} \leq R(x) \leq \bar{\Lambda},  \quad   Rc(x) \geq -(m-1)K^2, \quad \forall \; x \in \Omega,
\label{eqn:CF20_0} 
\end{align}
where $\underline{\Lambda}, \bar{\Lambda}$ and $K$ are nonnegative constants.

We first note that there are elementary relationships among the local functionals.

\begin{lemma}
 For each $\tau>0$ we have
 \begin{align}
  &\boldsymbol{\mu}(\Omega,\tau) - \bar{\Lambda} \tau \leq \bar{\boldsymbol{\mu}}(\Omega,\tau) \leq \boldsymbol{\mu}(\Omega,\tau) +\underline{\Lambda} \tau,  \label{eqn:MJ16_6}\\
  &\boldsymbol{\nu}(\Omega,\tau) -\bar{\Lambda}  \tau \leq \bar{\boldsymbol{\nu}}(\Omega,\tau) \leq \boldsymbol{\nu}(\Omega,\tau)+ \underline{\Lambda} \tau.  \label{eqn:MJ16_7}
 \end{align}
\label{lma:MJ16_1} 
\end{lemma}

\begin{proof}
  For each function $\varphi \in \mathscr{S}(\Omega)$ (c.f. equations (\ref{eqn:MJ16_a}) -(\ref{eqn:MJ16_d})) and each $\tau>0$, we have
  \begin{align}
   -\bar{\Lambda} \tau \leq  \mathcal{W}^{(0)}(\Omega, \varphi, \tau) -\mathcal{W}^{(R)}(\Omega, \varphi, \tau)
     = -\int_{\Omega} \tau R \varphi^2 dv \leq \underline{\Lambda} \tau \int_{\Omega} \varphi^2 dv=\underline{\Lambda}\tau. 
  \label{eqn:MJ16_8}   
  \end{align} 
  If we choose $\varphi$ as the minimizer of $\boldsymbol{\mu}^{(0)}(\Omega, \varphi, \tau)$, then the first part of (\ref{eqn:MJ16_8}) implies that
  \begin{align}
    \boldsymbol{\mu}^{(0)}(\Omega, \tau)=\mathcal{W}^{(0)}(\Omega, \varphi, \tau) \geq \mathcal{W}^{(R)}(\Omega, \varphi, \tau)-\bar{\Lambda} \tau \geq \boldsymbol{\mu}^{(R)}(\Omega, \tau)-\bar{\Lambda} \tau.
  \label{eqn:MJ16_9}  
  \end{align}
  However, if we choose $\varphi$ as the minimizer of $\boldsymbol{\mu}^{(R)}(\Omega, \varphi, \tau)$, then the last part of (\ref{eqn:MJ16_8}) implies that
  \begin{align}
    \boldsymbol{\mu}^{(R)}(\Omega, \tau)=\mathcal{W}^{(R)}(\Omega, \varphi, \tau) \geq \mathcal{W}^{(0)}(\Omega, \varphi, \tau) -\underline{\Lambda} \tau \geq \boldsymbol{\mu}^{(0)}(\Omega, \tau)-\underline{\Lambda} \tau.
  \label{eqn:MJ16_10}  
  \end{align}
  Combining (\ref{eqn:MJ16_9}) and (\ref{eqn:MJ16_10}) gives us
  \begin{align}
    \boldsymbol{\mu}^{(R)}(\Omega, \tau)-\bar{\Lambda} \tau \leq \boldsymbol{\mu}^{(0)}(\Omega,\tau) \leq \boldsymbol{\mu}^{(R)}(\Omega, \tau) +\underline{\Lambda} \tau,   \label{eqn:MJ16_11}
  \end{align}
  which is nothing but (\ref{eqn:MJ16_6}) by the choice of our notion in (\ref{eqn:MJ16_1}) and (\ref{eqn:MJ16_2}). 
  The proof of (\ref{eqn:MJ16_6}) is complete. 
  
  We proceed to prove (\ref{eqn:MJ16_7}). 
  The second inequality in (\ref{eqn:MJ16_11}) implies that for each $s \in (0,\tau)$, we have
  \begin{align*}
    \boldsymbol{\mu}^{(0)}(\Omega, s) \leq \boldsymbol{\mu}^{(R)}(\Omega, s) +\underline{\Lambda} s \leq \boldsymbol{\mu}^{(R)}(\Omega, s) +\underline{\Lambda} \tau. 
  \end{align*}
  Taking infimum of the above inequality yields that
  \begin{align*}
     \boldsymbol{\nu}^{(0)}(\Omega, \tau) \leq \boldsymbol{\nu}^{(R)}(\Omega, \tau) +\underline{\Lambda} \tau.  
  \end{align*}
  Similarly, we can analyze the first inequality in (\ref{eqn:MJ16_11}) and obtain that
  \begin{align*}
    \boldsymbol{\nu}^{(R)}(\Omega, \tau) \leq \boldsymbol{\nu}^{(0)}(\Omega, \tau) +\bar{\Lambda}\tau.
  \end{align*}
  Consequently, (\ref{eqn:MJ16_7}) follows from the combination of the previous two inequalities.    
\end{proof}

After the discussion of internal relationships among local functionals, we now study the relationships among local functionals and the volume ratios. 

\begin{proposition}
  Suppose $0<\tau_1<\tau_2$, then
\begin{align}
   &\bar{\boldsymbol{\nu}}(\Omega, \tau_2) \leq  \bar{\boldsymbol{\nu}}(\Omega, \tau_1) \leq \bar{\boldsymbol{\nu}}(\Omega, \tau_2) + \frac{m}{2} \log \frac{\tau_2}{\tau_1}.  \label{eqn:CF10_1}\\     
   &\boldsymbol{\nu}(\Omega, \tau_2) \leq  \boldsymbol{\nu}(\Omega, \tau_1) \leq \boldsymbol{\nu}(\Omega, \tau_2) + \frac{m}{2} \log \frac{\tau_2}{\tau_1} +\bar{\Lambda} \tau_1+\underline{\Lambda} \tau_2.  \label{eqn:CF10_3}   
\end{align} 
\label{prn:CF10_2}  
\end{proposition}

\begin{proof}
  The first inequality of (\ref{eqn:CF10_1}) follows trivially from the definition equation (\ref{eqn:MJ16_d}). We only need to prove the second part of (\ref{eqn:CF10_1}). 
We may assume $\bar{\boldsymbol{\nu}}(\Omega, \tau_2)<0$, for otherwise we have nothing to prove by the non-positivity of $\bar{\boldsymbol{\nu}}$(c.f. Proposition~\ref{prn:CA01_3}). 
  Therefore, there exists some $\tau_2' \in (0, \tau_2]$ such that $\bar{\boldsymbol{\nu}}(\Omega, \tau_2)=\bar{\boldsymbol{\mu}}(\Omega, \tau_2')$. 
  If $\tau_2' \leq \tau_1$, then we have nothing to prove. So we assume $\tau_2' \in (\tau_1, \tau_2]$. 
  Let $\varphi$ be the minimizer function of $\bar{\boldsymbol{\mu}}(\Omega, \tau_2')$.  Then the Euler-Lagrange equation (\ref{eqn:GH01_1}) reads as
  \begin{align*}
   \left( \bar{\boldsymbol{\mu}}(\Omega, \tau_2') +m+\frac{m}{2}\log (4\pi \tau_2') \right) \varphi=-4\tau_2' \Delta \varphi -2\varphi \log \varphi. 
  \end{align*}
  Multiplying both sides of the above equation by $\varphi$ and integrating on $\Omega$, we obtain
  \begin{align*}
     \bar{\boldsymbol{\mu}}(\Omega, \tau_2') +m+\frac{m}{2}\log (4\pi \tau_2') &=4\tau_2' \int_{\Omega} |\nabla \varphi|^2 dv  - \int_{\Omega} \varphi^2 \log \varphi^2 dv\\
         &\geq  4\tau_1 \int_{\Omega} |\nabla \varphi|^2 dv  - \int_{\Omega} \varphi^2 \log \varphi^2 dv\\
         &\geq \bar{\boldsymbol{\mu}}(\Omega, \tau_1) +m+\frac{m}{2}\log (4\pi \tau_1), 
  \end{align*}
  where we used the definition equation (\ref{eqn:MJ16_c}) in the last step. 
  Consequently, it follows from (\ref{eqn:MJ16_d}) that
  \begin{align*}
    \bar{\boldsymbol{\mu}}(\Omega, \tau_2') \geq \bar{\boldsymbol{\mu}}(\Omega, \tau_1)  + \frac{m}{2} \log \frac{\tau_1}{\tau_2'} \geq \bar{\boldsymbol{\mu}}(\Omega, \tau_1)  + \frac{m}{2} \log \frac{\tau_1}{\tau_2}
      \geq  \bar{\boldsymbol{\nu}}(\Omega, \tau_1)  + \frac{m}{2} \log \frac{\tau_1}{\tau_2},
  \end{align*}
  whence we obtain the second part of (\ref{eqn:CF10_1}) since $\bar{\boldsymbol{\nu}}(\Omega, \tau_2)=\bar{\boldsymbol{\mu}}(\Omega, \tau_2')$ by our choice of $\tau_2'$. \\
  
 Now we focus on the proof of  (\ref{eqn:CF10_3}).   
 Again, we only need to prove the second part of (\ref{eqn:CF10_3}).  It follows from the first part of (\ref{eqn:MJ16_7}) of Lemma~\ref{lma:MJ16_1} that
 \begin{align}
  \boldsymbol{\nu}(\Omega,\tau_1)  \leq \bar{\boldsymbol{\nu}}(\Omega,\tau_1) + \bar{\Lambda} \tau_1.
 \end{align}
 Plugging the second part of (\ref{eqn:CF10_1}) into the above inequality, we arrive at
 \begin{align*}
    \boldsymbol{\nu}(\Omega,\tau_1)  \leq \bar{\boldsymbol{\nu}}(\Omega,\tau_2) +\frac{m}{2} \log \frac{\tau_2}{\tau_1} + \bar{\Lambda} \tau_1.
 \end{align*}
 Then we apply the second part of (\ref{eqn:MJ16_7}) of Lemma~\ref{lma:MJ16_1} to obtain 
 \begin{align*}
   \bar{\boldsymbol{\nu}}(\Omega,\tau_2) \leq \boldsymbol{\nu}(\Omega,\tau_2) +\underline{\Lambda} \tau_2.
 \end{align*}
 Combining the previous two steps implies that
 \begin{align*}
     \boldsymbol{\nu}(\Omega,\tau_1)  \leq \boldsymbol{\nu}(\Omega,\tau_2)+\frac{m}{2} \log \frac{\tau_2}{\tau_1} + \bar{\Lambda} \tau_1+\underline{\Lambda} \tau_2,
 \end{align*}
 whence we derive (\ref{eqn:CF10_3}). 
\end{proof}

\begin{theorem}[\textbf{Lower bound of volume ratio in terms of $\boldsymbol{\nu}$ and scalar curvature}]
Suppose $B=B(x_0,r_0) \subset \Omega$ is a geodesic ball, then we have
\begin{align}
 &\frac{|B|}{\omega_m r_0^m} \geq  e^{\bar{\boldsymbol{\nu}}-2^{m+7}} \geq  e^{\boldsymbol{\nu}-2^{m+7}-\bar{\Lambda}r_0^2}
 \label{eqn:GC28_2}   
\end{align}
where $\bar{\boldsymbol{\nu}}=\bar{\boldsymbol{\nu}}(B,r_0^2), \; \boldsymbol{\nu}=\boldsymbol{\nu}(B,r_0^2)$. 
\label{thm:CF21_3}
\end{theorem}

\begin{proof}
The second inequality in (\ref{eqn:GC28_2}) follows directly from (\ref{eqn:MJ16_7}) in Lemma~\ref{lma:MJ16_1}. 
It suffices to prove the first inequality in (\ref{eqn:GC28_2}). 
Let $q$ be the volume ratio function:
\begin{align*}
  q(\rho) \coloneqq 
  \begin{cases}
    \frac{|B(x_0,\rho)|}{\omega_m \rho^{m}},  &     \textrm{if} \; 0<\rho \leq r_0;\\
    1, & \textrm{if} \; \rho=0.
  \end{cases}
\end{align*}
Clearly, $q$ is continuous on $[0,r_0]$. Suppose the minimum value of $q$ is achieved at $\rho_0$.
If $\rho_0=0$, we have
\begin{align}
  \frac{|B|}{\omega_m r_0^m} \geq 1.    \label{eqn:CF20_1}
\end{align} 
Now we assume $\rho_0>0$. It follows from the definition of
$\rho_0$ that
\begin{align}
  \frac{|B(x_0, \frac{\rho_0}{2})|}{\omega_m (\frac{\rho_0}{2})^{m}} \geq \frac{|B(x_0, \rho_0)|}{\omega_m \rho_0^{m}}=q(\rho_0), \quad \Rightarrow
  \quad |B(x_0, \rho_0)| \leq 2^m \left|B\left(x_0,\frac{\rho_0}{2} \right)\right|.  
  \label{eqn:GB12_1}  
\end{align}
Take a cutoff function $\eta$ which equals $1$ on $(-\infty, \frac{1}{2}]$, decreases $0$ on $[\frac{1}{2}, 1]$ and equals $0$ on $[1,\infty)$. 
 Furthermore, we have $|\eta'| \leq 4$.  Set $d=d(\cdot, x_0)$ and define
 \begin{align*}
   L \coloneqq \int_{M} \eta^2 \left( \frac{d}{\rho_0} \right) dv,  \quad \varphi \coloneqq L^{-\frac{1}{2}}\eta \left(\frac{d}{\rho_0} \right). 
 \end{align*}
 For simplicity of notations, denote $B(x_0, \rho_0)$ by $B'$ and $B(x_0,\frac{\rho_0}{2})$ by $\frac{1}{2}B'$.
 It follows from (\ref{eqn:GB12_1}) and the definition of $\varphi$ and $L$ that
 \begin{align}
   |B'| \geq L \geq \left|\frac{1}{2}B' \right| \geq 2^{-m}|B'|.   \label{eqn:GB12_3}
 \end{align}
 Then we have
 \begin{align*}
   &\quad \bar{\boldsymbol{\mu}}(B',\rho_0^2) +m +\frac{m}{2} \log (4\pi \rho_0^2)\\
   &\leq \mathcal{W}^{(0)}(B',\varphi, \rho_0^2) +m +\frac{m}{2} \log (4\pi \rho_0^2)
    =\int_{B'} -\varphi^2 \log \varphi^2 dv+\int_{B'\backslash \frac{1}{2}B'} 4\rho_0^2 |\nabla \varphi|^2 dv\\
   &=\log L + \frac{\int_{B'} \left\{ -\eta^2 \log \eta^2 \right\} dv}{L}+   \frac{\int_{B'\backslash \frac{1}{2}B'} 4|\eta'|^2 dv}{L}. 
 \end{align*}
 Note that $-\eta^2 \log \eta^2 \leq e^{-1}$ and $|\eta'| \leq 4$ in the ball $B'$. 
 Therefore, the last two terms in the above inequality can be bounded.  
 \begin{align*}
   &\int_{B'} \left\{ -\eta^2 \log \eta^2 \right\} dv \leq e^{-1} |B'|, \\
   &\int_{B' \backslash \frac{1}{2}B'} |\eta'|^2 dv \leq 16\left(|B'|-\left|\frac{1}{2}B'\right|  \right) \leq 16(1-2^{-m}) |B'|, 
 \end{align*}
 where we used the conditions (\ref{eqn:GB12_3}) in the last step. 
 Combining all the above inequalities, we obtain
\begin{align*}
    &\quad \bar{\boldsymbol{\mu}}(B',\rho_0^2) +m +\frac{m}{2} \log (4\pi \rho_0^2)\\
    &\leq \log L + \left(e^{-1} + 64(1-2^{-m})  \right) \frac{|B'|}{L}
    \leq \log L + \left(e^{-1} + 64(1-2^{-m})  \right) 2^{m}\\
    &< \log |B'| + 65 \cdot 2^{m}.
\end{align*}
It follows that
\begin{align*}
  \log \frac{|B'|}{\rho_0^{m}} > \bar{\boldsymbol{\mu}}(B',\rho_0^2) +m+\frac{m}{2} \log 4\pi -65 \cdot 2^{m}. 
\end{align*}
Recall that $r_0 \geq \rho_0 >0$, $\frac{|B|}{r_0^m} \geq \frac{|B'|}{\rho_0^{m}}$ and $\bar{\boldsymbol{\mu}}(B',\rho_0^2) \geq \bar{\boldsymbol{\mu}}(B,\rho_0^2) \geq \bar{\boldsymbol{\nu}}(B,r_0^2)=\bar{\boldsymbol{\nu}}$. 
Recall also the explicit formula of unit ball volume in Euclidean space $\R^m$:
\begin{align}
    \omega_m =\frac{\pi^{\frac{m}{2}}}{\Gamma(\frac{m}{2}+1)}     \label{eqn:CF19_9}
\end{align}
where $\Gamma$ is the traditional $\Gamma$-function.  We then obtain
\begin{align}
  \frac{|B|}{\omega_m r_0^{m}} > e^{\bar{\boldsymbol{\nu}}-\log \omega_m+m+\frac{m}{2} \log 4\pi -65 \cdot 2^{m}}
  =e^{\bar{\boldsymbol{\nu}}+\left\{\log \Gamma(\frac{m}{2}+1) +m+m\log 2 \right\}-65 \cdot 2^{m}}.
  \label{eqn:CF20_2}
\end{align}
Since $m \geq 3$, it is easy to check that $\log \Gamma(\frac{m}{2}+1) +m+m\log 2 <2^m$. 
Notice that $\bar{\boldsymbol{\nu}} \leq 0$ by Proposition~\ref{prn:CA01_3}. 
Therefore,  we have
\begin{align*}
  e^{\bar{\boldsymbol{\nu}}+\left\{\log \Gamma(\frac{m}{2}+1) +m+m\log 2 -2^m\right\} -2^{m+6}} 
  <e^{-2^{m+6}} <<1.
\end{align*}
Combining the cases that $\rho_0=0$ and $\rho_ 0 \in (0, r_0]$, i.e., (\ref{eqn:CF20_1}) and (\ref{eqn:CF20_2}), we obtain
\begin{align}
   \frac{|B|}{\omega_m r_0^{m}} >  e^{\bar{\boldsymbol{\nu}}+\left\{\log \Gamma(\frac{m}{2}+1) +m+m\log 2 \right\} -65 \cdot 2^m}
   >e^{\bar{\boldsymbol{\nu}}-2^{m+7}}
\label{eqn:CF19_7}     
\end{align}
which is exactly the first part of (\ref{eqn:GC28_2}). 
The proof of Theorem~\ref{thm:CF21_3} is complete. 
\end{proof}

\begin{remark}
 The volume ratio lower bound determined by Theorem~\ref{thm:CF21_3} are far away from being sharp. 
 For example, on the Euclidean space,  applying Theorem~\ref{thm:CF21_3} on unit ball $B$, 
 and noting that $\bar{\boldsymbol{\nu}}(B,g_{E}, 1)=\boldsymbol{\nu}(B,g_{E}, 1)=0$,
 we obtain an explicit lower bound of the volume of each unit ball
 \begin{align*}
     e^{-2^{m+7}} \cdot \omega_m
 \end{align*}
 which is a very small number compared to the actual value $\omega_m$. 
\label{rmk:MJ16_3} 
\end{remark}

Recall that $\bar{\boldsymbol{\nu}}$ is nothing but the optimal uniform Logarithmic Sobolev constant.  It is well-known(c.f. Section 6.1 of~\cite{AGrig}) that the Logarithmic Sobolev inequality is dominated by
the $L^2$-Sobolev inequality constant $C_S$ and consequently the $L^1$-Sobolev inequality constant $C_I$. 
It is also well-known(c.f. section 3.1 of~\cite{SY}) that $C_I$ is equivalent to $\mathbf{I}^{-1}$ where $\mathbf{I}$ is the isoperimetric constant
\begin{align}
 \mathbf{I}(\Omega) \coloneqq \inf_{D \Subset \Omega}  \frac{|\partial D|}{|D|^{\frac{m-1}{m}}}. 
\label{eqn:CF08_8} 
\end{align}
Combing the previous steps, it is clear that $\bar{\boldsymbol{\nu}}$ can be bounded by $\mathbf{I}(\Omega)$. However, there are too many intermediate steps in the above deduction where errors occur. 
So the result obtained in this way cannot be sharp. 
In order to reduce the errors, we shall develop a direct estimate of $\bar{\boldsymbol{\nu}}$ by $\mathbf{I}(\Omega)$, in terms of Logarithmic eigenfunctions. 
For simplicity of notation, we define
\begin{align}
  \mathbf{I}_m \coloneqq \mathbf{I}(\R^{m}, g_{E}). 
\label{eqn:CF19_1}  
\end{align}
Clearly, the best isoperimetric constant in the Euclidean space $(\R^{m}, g_{E})$ is achieved by standard balls.  Consequently, $\mathbf{I}_m$ can be calculated explicitly as
\begin{align}
    \mathbf{I}_m=  \frac{m \omega_m}{\omega_m^{\frac{m-1}{m}}}=m \omega_m^{\frac{1}{m}}. 
\label{eqn:CF19_2}    
\end{align}

\begin{lemma}[\textbf{Estimate of functionals by isoperimetric constant and scalar lower bound}]
Suppose $\Omega$ is a bounded domain in $(M^{m}, g)$, $\tilde{\Omega}$ is a ball in $(\R^{m}, g_{E})$ such that $|\tilde{\Omega}|=|\Omega|$.
Define
\begin{align}
  \lambda \coloneqq \frac{\mathbf{I}(\Omega)}{\mathbf{I}_m}.   
  \label{eqn:CF20_4}
\end{align}
Then we have
\begin{align}
   \bar{\boldsymbol{\mu}}(\Omega, g, \tau) \geq  \bar{\boldsymbol{\mu}} \left(\tilde{\Omega}, g_{E}, \tau \lambda^2 \right) + m\log \lambda.   \label{eqn:MJ26_5}
\end{align}
Consequently, we have
\begin{align}
     &\bar{\boldsymbol{\nu}}(\Omega, g, \tau) \geq   m\log \lambda,   \label{eqn:MJ26_6}\\
     &\boldsymbol{\mu}(\Omega, g, \tau) \geq \bar{\boldsymbol{\mu}} \left(\tilde{\Omega}, g_{E}, \tau \lambda^2 \right) + m\log \lambda - \underline{\Lambda} \tau, \label{eqn:MJ26_8}\\
     &\boldsymbol{\nu}(\Omega, g, \tau) \geq m\log \lambda -\underline{\Lambda} \tau. \label{eqn:MJ26_9}
\end{align}
\label{lma:MJ25_1}
\end{lemma}

\begin{proof}
   We shall partly follow the argument in the proof of Proposition 4.1 of L. Ni~\cite{Ni}, which was suggested by Perelman at the end of the proof of Theorem 10.1 of~\cite{Pe1}. 
   The new ingredient here is to only estimate the smooth eigenfunctions. 
      
   Let $\varphi$ be a minimizer function of $\bar{\boldsymbol{\mu}}(\Omega, g, \tau)$.  Clearly, $\varphi>0$ in $\Omega$ and $\varphi=0$ on $\partial \Omega$. 
   Define
   \begin{align*}
     \Omega_{t} \coloneqq \{x \in \Omega| \varphi(x) \geq t\},     \quad F(t) \coloneqq |\Omega_t|. 
   \end{align*}
   Let $h=h(|y|)$ be a radial symmetric function on $\tilde{\Omega}$ such that 
   \begin{align*}
      |\{h(|y|) \geq t\}|=F(t), \quad \left. h \right|_{\partial B}=0. 
   \end{align*}
   We can similarly define $\tilde{\Omega}_t \coloneqq \{x \in B| h  \geq t\}$.  Then the above equations imply
   \begin{align}
      |\Omega_t|=F(t)=|\tilde{\Omega}_t|.         \label{eqn:MJ26_1}
   \end{align}
   It follows from the definition of the isoperimetric constant that
   \begin{align}
      |\partial \Omega_t| \geq \mathbf{I} |\Omega_t|^{\frac{m-1}{m}}=\mathbf{I} |\tilde{\Omega}_t|^{\frac{m-1}{m}}=\frac{\mathbf{I}}{\mathbf{I}_m} |\partial \tilde{\Omega}_t| = \lambda |\partial \tilde{\Omega}_t|,
    \label{eqn:MJ26_2}  
   \end{align}
   where we used the facts that $\mathbf{I}_m$ is achieved by balls and $\tilde{\Omega}_t$ is a ball for each $t$ by its construction. 
   
   Recall that the integration by parts implies that
   \begin{align*}
     \int_{\tilde{\Omega}} \chi (h) d\tilde{v}=\int_{0}^{\infty} \chi'(s) F(s) ds = \int_{\Omega} \chi(\varphi) dv
   \end{align*}
   for any Lipschitz function $\chi(t)$ with $\chi(0)=0$.  Let $\chi(t)=t^2\log t^2$, we obtain
   \begin{align}
    \int_{\tilde{\Omega}}  h^2 \log h^2 d\tilde{v}= \int_{\Omega}  \varphi^2 \log \varphi^2 dv.
   \label{eqn:MJ26_3}  
   \end{align}
   Similarly, we have
   \begin{align}
    \int_{\tilde{\Omega}}  h^2  d\tilde{v}=\int_{\Omega}  \varphi^2 dv=1.    \label{eqn:MJ26_7}
   \end{align}
   On the other hand, by the co-area formula, we have
   \begin{align*}
      \int_{t}^{\infty} \int_{\varphi=s} \frac{1}{|\nabla \varphi|} d\sigma ds=|\Omega_t|=F(t)=|\tilde{\Omega}_t|=\int_{t}^{\infty} \int_{h=s} \frac{1}{|\nabla h|} d\tilde{\sigma} ds,  \quad \forall \; t>0, 
   \end{align*}
   which implies that
   \begin{align}
       \int_{\partial \Omega_t} \frac{1}{|\nabla \varphi|} d\sigma = \int_{\partial \tilde{\Omega}_t} \frac{1}{|\nabla h|} d \tilde{\sigma}, \quad \forall \; t>0. 
   \label{eqn:CF24_1}    
   \end{align}
   Note that $|\nabla h|$ is a constant on $\partial \tilde{\Omega}_t$. Therefore, in light of (\ref{eqn:MJ26_2}) and (\ref{eqn:CF24_1}), we can apply the H\"older inequality to obtain that
   \begin{align*}
    & \int_{\partial \Omega_t} |\nabla \varphi| d\sigma  \geq \frac{|\partial \Omega_t|^2}{\int_{\partial \Omega_t} \frac{1}{|\nabla \varphi|} d\sigma} 
    \geq \lambda^2 \frac{|\partial \tilde{\Omega}_t|^2}{\int_{\partial \Omega_t} \frac{1}{|\nabla \varphi|} d\sigma} 
    =\lambda^2 \frac{  \int_{\partial \tilde{\Omega}_t} |\nabla h| d\tilde{\sigma}  \int_{\partial \tilde{\Omega}_t} \frac{1}{|\nabla h|} d \tilde{\sigma}}{\int_{\partial \Omega_t} \frac{1}{|\nabla \varphi|} d\sigma} 
    =\lambda^2 \int_{\partial \tilde{\Omega}_t} |\nabla h| d \tilde{\sigma}. 
   \end{align*} 
   Using the co-area formula again, we have
    \begin{align}
     \int_{\Omega} |\nabla \varphi|^2 dv=\int_0^{\infty} \int_{\partial \Omega_t} |\nabla \varphi| d\sigma dt \geq \lambda^2 \int_0^{\infty} \int_{\partial \tilde{\Omega}_t} |\nabla h| d \tilde{\sigma} dt
     =\lambda^2 \int_{\tilde{\Omega}} |\nabla h|^2 d\tilde{v}.     \label{eqn:MJ26_4}
    \end{align}
    Recall that $\varphi$ is a minimizer for $\bar{\boldsymbol{\mu}}(\Omega, g, \tau)$ and that $h$ satisfies the normalization condition (\ref{eqn:MJ26_7}).
    Consequently, it follows from (\ref{eqn:MJ26_3}) and (\ref{eqn:MJ26_4}) that
    \begin{align*}
     \bar{\boldsymbol{\mu}}(\Omega, g, \tau) + \left( \frac{m}{2} \log (4\pi \tau) + m\right)&=\int_{\Omega} \left\{ \tau |\nabla \varphi|^2-\varphi^2 \log \varphi^2 \right\} dv
         \geq \int_{\tilde{\Omega}} \left\{  \tau \lambda^2 |\nabla h|^2 -h^2\log h^2 \right\} d\tilde{v}\\
       &\geq \bar{\boldsymbol{\mu}} \left(\tilde{\Omega}, g_{E}, \tau \lambda^2 \right) +\left( \frac{m}{2} \log (4\pi \tau \lambda^2) + m\right),
    \end{align*}
    which yields (\ref{eqn:MJ26_5}) directly.   Since $\tilde{\Omega} \subset \R^m$, it is clear(c.f. Proposition~\ref{prn:CA01_2}) that
    \begin{align*}
      \bar{\boldsymbol{\mu}} \left(\tilde{\Omega}, g_{E}, \tau \lambda^2 \right) \geq \bar{\boldsymbol{\mu}} \left(\R^{m}, g_{E}, \tau \lambda^2 \right)=0.
    \end{align*}
    Plugging the above inequality into (\ref{eqn:MJ26_5}) and taking infimum over $s \in (0,\tau]$, we obtain (\ref{eqn:MJ26_6}). 
    Then (\ref{eqn:MJ26_8}) follows from the combination of (\ref{eqn:MJ26_5}) and (\ref{eqn:MJ16_6}).
    Similarly, (\ref{eqn:MJ26_9}) follows from the combination of (\ref{eqn:MJ26_6}) and (\ref{eqn:MJ16_7}). 
    The proof of Lemma~\ref{lma:MJ25_1} is complete.    
\end{proof}

Up to now,  we only used the information from scalar curvature bound.  In the next theorem, the Ricci lower bound start to play an important role, because of the comparison geometry. 
We remind the reader that we have the condition (\ref{eqn:CF20_0}) holds by default. 

\begin{theorem}[\textbf{Equivalence of volume ratio and functionals}]
  Suppose
  \begin{align*}
      B=B(x_0,r_0) \subset  \tilde{B}=B(x_0, 5r_0) \subset \Omega. 
  \end{align*}
  Let $\bar{\boldsymbol{\nu}}=\bar{\boldsymbol{\nu}}(B, r_0^2)$ and $\boldsymbol{\nu}=\boldsymbol{\nu} (B, r_0^2)$. Then we have
  \begin{align}
   &-10m^2(1+Kr_0) + (m+1) \log \frac{|B|}{\omega_m r_0^m}    \leq \bar{\boldsymbol{\nu}} \leq  \log \frac{|B|}{\omega_m r_0^m} +2^{m+7},   \label{eqn:CF09_1} \\
   &-\left\{10m^2(1+Kr_0) + \underline{\Lambda} r_0^2 \right\} + (m+1) \log \frac{|B|}{\omega_m r_0^m}    \leq \boldsymbol{\nu} \leq  \log \frac{|B|}{\omega_m r_0^m} + \left\{ 2^{m+7} + \bar{\Lambda} r_0^2\right\}. \label{eqn:CF20_3}
  \end{align}
\label{thm:CF21_2}  
\end{theorem}

\begin{proof}
Clearly, (\ref{eqn:CF20_3}) follows from the combination of (\ref{eqn:CF09_1}) and (\ref{eqn:MJ16_7}) in Lemma~\ref{lma:MJ16_1}. 
Therefore, it suffices to show (\ref{eqn:CF09_1}) only. 

By the scaling invariance of $\bar{\boldsymbol{\nu}}$,  the volume ratios,  the numbers $Kr_0$, $\underline{\Lambda}r_0^2$ and $\bar{\Lambda}r_0^2$, we can assume $r_0=1$ without loss of generality. 

Following the argument of C.B. Croke(c.f. Theorem 13 of~\cite{Croke}), one can obtain the local isoperimetric constant estimate.
Actually, for each point $p \in B=B(x_0, 1)$, it follows from triangle inequality that $\tilde{B}=B(x_0,5) \subset B(p, 6)$. 
Let $\tilde{U}_p$ be the directions $\vec{v}$ in the unit sphere of $T_p M$ such that the unit geodesic staring from $\vec{v}$ is always the shortest geodesic before it hits the boundary of $\tilde{B}$. 
Then calculation in the unit bundle $UM$ implies that(c.f. Lemma 4.2 of~\cite{BW1}):
\begin{align}
  \tilde{\omega} \coloneqq  \inf_{p \in B} |\tilde{U}_p| \geq \frac{|\tilde{B} \backslash B|}{m\omega_m \int_{0}^{6} \left( \frac{\sinh Kr}{K}\right)^{m-1}dr}.   \label{eqn:CF19_4}
\end{align}
Consequently, the isoperimetric constant $\mathbf{I}(B)$ can be bounded by(c.f. Lemma 4.1 of~\cite{BW1}):

\begin{align*}
    \mathbf{I}^m(B) &\geq 2^{m-1} \cdot \frac{(m \omega_m)^m}{\{(m+1) \omega_{m+1}\}^{m-1}} \cdot \tilde{\omega}^{m+1}.  
\end{align*}
Recall that $\mathbf{I}_m=m\omega_m^{\frac{1}{m}}$ by (\ref{eqn:CF19_2}).  It follows from the definition(c.f. (\ref{eqn:CF20_4})) that
\begin{align}
 \lambda^m=
 \left\{ \frac{\mathbf{I}(B)}{\mathbf{I}_m} \right\}^{m} \geq \left( \frac{2\omega_{m}}{(m+1) \omega_{m+1}} \right)^{m-1}  \tilde{\omega}^{m+1}. 
\label{eqn:CF04_0} 
\end{align}
We now focus on the estimate of $\tilde{\omega}$ defined in (\ref{eqn:CF19_4}). 
For each positive number $t$, it is clear that $1<\frac{\sinh t}{t}<e^{t}$.  Therefore, we have
\begin{align}
   \frac{6^m}{m}<\int_{0}^{6} \left( \frac{\sinh Kr}{K}\right)^{m-1}dr<e^{6(m-1)K} \int_{0}^{6} r^{m-1} dr=e^{6(m-1)K} \cdot \frac{6^m}{m}. 
\label{eqn:CF19_3}    
\end{align}
Since $\partial \Omega \neq \emptyset$, we can choose a shortest unit-speed geodesic $\gamma$ connecting $x_0$ to some point on $\partial \Omega$ satisfying $|\gamma| \geq 5$. 
Let $y_0=\gamma(2)$. By triangle inequalities,  it is clear that
\begin{align*}
  B=B(x_0,1) \subset B(y_0, 3) \subset \tilde{B}=B(x_0, 5) \subset \Omega. 
\end{align*}
Note that $B(y_0,1) \subset \tilde{B} \backslash B$. 
In view of the Bishop-Gromov volume comparison, we obtain
\begin{align}
 |\tilde{B} \backslash B| \geq |B(y_0, 1)| \geq |B(y_0, 3)|  \cdot \frac{\int_{0}^{1}  \left( \frac{\sinh Kr}{K}\right)^{m-1} dr}{\int_{0}^{3}  \left( \frac{\sinh Kr}{K}\right)^{m-1} dr}   
  \geq |B| \cdot 3^{-m} \cdot e^{-3(m-1)K},
\label{eqn:CF19_5}  
\end{align}
where we used an inequality similar to (\ref{eqn:CF19_3}) in the last step.  Combining (\ref{eqn:CF19_5}) and (\ref{eqn:CF19_3}) with (\ref{eqn:CF19_4}), we arrive at the estimate of $\tilde{\omega}$:
\begin{align}
  \tilde{\omega} \geq \left( \frac{|B|}{\omega_m} \right) \cdot 18^{-m} \cdot e^{-9(m-1)K}. 
\label{eqn:CF19_6}  
\end{align}
Then we estimate the term in the parenthesis of the right hand side of (\ref{eqn:CF04_0}). 
 In view of the formula (\ref{eqn:CF19_9}), we calculate
\begin{align*}
   \frac{\omega_{m+1}}{\omega_{m}}=\sqrt{\pi} \cdot \frac{\Gamma(\frac{m}{2}+1)}{\Gamma(\frac{m+1}{2}+1)} \leq \sqrt{\pi}, 
\end{align*}
where we used the fact that $m \geq 3$ and $\Gamma$-function is increasing on $[2,\infty)$. 
It follows that
\begin{align}
   \frac{2\omega_m}{(m+1)\omega_{m+1}} \geq \frac{2}{\sqrt{\pi}(m+1)} \geq \frac{1}{2m}
\label{eqn:CF04_1}   
\end{align}
whenever $m \geq 3$. 
Plugging (\ref{eqn:CF04_1}) and (\ref{eqn:CF19_6}) into (\ref{eqn:CF04_0}) and taking logarithm on both sides, we obtain that
\begin{align*}
 m \log \lambda &\geq -(m-1) \log 2m + (m+1) \left\{ \log \frac{|B|}{\omega_m} -m \log 18 -9(m-1)K \right\}\\
   &> -10m^2-10m^2K + (m+1) \log \frac{|B|}{\omega_m}=-10m^2(1+K) + (m+1) \log \frac{|B|}{\omega_m},
\end{align*}
whence we obtain the first inequality of (\ref{eqn:CF09_1}), up to rescaling and application of (\ref{eqn:MJ26_6}). 
The second inequality of (\ref{eqn:CF09_1})  follows from inequality (\ref{eqn:GC28_2}) in Theorem~\ref{thm:CF21_3}. 
The proof of Theorem~\ref{thm:CF21_2} is complete. 
\end{proof}

\section{Li-Yau-Hamilton-Perelman type Harnack inequality}
\label{sec:LYHP}

The purpose of this section is to improve the  Harnack inequality along the Ricci flow. 

\begin{theorem}[Corollary 9.3 of Perelman~\cite{Pe1}]
Suppose $\{(M^m, g(t)), 0 \leq t \leq T\}$ is a Ricci flow solution. 
Let $u$ be a conjugate heat solution, i.e., $\square^* u=(-\partial_t-\Delta+R) u=0$, starting from a $\delta$-function from $(x_0, T)$ for some point $x_0 \in M$.
Then we have
\begin{align}
  v=\left\{(T-t)(2\Delta f -|\nabla f|^2 +R) + f-m\right\} u \leq 0, \quad \textrm{on} \; M \times [0, T),   \label{eqn:CA02_1}
\end{align}
where $f=-\log u -\frac{m}{2} \log \left\{ 4\pi (T-t)\right\}$. 
\label{thm:CA02_1}
\end{theorem}

The inequality (\ref{eqn:CA02_1}) was discovered by Perelman~\cite{Pe1}.   If one regard the Ricci-flat space as a trivial Ricci flow solution, then the conjugate heat equation is nothing but the backward heat equation since $R=0$.
Then the classical Li-Yau estimate(c.f.~\cite{LiYau}) reads as
\begin{align}
  2(T-t) \Delta f -m \leq 0,  \label{eqn:CA02_2}
\end{align}
which is similar(c.f. Section 3.3 of M\"{u}ller~\cite{Muller}) to (\ref{eqn:CA02_1}) and holds for all positive backward heat solutions. 
For evolving space-times,  Hamilton used his matrix-Harnack inequality to study the backward heat solution starting from fundamental solutions in~\cite{Ha93}. 
Li-Yau estimate (\ref{eqn:CA02_2}) suggests that (\ref{eqn:CA02_1}) holds for more conjugate heat solutions. 
We shall generalize (\ref{eqn:CA02_1}) to the case that $u$ is a positive conjugate heat solution starting from a local minimizer function, rather than a $\delta$-function.

\begin{theorem}[\textbf{Li-Yau-Hamilton-Perelman type Harnack inequality}]
Suppose  $\{(M^m, g(t)), 0 \leq t \leq T\}$ is a Ricci flow solution, $\Omega$ is a bounded domain of $M$ with smooth boundary. 
Fix $\tau_T>0$, let $\varphi_T$ be the minimizer function of $\boldsymbol{\mu}(\Omega, g(T), \tau_T)$ for some $\tau_T>0$, $u_T=\varphi_T^2$.
Starting from $u_T$ at time $t=T$, let $u$ solve the conjugate heat equation 
\begin{align}
  \square^* u=(-\partial_t -\Delta +R)u=0.     \label{eqn:MJ11_6}
\end{align}
Define
\begin{align}
 &\tau \coloneqq \tau_T+T-t, \label{eqn:MJ11_7}\\
 &f \coloneqq -\frac{m}{2} \log (4\pi \tau) -\log u, \label{eqn:MJ11_8} \\
 &v \coloneqq \left\{ \tau(2\Delta f -|\nabla f|^2 + R) +f-m-\boldsymbol{\mu} \right\} u,  \label{eqn:MJ11_9}
\end{align}
where $\boldsymbol{\mu}=\boldsymbol{\mu}(\Omega,g(T), \tau_T)$. Then we have 
\begin{align}
  v \leq 0   \label{eqn:CF06_4}
\end{align}
on $M \times [0, T)$.  Moreover, if $v=0$ at some point $(x_0,s_0) \in M \times [0,T)$, then $v \equiv 0$ on $M \times [0,T]$
and $\Omega=M$, the flow is induced by a gradient shrinking Ricci soliton metric. 
\label{thm:CA02_2}
\end{theorem}

Note that Theorem~\ref{thm:CA02_2} differs from Theorem~\ref{thm:CA02_1} by the choice of $u$ and $\tau$. 
In Theorem~\ref{thm:CA02_1}, $u$ is chosen as a conjugate heat solution coming out of a $\delta$-function at time $t=T$ and $\tau=T-t$.
However, in Theorem~\ref{thm:CA02_2}, $u$ arises from a minimizer function $\varphi_T^2$ for $\mu(\Omega, g(T), \tau_T)$ and $\tau=\tau_T+T-t$.
Since $\partial \Omega$ is a non-empty smooth manifold, some extra technical difficulties appear for the analysis of $u$ and $v$.  
Fortunately, the extra efforts needed to prove Theorem~\ref{thm:CA02_2} are justified by the consequence that 
Theorem~\ref{thm:CA02_2} is more general (c.f. Remark~\ref{rmk:MJ13_3}) and provides more information for later applications.

Formally, Theorem~\ref{thm:CA02_2} can be ``proved"  in the same way as Theorem~\ref{thm:CA02_1}. 
Actually, it follows from the Euler-Lagrangian equation satisfied by $\varphi_T$ that 
\begin{align}
  v(x,T)=0, \quad \forall \; x \in \Omega.    \label{eqn:MJ11_14}
\end{align}
By extending $v$ as zero on $M \backslash \bar{\Omega}$, we can regard $v$ as a zero function defined on $M$.
By the calculation of Perelman(c.f. Proposition 9.1 of~\cite{Pe1}), we have
\begin{align*}
\square^* v=-2\tau \left|R_{ij}+ f_{ij}-\frac{g_{ij}}{2\tau} \right|^2 \leq 0, \quad \textrm{on} \;  M \times [0, T).
\end{align*}
Therefore, at least intuitively, one can apply maximum principle to show that $v \leq 0$ on $M \times [0, T)$. 
However, due to the fact that $v$ may not be a continuous function on $M \times [0, T]$, there exist technical difficulties 
in applying such argument directly.
We shall show the same conclusion, using a different approach which requires less regularity of $v$ around $t=T$.

Theorem~\ref{thm:CA02_2} is proved by carefully analyzing the minimizer function $\varphi_T$ and the conjugate heat solution $u$. 
Before we delve into the details of the proof of Theorem~\ref{thm:CA02_2}, we first setup several lemmas. 
For simplicity of notation, from now on,  we always assume that 
\begin{align}
  T=1, \quad   \tau_0=\tau_T+T=\tau_1+1.     \label{eqn:CF06_5}
\end{align}

\begin{lemma}[\textbf{Precise gradient estimate of conjugate heat solutions}]
Then there is a constant $C$ depending on $\Omega$ and the flow such that
\begin{align}
\left|\nabla \sqrt{u} \right| \leq C, \quad \textrm{on} \quad M \times [0,1).
  \label{eqn:MJ10_2}   
\end{align}
\label{lma:MJ10_1}
\end{lemma}

\begin{proof}
Since $\Omega$ is a bounded domain, (\ref{eqn:MJ11_1}) clearly yields that $|\nabla \varphi_1| \leq C$. 
This in turn means that
\begin{align}
   |\nabla u_1|^2 \leq Cu_1     \label{eqn:MJ11_5}
\end{align}
at time $t=1$. Since $u$ satisfies $\square^* u=0$,  direct calculation shows that
\begin{align}
 &\quad \square^* |\nabla u|^2=\left( -\partial_t - \Delta + R\right) |\nabla u|^2 \notag\\
 &=-2Rc(\nabla u, \nabla u)-2\langle \nabla u, \nabla u_t\rangle -2|\nabla \nabla u|^2 - 2Rc(\nabla u, \nabla u) - 2 \langle \nabla u, \nabla \Delta u\rangle + R|\nabla u|^2 \notag\\
 &=-4Rc(\nabla u, \nabla u)-R|\nabla u|^2-2u\langle \nabla u, \nabla R\rangle -2|\nabla \nabla u|^2.   \label{eqn:MJ11_2}
\end{align}
Note that $u$ is a positive function on $M \times [0,1)$, we have
\begin{align*}
  -2u\langle \nabla u, \nabla R\rangle \leq u (|\nabla u|^2 + |\nabla R|^2) \leq u|\nabla u|^2 + K_0u
\end{align*}
where $K_0$ is some constant depending on the geometry of the flow.    Plugging the above inequality into (\ref{eqn:MJ11_2}) implies that
\begin{align*}
 \square^* |\nabla u|^2 \leq K(|\nabla u|^2+u)
\end{align*}
for some constant $K$ depending on the flow.  Consequently, we have
\begin{align*}
 \square^* (|\nabla u|^2 +u)= \square^* |\nabla u|^2 \leq K(|\nabla u|^2+u),
\end{align*}
which yields that
\begin{align}
\square^* \left\{ e^{Kt}(u+|\nabla u|^2) \right\} \leq 0.      \label{eqn:MJ11_10}
\end{align}
At time $t=1$, note that $|\nabla u_1|^2=4\varphi_1^2 |\nabla \varphi_1|^2$.
By (\ref{eqn:MJ11_1}), it is clear that $u_1$ is at least in $C^{1,\frac{1}{2}}(M)$. 
Therefore, $|\nabla u|^2$ is at least a continuous function on $M \times [0, 1]$. 
By (\ref{eqn:MJ11_1}) again, at time $t=1$, we have $|\nabla u_1|^2 \leq C u_1$. 
Therefore, we can choose a large constant $L$ such that $u+|\nabla u|^2 \leq Lu$ at time $t=1$.
Combining (\ref{eqn:MJ11_6}) and (\ref{eqn:MJ11_10}), we obtain
\begin{align*}
 \square^* \left\{ e^{Kt}(u+|\nabla u|^2) - Lu\right\} \leq 0.
\end{align*}
It follows from the maximum principle that 
\begin{align}
  u+|\nabla u|^2 \leq Le^{-Kt} u, \quad \Rightarrow \quad |\nabla u|^2 \leq C u, \quad \textrm{on} \; M \times [0, 1].  \label{eqn:MJ12_9}
\end{align}
Since $u>0$ on $M \times [0,1)$,  it is clear that (\ref{eqn:MJ10_2}) follows from the above inequality immediately. 
\end{proof}

It seems tempting to apply the maximum principle directly to $\nabla \sqrt{u}$ by calculating $\square^* \left|\nabla \sqrt{u} \right|^2$. 
However, for this purpose, one needs better boundary regularity of $\sqrt{u}$ around $\partial \Omega$ at time $t=1$, which is in general hard to obtain.
In the proof of Lemma~\ref{lma:MJ10_1}, we transfer (\ref{eqn:MJ10_2}) to its equivalent version (\ref{eqn:MJ12_9}), whose requirement of regularity is much weaker.

\begin{lemma}[\textbf{Heat kernel and gradient estimate}, c.f. Chau-Tam-Yu~\cite{CTY}]
Let $p$ be the heat kernel function on $M \times [0, 1]$.
There exists a constant depending on the flow and $\Omega$ such that
\begin{align}
 &C^{-1}|t-1|^{-\frac{m}{2}} e^{-\frac{d^2(x,y)}{3|t-1|}}<p(y,1; x,t)< C |t-1|^{-\frac{m}{2}} e^{-\frac{d^2(x,y)}{5|t-1|}}, \quad \forall \; x \in M, y \in \Omega, t \in [0.5,1]; \label{eqn:MJ15_1}\\
 &(1-t) \left| \nabla \sqrt{u} \right|^2 \leq C \left\{ u \log \frac{\norm{u_1}{C^0(M)}}{u} + u\right\}, \quad \forall \; t \in [0.5, 1].  \label{eqn:MJ15_2}
\end{align}
\end{lemma}

\begin{proof}
 Note that on $M \times [0.5,1]$, all curvature and curvature derivatives are bounded, based on our curvature assumption and the curvature derivatives' estimate of Shi(c.f.~\cite{Shi}). 
 Therefore, (\ref{eqn:MJ15_1}) is nothing but the Corollary 5.6 of Chau-Tam-Yu~\cite{CTY}. 
 On the other hand, since $u$ is a nonnegative solution of $\square^* u=0$, the maximum principle implies that $0<u \leq \norm{u_1}{C^0(M)}$. 
 Therefore, (\ref{eqn:MJ15_2}) follows directly from Lemma 6.3 of Chau-Tam-Yu~\cite{CTY}. 
\end{proof}

\begin{lemma}[\textbf{Continuity of integral of gradients}]
Suppose that $h$ is a smooth positive function on $M \times [1-\eta, 1]$ for some small $\eta>0$, then we have
\begin{align}
   \lim_{t \to 1^{-}} \int_{M} \left|\nabla \sqrt{u} \right|^2 h dv= \int_{\Omega} \left|\nabla \sqrt{u_1} \right|^2 h dv_1. 
\label{eqn:MJ11_13}   
\end{align}
\label{lma:MC31_1}
\end{lemma}

\begin{proof}

Since all metrics $g(t)$ and measures $dv_{g(t)}$ are equivalent on $M \times [0.5, 1]$ by the uniform bound of $|Rm|$. In the following discussion, we shall assume $g(1)$ as the default metric and $dv_{g(1)}$
as the default measure.

Fix $\delta>0$.  Define
\begin{align*}
  \Omega_{-\delta} \coloneqq \{x \in M| d(x, \Omega) \leq \delta\}, \quad
  \Omega_{\delta} \coloneqq \{x \in M| d(x, \Omega^c) \geq \delta\}. 
\end{align*}
Choose $y_0 \in \Omega$ and $L$ as a large number such that $\Omega \subset B(y_0, L)$. See Figure~\ref{fig:omegadelta}.

\begin{figure}[H]
 \begin{center}
 \psfrag{A}[c][c]{$y_0$}
 \psfrag{B}[c][c]{$\color{green}\Omega_{\delta}$}
 \psfrag{C}[c][c]{$\color{red}{\Omega}$}
 \psfrag{D}[c][c]{$\color{blue}{\Omega_{-\delta}}$}
 \psfrag{E}[c][c]{$B(y_0, L)$}
 \psfrag{F}[c][c]{$B(y_0, 2L)$}
 \includegraphics[width=0.5 \columnwidth]{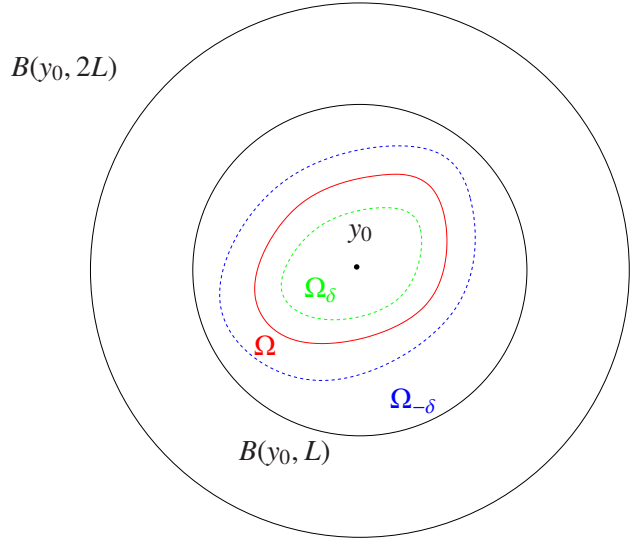}
 \caption{Separating integrals into several domains}
 \label{fig:omegadelta}
 \end{center}
 \end{figure}

  By triangle inequality, for each $x \in M \backslash \Omega_{-\delta}$ and $y \in \Omega$, we have 
  \begin{align*}
  d^2(x,y) \leq (d(x, y_0) + d(y_0, y))^2 \leq 2 d^2(x,y_0)+ 2L^2. 
  \end{align*}
  Recall that $u(x,t)=\int_{\Omega} u_1(y) p(y,1; x, t) dv_{y}$.  By the heat kernel estimate, i.e. (\ref{eqn:MJ15_1}), for each $t \in [0.5,1]$, we have
\begin{align}
 C|t-1|^{-\frac{m}{2}} e^{-\frac{2(d^2(x,y_0)+L^2)}{3|t-1|}}  \leq u(x,t) \leq C |t-1|^{-\frac{m}{2}} e^{-\frac{d^2(x, \Omega)}{5|t-1|}},  \quad \forall\; x \in M \backslash \Omega_{-\delta}.
 \label{eqn:MJ15_3}
\end{align}
For simplicity of notation, we denote $d(x,y_0)$ by $r=r(x)$. 
If $r(x)<2L$, then (\ref{eqn:MJ15_3}) yields that
\begin{align}
 C|t-1|^{-\frac{m}{2}} e^{-\frac{4L^2}{|t-1|}}  \leq u(x,t) \leq C |t-1|^{-\frac{m}{2}} e^{-\frac{\delta^2}{5|t-1|}},  \quad \forall\; x \in B(y_0, 2L) \backslash \Omega_{-\delta}.
 \label{eqn:MJ15_4}
\end{align}
Now we suppose $r(x)>2L$.  Let $z$ be a point in $\bar{\Omega}$ such that $d(x, \Omega)$ is achieved at $z$. 
Then the triangle inequality implies that
\begin{align*}
  d(x,\Omega)=d(x,z)>r(x)-d(z,y_0)>r(x)-L>0.5 r(x). 
\end{align*}
Therefore, (\ref{eqn:MJ15_3}) implies that
\begin{align}
  C|t-1|^{-\frac{m}{2}} e^{-\frac{r^2}{|t-1|}} \leq u(x,t) \leq C|t-1|^{-\frac{m}{2}} e^{-\frac{r^2}{20|t-1|}}, \quad \forall \; x \in M \backslash B(y_0, 2L). \label{eqn:MJ15_5}
\end{align}

We first estimate the integral of $\left| \nabla \sqrt{u} \right|^2$ on $B(y_0, 2L) \backslash \Omega_{-\delta}$. 
In this case, combining (\ref{eqn:MJ15_4}) with the gradient estimate (\ref{eqn:MJ15_2}) implies that
\begin{align*}
  \left|\nabla \sqrt{u} \right|^2(x,t) &\leq C |t-1|^{-\frac{m}{2}-1} e^{-\frac{\delta^2}{5|t-1|}} \left\{ 1+ \frac{m}{2}  \log |t-1| + \frac{4L^2}{|t-1|}\right\} \leq C e^{-\frac{\delta^2}{10|t-1|}},
\end{align*}
where $C$ depends on $m,L, \delta, \Omega$ and the given flow.  It follows that
\begin{align}
  \lim_{t \to 1^{-}} \int_{B(y_0, 2L) \backslash \Omega_{-\delta}}  \left|\nabla \sqrt{u} \right|^2 dv \leq C \Vol(B(y_0, 2L)) \lim_{t \to 1^{-}} e^{-\frac{\delta^2}{10|t-1|}}=0.   \label{eqn:MJ15_6}
\end{align}
Next we estimate the integral of $\left| \nabla \sqrt{u} \right|^2$ on $M \backslash B(y_0, 2L)$. 
Now the combination of (\ref{eqn:MJ15_5}) and (\ref{eqn:MJ15_2}) gives us 
\begin{align*}
  \left|\nabla \sqrt{u} \right|^2(x,t) &\leq C |t-1|^{-\frac{m}{2}-1} e^{-\frac{r^2}{20|t-1|}} \left\{ 1+ \frac{m}{2}  \log |t-1| + \frac{r^2}{|t-1|}\right\} \leq C e^{-\frac{r^2}{40|t-1|}},
\end{align*}
for some $C$ depending on $m,L, \Omega$ and the flow. 
Since $|Rm|$ is bounded for space-time $M \times [0.5,1]$, there is a constant $K$ such that $|B(y_0, r)| \leq e^{Kr^2}$ for each $r>2L$.  
Therefore, by calculating integrals on annular parts $B(y_0, 2^{i+1}L \backslash B(y_0, 2^{i}L))$, we have
\begin{align*}
 \int_{M \backslash B(y_0, 2L)}  \left|\nabla \sqrt{u} \right|^2 dv 
 &\leq  C\sum_{i=1}^{\infty}  e^{KL^2 4^{i}}  e^{-\frac{4^iL^2}{40|t-1|}} \leq C \sum_{i=1}^{\infty} e^{-\frac{4^iL^2}{80|t-1|}}
\end{align*}
when $|t-1|<\frac{1}{80KL^2}$.   Note that the series on the right hand side converges faster than the geometric series when $|t-1|$ very small.
It follows that
\begin{align*}
    \int_{M \backslash B(y_0, 2L)}  \left|\nabla \sqrt{u} \right|^2 dv  \leq C e^{-\frac{L^2}{20|t-1|}}. 
\end{align*}
Consequently, we obtain
\begin{align*}
  \lim_{t \to 1^{-}} \int_{M \backslash B(y_0, 2L)}  \left|\nabla \sqrt{u} \right|^2 dv \leq C \lim_{t \to 1^{-}} e^{-\frac{L^2}{20|t-1|}}=0, 
\end{align*}
which together with (\ref{eqn:MJ15_6}) implies that
\begin{align}
   \lim_{t \to 1^{-}} \int_{M \backslash \Omega_{-\delta}}  \left|\nabla \sqrt{u} \right|^2 dv=0.   \label{eqn:MJ11_3}
\end{align}

On the other hand, note that $u$ is a $C^1$-function at time $t=1$. It follows from standard parabolic equation theory that $|\nabla u|^2$ is a continuous function on $M \times [0, 1]$.
Clearly, $u$ is a continuous function on $M \times [0, 1]$ and $u_1>0$ on $\Omega_{\delta}$. It follows that $u>0$ on $\Omega_{\delta} \times [1-\xi, 1]$ for some small positive number $\xi$.  
Consequently, $|\nabla \sqrt{u}|^2=\frac{|\nabla u|^2}{4u}$ is a continuous function on $\Omega_{\delta} \times [1-\xi, 1]$. 
Because $h$ is a smooth function on $M \times [1-\eta, 1] \supset \Omega_{\delta} \times [1-\xi, 1]$, we immediately obtain
\begin{align}
  \lim_{t \to 1^{-}} \int_{\Omega_{\delta}}  \left|\nabla \sqrt{u} \right|^2 h dv = \int_{\Omega_{\delta}} \left|\nabla \sqrt{u_1} \right|^2 h dv_1.   \label{eqn:MJ11_4}
\end{align}
Combining (\ref{eqn:MJ11_3}) and (\ref{eqn:MJ11_4}), we have
\begin{align*}
 &\quad \lim_{t \to 1^{-}} \left| \int_{M} \left|\nabla \sqrt{u} \right|^2 hdv - \int_{\Omega} \left|\nabla \sqrt{u_1} \right|^2 h dv_1 \right| 
 \leq \lim_{t \to 1^{-}} \left| \int_{\Omega_{-\delta} \backslash \Omega_{\delta}} \left|\nabla \sqrt{u} \right|^2 h dv \right| +  \left| \int_{\Omega_{-\delta} \backslash \Omega_{\delta}} \left|\nabla \sqrt{u_1} \right|^2 h dv_1 \right|\\
 & \leq C_{h} \left\{  \lim_{t \to 1^{-}} \left| \int_{\Omega_{-\delta} \backslash \Omega_{\delta}} \left|\nabla \sqrt{u} \right|^2  dv \right| +  \left| \int_{\Omega_{-\delta} \backslash \Omega_{\delta}} \left|\nabla \sqrt{u_1} \right|^2  dv_1 \right| \right\}
     \leq C\delta, 
\end{align*}
where we used (\ref{eqn:MJ10_2}) and (\ref{eqn:MJ11_5}) in the last step. 
Since the choice of $\delta$ can be arbitrarily small, we obtain (\ref{eqn:MJ11_13}) from the above inequality. 

\end{proof}

 \begin{lemma}[\textbf{Interplay between $v$ and positive smooth functions}]
 For every smooth positive function $h$ defined on $M \times [0, 1]$, we have
 \begin{align}
    \lim_{t \to 1^{-}} \int_M hv = 0,   
 \label{eqn:MC30_1}   
 \end{align}
 where the integral is with respect to the classical volume element induced by $g(t)$. 
 \label{lma:MC30_1}
 \end{lemma}

 \begin{proof}
 For simplicity, we omit the default volume element, which is the classical one induced by $g(t)$. 
  
At each time $t<1$, everything is smooth. Hence integration by parts implies that
 \begin{align*}
   \int_M hv =\int_M \left\{ \tau (R+|\nabla f|^2) +f-m-\boldsymbol{\mu} \right\} uh   -2\tau  \int_M  u \Delta h.
 \end{align*}
 Note that $u$ is a nonnegative continuous function on $M \times [0, 1]$.  Therefore, $u\log u$ and hence $uf$ are continuous functions on $M \times [0, 1]$. 
 Since $h$ is a smooth positive function, we can take limit on both sides of the above inequality to obtain that
 \begin{align}
   \lim_{s \to 1^{-}} \left. \int_M hv \right|_{t=s}= \lim_{s \to 1^{-}} \left. \int_M \tau |\nabla f|^2 uh \right|_{t=s}   +\left. \int_M \left( \tau_1R + f-m-\boldsymbol{\mu} \right) uh \right|_{t=1} -  \left. 2\tau_1\int_M u \Delta h \right|_{t=1}.
 \label{eqn:GC31_1}  
 \end{align}
 Splitting the right hand side of the above inequality into three terms, we shall show that
 \begin{align}
 &\lim_{s \to 1^{-}} \left. \int_M \tau |\nabla f|^2 uh \right|_{t=s}=\left. \int_{\Omega} \tau |\nabla f|^2 uh \right|_{t=1}, \label{eqn:GD01_1} \\
 &\left. \int_M \left( \tau_1R + f-m-\boldsymbol{\mu} \right) uh \right|_{t=1}=\left. \int_{\Omega} \left( \tau_1R + f-m-\boldsymbol{\mu} \right) uh \right|_{t=1}, \label{eqn:GD01_2}\\
 & \left.\int_M u \Delta h \right|_{t=1} = \left. \int_{\Omega} h \Delta u \right|_{t=1}.   \label{eqn:GD01_3}
 \end{align}
 Actually, it follows from the definition of $f$, i.e., equation (\ref{eqn:MJ11_8}), that $u|\nabla f|^2=4|\nabla \sqrt{u}|^2$. 
 Therefore, (\ref{eqn:GD01_1}) is nothing but (\ref{eqn:MJ11_13}). 
 The equation (\ref{eqn:GD01_2}) holds since $u$ and $uf$ are supported on $\Omega$ at time $t=1$.  
 Recall that $\partial \Omega$ is smooth and $\varphi_1 \in C^{2,\alpha}(\bar{\Omega})$ for each $\alpha \in (0,1)$.  
 At time $t=1$, note that $u=\varphi_1^2$ satisfies 
 \begin{align*}
    \frac{\partial u}{\partial \vec{n}}(x)=  2\varphi_1(x) \frac{\partial \varphi_1}{\partial \vec{n}}(x)=0, \quad \; \forall \; x \in \partial \Omega 
 \end{align*}
 where $\vec{n}$ is the outward normal vector field on $\partial \Omega$.  Therefore, it follows from Green's formula that
  \begin{align*}
     \int_M u \Delta h=\int_{\Omega} u \Delta h=\int_{\Omega} h \Delta u + \int_{\partial \Omega} u \nabla h \cdot \vec{n}-h \nabla u \cdot \vec{n}=\int_{\Omega} h \Delta u, 
 \end{align*}
 which is exactly (\ref{eqn:GD01_3}).  Therefore, we have finished the proof of (\ref{eqn:GD01_1}), (\ref{eqn:GD01_2}) and (\ref{eqn:GD01_3}). 
Plugging them into (\ref{eqn:GC31_1}), we obtain
 \begin{align*}
    \lim_{s \to 1^{-}} \left. \int_M hv \right|_{t=s} =\left. \int_{\Omega} \left\{ \tau_1(R+2\Delta f-|\nabla f|^2) +f-m-\boldsymbol{\mu} \right\} uh \right|_{t=1}=0,  
 \end{align*}
which is nothing but (\ref{eqn:MC30_1}).   The proof of Lemma~\ref{lma:MC30_1} is complete. 
 \end{proof}

Now we are ready to prove Theorem~\ref{thm:CA02_2}. 

\begin{proof}[Proof of Theorem~\ref{thm:CA02_2}] 
Recall that $ v=\left\{ \tau\left( 2\Delta f-|\nabla f|^2+R \right)+f-m-\boldsymbol{\mu} \right\}u$ is a smooth function on $M \times [0,1)$
satisfying (\ref{eqn:MC30_1}) in  Lemma~\ref{lma:MC30_1}.
Direct calculation(c.f. Section of Perelman~\cite{Pe1}) shows that
  \begin{align}
    \square^* v=-2\tau \left| R_{ij}+f_{ij}-\frac{g_{ij}}{2\tau}\right|^2 u \leq 0. 
  \label{eqn:GC27_1}  
  \end{align}
 Fix an arbitrary point $(x_0, s_0) \in M \times [0,1)$. Let $w$ be a heat solution, i.e., $\square w=(\partial_t -\Delta) w=0$,  starting from a $\delta$-function 
 at $(x_0, s_0)$.  Because of (\ref{eqn:GC27_1}),  for each $s \in (s_0, 1)$,  integration by parts implies that
 \begin{align*}
   v(x_0, s_0)=\left. \int_{M} wv \right|_{t=s_0}=\left. \int_{M} wv \right|_{t=s} +\int_{s_0}^{s} \int_M w \square^* v=\left. \int_{M} wv \right|_{t=s} -\int_{s_0}^{s} \int_M 2\tau   \left| R_{ij}+f_{ij}-\frac{g_{ij}}{2\tau}\right|^2 wu.
 \end{align*}
 Letting $s \to 1^{-}$, it follows from (\ref{eqn:MC30_1}) that
 \begin{align}
  v(x_0, s_0)=-\int_{s_0}^{1} \int_M 2\tau   \left| R_{ij}+f_{ij}-\frac{g_{ij}}{2\tau}\right|^2 wu \leq 0.     \label{eqn:CF06_6}
 \end{align}
 Therefore,  by the arbitrary choice of $(x_0, s_0)$, we obtain $v \leq 0$ on $M \times [0, 1)$ and we finish the proof of (\ref{eqn:CF06_4}). 
 
 Suppose $v(x_0, s_0)=0$. By strong maximum principle, it is clear that $v \equiv 0$ on $M \times [s_0, 1)$. 
 In other words, we have
 \begin{align}
   \tau(2\Delta f -|\nabla f|^2 +R) + f-m-\boldsymbol{\mu} \equiv 0.   \label{eqn:CF06_8}
 \end{align}
  It follows from (\ref{eqn:CF06_6}) that
 \begin{align}
   R_{ij}+f_{ij}-\frac{g_{ij}}{2\tau} \equiv 0   \label{eqn:CF06_7}
 \end{align}
 on $M \times [s_0, 1)$.  
 We need to show that $\Omega=M$. Actually, taking trace of the shrinking Ricci soliton equation (\ref{eqn:CF06_7}) implies that
 \begin{align*}
    \tau(R+\Delta f)-\frac{m}{2}=0,
 \end{align*}
 which combined with (\ref{eqn:CF06_8}) implies that
 \begin{align}
     \tau (R+|\nabla f|^2) -f+\boldsymbol{\mu}=0.   \label{eqn:CF06_9}
 \end{align}
 Since $\square^* u=0$, it follows from the definition (\ref{eqn:MJ11_9}) that 
   \begin{align*}
      f_{t}=-\Delta f + |\nabla f|^2-R + \frac{m}{2\tau}=|\nabla f|^2.    
   \end{align*}
 On a shrinking Ricci soliton, it is well known(c.f. B.L. Chen~\cite{CBL07}) that $R \geq 0$. Then we have
 \begin{align*}
   f_t=|\nabla f|^2 \leq R+|\nabla f|^2=\frac{f-\boldsymbol{\mu}}{\tau},
 \end{align*}
 where we applied (\ref{eqn:CF06_9}) in the last step.  The above inequality can be written as
 \begin{align*}
   (\tau f)_{t} \leq -\boldsymbol{\mu}.
 \end{align*}
 Note that $u(\cdot, s_0)>0$ everywhere on $M$. Fix an arbitrary point $x \in M$, integrating the above inequality implies that
 \begin{align*}
   \tau_1 f(x, 1) \leq (\tau_1+1-s_0) f(x, s_0) -\boldsymbol{\mu}(1-s_0)<\infty. 
 \end{align*}
 Since $\tau_1>0$, this means $f(x,1)<\infty$ and $u(x,1)>0$. 
 By the arbitrary choice of $x$, we know the support of $u_1=u(\cdot, 1)$ is $M$.
 However, the support of $u_1$ is $\Omega$ according to our assumption.  
 Therefore, we obtain $M=\Omega$. 
 
 We already know that the space-time $M \times [s_0, 1]$ is induced by a gradient shrinking Ricci soliton metric. 
 It follows from the backward uniqueness of the Ricci flow solution(c.f. Kotschwar~\cite{Kot}) that the whole space-time $M \times [0,1]$ is induced by a gradient shrinking Ricci soliton metric.  
 The proof of Theorem~\ref{thm:CA02_2} is complete. 
\end{proof}
 
 Following the proof of Theorem~\ref{thm:CA02_2}, we can obtain several  rigidity results.
 
 \begin{proposition}[\textbf{Strong maximum principle in terms of $v$}]
  Same conditions as in Theorem~\ref{thm:CA02_1}.  
  If  $v=0$ at some point $(x_0, s_0) \in M \times [0, T)$, then $v \equiv 0$ on $M \times [0, T]$ and the flow is induced by
  the Gaussian shrinking soliton metric on $\R^m$.
 \label{prn:CF10_1} 
 \end{proposition}
 
 \begin{proof}
  By the exact same argument as in the proof of Theorem~\ref{thm:CA02_2}, we know that $v \equiv 0$ and (\ref{eqn:CF06_7}) holds on $M \times [s_0, T)$.
  Note that (\ref{eqn:CF06_7}) now reads as
  \begin{align}
     R_{ij} +f_{ij} -\frac{g_{ij}}{2(T-t)}=0.   \label{eqn:CF07_7}
  \end{align}
  In other words, $M \times [s_0, T)$ is induced by a gradient shrinking Ricci soliton metric with termination time $T$. 
  We shall show that this soliton must be the Gaussian shrinking soliton on $\R^m$.  Actually, by the self-similar property of the shrinking soliton, we have
 \begin{align*}
    Q(t)(T-t) \equiv TQ(0),
 \end{align*}
 where $\displaystyle Q(t)=\sup_{x \in M} |Rm|(x,t)$.  By our hypothesis, the flow $\{(M^m, x_0, g(t)),  0 \leq t \leq T\}$ has bounded curvature. 
 So we have
 \begin{align*}
  C \geq  \limsup_{t \to T^{-}} Q(t)= Q(0) \lim_{t \to T^{-}} \frac{T}{T-t},
 \end{align*}
 which forces that $Q(0)=0$ and consequently $Q(t) \equiv 0$ for each $t \in [0,T]$.  In particular, $Rc \equiv 0$ on $M \times [0, T]$.
 At time $t \in [s_0, T)$,  (\ref{eqn:CF07_7}) becomes
 \begin{align*}
     f_{ij}=\frac{g_{ij}}{2(T-t)}, 
 \end{align*}
 which implies that $(M, g(t))$ is a metric cone.  However, a flat manifold which is also a metric cone can only be the Euclidean space $(\R^{m}, g_{E})$ and $f=\frac{d^2(\cdot, y_0)}{4(T-t)}$ for some point $y_0 \in \R^m$.
 Therefore, $\{(M^m, x_0, g(t)),  s_0 \leq t \leq T\}$ is the flat Ricci flow induced by the Gaussian soliton metric on $(\R^{m}, g_{E})$. 
 It then follows from backward uniqueness of Ricci flow again(c.f.~\cite{Kot}) that the whole space-time $M \times [0,1)$ is induced by the Gaussian soliton metric. 
 \end{proof}

 \begin{proposition}[\textbf{Rigidity of Ricci flows in terms of $\boldsymbol{\mu}$}]
 Let $\{(M^m, x_0, g(t)),  0 \leq t \leq T\}$ be a Ricci flow solution. Then we have
 \begin{align}
   \boldsymbol{\mu}(M,g(0),T) \leq 0.   \label{eqn:MJ18_5}
 \end{align}
 Moreover, the equality in (\ref{eqn:MJ18_5}) holds if and only if the flow is the static flow on Euclidean space $(\R^{m}, g_{E})$. 
\label{prn:MJ18_1} 
\end{proposition}

\begin{proof}
 Let $u$ solve the conjugate heat equation $\square^* u=0$ and start from a $\delta$-function at $(x_0,T)$. Then we have
 \begin{align}
  &\quad \boldsymbol{\mu}(M,g(0),T) \notag\\
  &\leq \mathcal{W}^{(R)}\left(M, g(0), \sqrt{u}, T\right) \leq \mathcal{W}^{(R)} \left( M, g(t), \sqrt{u}, T-t \right) \notag\\
  &\leq \lim_{t \to T^{-}} \mathcal{W}^{(R)} \left( M, g(t), \sqrt{u}, T-t \right)=0, \label{eqn:MJ18_6}
 \end{align}
 for every $t \in (0,T)$. Therefore, we finish the proof of (\ref{eqn:MJ18_5}).   
 Clearly, (\ref{eqn:MJ18_5}) becomes an equality if the underlying flow is a static flow on $\R^m$.
 Now we assume (\ref{eqn:MJ18_5}) is an equality.  Combining the equality form of (\ref{eqn:MJ18_5}) and (\ref{eqn:MJ18_6}), we obtain
 \begin{align*}
     \mathcal{W}^{(R)} \left(M, g(t), \sqrt{u}, T-t \right)  \equiv 0, \quad \forall \; t \in (0,T). 
 \end{align*}
 Note that the left hand side of the above equation is exactly $\int_{M} v$, where $v \leq 0$ pointwise by Theorem~\ref{thm:CA02_1}.
 Therefore, we have $v \equiv 0$. Then it follows from Proposition~\ref{prn:CF10_1} that the flow is induced by the Gaussian shrinking soliton  on $(\R^{m}, g_{E})$. 
 \end{proof}

\begin{proposition}[\textbf{Non-positivity of $\boldsymbol{\nu}$}]
Suppose $(M^{m}, g)$ is a complete Riemannian manifold with bounded curvature, $T$ is a positive number.
Then 
\begin{align}
    \boldsymbol{\nu}(M,g,T) \leq 0.  \label{eqn:CA05_1}
\end{align}
Moreover, the equality in (\ref{eqn:CA05_1}) holds if and only if  $(M, g)$ is isometric to the Euclidean space $(\R^{m}, g_{Euc})$.
\label{prn:CA05_1}
\end{proposition}

\begin{proof}
Since $(M,g)$ has bounded curvature,  starting from $g$, there exists a Ricci flow solution $\{(M, g(t)), 0 \leq t \leq \epsilon^2\}$ for some $\epsilon$.
Moreover, each time slice of this flow has bounded curvature.  Without loss of generality, we may assume that $0<\epsilon^2<T$.  It follows from the definition of $\boldsymbol{\mu}$, $\boldsymbol{\nu}$ and Proposition~\ref{prn:MJ18_1} that
  \begin{align}
  \boldsymbol{\nu}(M, g, T)  \leq \boldsymbol{\nu}(M, g, \epsilon^2) \leq \boldsymbol{\mu}(M, g, \epsilon^2)=\boldsymbol{\mu}(M, g(0), \epsilon^2) \leq  0,   \label{eqn:CA05_2}
  \end{align}
  which implies (\ref{eqn:CA05_1}). 
  
  If equality in (\ref{eqn:CA05_1}) holds, then every inequality in (\ref{eqn:CA05_2}) becomes equality. In particular, we have $\boldsymbol{\mu}(M, g(0), \epsilon^2)=0$.
  By Proposition~\ref{prn:MJ18_1} again, we know the flow $\{(M, g(t)), 0 \leq t \leq \epsilon^2\}$ is the static Ricci flow on Euclidean space. In particular, $(M, g)=(M, g(0))$ is isometric to the Euclidean space $(\R^{m}, g_{Euc})$. 
  On the other hand, if $(M,g)$ is isometric to $(\R^{m}, g_{Euc})$, it is clear that $ \boldsymbol{\nu}(M, g, T)=0$.  The proof of Proposition~\ref{prn:CA05_1} is complete. 
\end{proof}

 \begin{remark}
  Theorem~\ref{thm:CA02_1} is implied by Theorem~\ref{thm:CA02_2}.
  Actually, fix an arbitrary point $x_0 \in M$, we can choose small domains $\Omega_i$ with smooth boundaries and $\diam_{g(T)}(\Omega_i) \to 0$.
  Accordingly, we choose positive $\tau_i<<\diam_{g(T)}(\Omega_i)$ such that $\boldsymbol{\mu}(\Omega_i, g(T), \tau_i) \to 0$. 
  Let $\varphi_i$ be the minimizer of $\boldsymbol{\mu}(\Omega, g(T), \tau_i)$, and $u_i$ be the conjugate heat solution $\square^* u=0$, starting from $\varphi_i^2$ at time $T$. 
  Then we define $v_i$ by Theorem~\ref{thm:CA02_2}. It is not hard to see that $u_i$ converges to a solution $u$ such that $\square^* u=0$ on $M \times [0, T)$ and $\displaystyle \lim_{t \to T^{-}} u$
  is the $\delta$-function based at $(x_0, T)$. In this way, it is clear that $v$ is the smooth limit of $v_i$ on $M \times [0, T)$.
  Since each $v_i \leq 0$ on $M \times [0,T)$ by Theorem~\ref{thm:CA02_2}, we obtain $v \leq 0$ and consequently prove Theorem~\ref{thm:CA02_1}. 
 \label{rmk:MJ13_3} 
 \end{remark}

 \begin{theorem}[\textbf{Harnack inequality in terms of positive heat solution values}]
 Same conditions as in Theorem~\ref{thm:CA02_2}.
 Then the following differential Harnack inequality holds:
  \begin{align}
       \frac{d}{d\tau} \left\{ \sqrt{\tau}  f(\boldsymbol{\gamma}(\tau)) \right\} \leq \frac{1}{2} \sqrt{\tau} \left(R+|\dot{\gamma}|^2 \right) + \frac{\boldsymbol{\mu}}{2\sqrt{\tau}}. 
  \label{eqn:CF06_1}     
  \end{align}
 Suppose $x_0, y_0$ are two points on $M$, $\boldsymbol{\gamma}(\tau)=(\gamma(\tau), \tau_T+T-t)$ is a space-time curve parametrized by $\tau=\tau_T+T-t \in [\tau_{T}, \tau_T+T]$ such that
 $\boldsymbol{\gamma}(\tau_T)=(x_0, T)$ and $\boldsymbol{\gamma}(\tau_0)=(y_0, 0)$.  Then the following Harnack inequality holds:
 \begin{align}
       u(y_0, 0) \geq (4\pi \tau_0)^{-\frac{m}{2}} \cdot e^{\left( -1+\sqrt{\frac{\tau_T}{\tau_0}}\right) \boldsymbol{\mu}} \cdot e^{-\frac{\mathcal{L}(\boldsymbol{\gamma})}{2\sqrt{\tau_0}}} \cdot 
       \left\{ (4\pi \tau_T)^{\frac{m}{2}} 
       \cdot   u(x_0, T) \right\}^{\sqrt{\frac{\tau_T}{\tau_0}}}    
 \label{eqn:CF06_2}      
 \end{align}
 where $\mathcal{L}$ is the Lagrangian defined by
 \begin{align}
   \mathcal{L}(\boldsymbol{\gamma}) \coloneqq \int_{\tau_T}^{\tau_0} \sqrt{\tau} \left( R + |\dot{\gamma}|^2 \right) d\tau.   
 \label{eqn:CF06_3}  
 \end{align}
 \label{thm:CF05_1}
 \end{theorem}
 
 \begin{proof}
 Since $\square^* u=0$, it is clear that
   \begin{align}
      f_{\tau}=\Delta f - |\nabla f|^2+R - \frac{m}{2\tau}.    \label{eqn:CF05_2}
   \end{align}
   The fact $v \leq 0$ implies that
   \begin{align}
        \left( \Delta f- \frac{1}{2}|\nabla f|^2+ \frac{1}{2} R \right) + \frac{f-m-\boldsymbol{\mu}}{2\tau} \leq 0.   \label{eqn:CF05_3}
   \end{align}
   Combining (\ref{eqn:CF05_2}) and (\ref{eqn:CF05_3}) yields that
   \begin{align*}
        f_{\tau} - \frac{1}{2}R + \frac{1}{2}|\nabla f|^2 +\frac{f-\boldsymbol{\mu}}{2\tau} \leq 0. 
   \end{align*}
   Consequently, we have
   \begin{align*}
      \frac{d}{d \tau} f(\boldsymbol{\gamma}(\tau))&=f_{\tau} + \langle \nabla f, \dot{\gamma}(\tau)\rangle \leq f_{\tau} +\frac{1}{2}|\nabla f|^2 +\frac{1}{2} |\dot{\gamma}(\tau)|^2
      \leq \frac{1}{2} \left( R+ |\dot{\gamma}|^2 \right) + \frac{\boldsymbol{\mu}-f}{2\tau}, 
   \end{align*}
   whence we obtain (\ref{eqn:CF06_1}).    Integrating (\ref{eqn:CF06_1}), we can apply the definition of $\mathcal{L}$ in (\ref{eqn:CF06_3}) to obtain
   \begin{align*}
       \sqrt{\tau_0} f(\boldsymbol{\gamma}(\tau_0))-\sqrt{\tau_T} f(\boldsymbol{\gamma}(\tau_T)) 
       &\leq \frac{1}{2} \int_{\tau_T}^{\tau_0} \sqrt{\tau} \left( R + |\dot{\gamma}|^2 \right) d\tau + \left( \sqrt{\tau_0} -\sqrt{\tau_T}\right)\boldsymbol{\mu}\\
       &\leq \frac{1}{2} \mathcal{L}(\boldsymbol{\gamma}) + \left( \sqrt{\tau_0} -\sqrt{\tau_T}\right)\boldsymbol{\mu}. 
   \end{align*}
   Thus, we have
   \begin{align*}
      f(\boldsymbol{\gamma}(\tau_0)) \leq  \frac{\mathcal{L}(\boldsymbol{\gamma})}{2\sqrt{\tau_0}} +\sqrt{\frac{\tau_T}{\tau_0}} f(\boldsymbol{\gamma}(\tau_T)) + \left( 1- \sqrt{\frac{\tau_T}{\tau_0}}\right) \boldsymbol{\mu}. 
   \end{align*}
   Note that $\gamma(\tau_0)=y_0, \gamma(\tau_T)=x_0$.  Recall also that $f=-\log u-\frac{m}{2} \log (4\pi \tau)$.  It follows that
   \begin{align*}
     \log u(y_0, 0) + \frac{m}{2} \log (4\pi \tau_0) \geq -\frac{\mathcal{L}(\boldsymbol{\gamma})}{2\sqrt{\tau_0}} +\sqrt{\frac{\tau_T}{\tau_0}} \left( \log u(x_0, T)+\frac{m}{2} \log (4\pi \tau_T)\right) + \left( -1+\sqrt{\frac{\tau_T}{\tau_0}}\right) \boldsymbol{\mu}, 
   \end{align*}
   which is equivalent to (\ref{eqn:CF06_2}). 
 \end{proof}

 \begin{remark}
 The proof of Theorem~\ref{thm:CF05_1} follows the route of the proof of Theorem 2.1 in the fundamental work of Li-Yau~\cite{LiYau}.  Similar argument was used by Perelman in Corollary 9.4 of~\cite{Pe1}. 
 The precise definition of $\mathcal{L}$ under Ricci flow was first given by Perelman in Section 7.1 of his celebrated work~\cite{Pe1}, motivated by physics and infinite dimensional geometry(c.f. Section 5 and 6 of \cite{Pe1}). 
 In the setup of Theorem~\ref{thm:CF05_1},  the quest of the space-time curve $\boldsymbol{\gamma}$ minimizing $\mathcal{L}$ leads naturally to the concept of reduced distance and  reduced geodesic of Perelman(c.f. Section 7 of~\cite{Pe1} or Section~\ref{sec:alter} of the current paper). 
 Therefore, it seems that Theorem~\ref{thm:CF05_1} provides a new perspective to understand the definition of  reduced geodesic and reduced distance, which is  similar to the $\rho$-functional of Li-Yau(c.f. equation (3.1) of~\cite{LiYau}). 
 \label{rmk:CF06_1} 
 \end{remark}
 
 \begin{remark}
 Suppose $\left\{(M, g(t)), 0 \leq t < \infty \right\}$ is the flat Ricci flow on Euclidean space $\R^m$.
 Let $\vec{a} \in \R^m$ and let $x_0=\sqrt{\tau_T} \vec{a}$ and $y_0=\sqrt{\tau_0} \vec{a}$. 
 Let $u_T(\cdot)=(4\pi \tau_T)^{-\frac{m}{2}} e^{-\frac{d^2(\cdot, 0)}{4}}$, which is a minimizer for $\boldsymbol{\mu}=\boldsymbol{\mu}(M, g(T), \tau_T)=0$. 
 Let $\boldsymbol{\gamma}(\tau)=(\sqrt{\tau} \vec{a}, \tau_T+T-\tau)$, which connects $(x_0, T)$ and $(y_0, 0)$.  
 Then direct calculation shows that both (\ref{eqn:CF06_1}) and (\ref{eqn:CF06_2}) become equality.   
 This means that all constants in (\ref{eqn:CF06_1}) and (\ref{eqn:CF06_2})  are sharp.
 \label{rmk:CF06_2}
 \end{remark}

\section{Effective monotonicity formulas for local functionals}
\label{sec:almostmon}

  The formula of $\frac{v}{u}$ in Theorem~\ref{thm:CA02_2}  contains an extra term $-\boldsymbol{\mu}=-\boldsymbol{\mu}(\Omega, g(T), \tau)$, which carries the information of the Riemannian manifold(with boundary) $(\Omega, g(T))$.  We can use this information to relate the geometry of $(\Omega, g(T))$
  with other domains at some time $t<T$.  The study along this direction leads to effective monotonicity formulas along the Ricci flow, which generalize the global monotonicity formulas of Perelman(c.f. Section 3 of~\cite{Pe1}). 
 
 \begin{theorem}[\textbf{Harnack inequality in terms of local $\boldsymbol{\mu}$-functional values}]
  Same conditions as in Theorem~\ref{thm:CA02_2}.  Let $\Omega_0' \subset \Omega_0 \subset M$ and $h: M \to [0,1]$ be a cutoff function supported on $\Omega_0$ and $h \equiv 1$ on $\Omega_0'$.
  Set 
  \begin{align}
     C_h \coloneqq \sup_{\Omega_0} \left|\nabla \sqrt{h} \right|_{g(0)}^2.    \label{eqn:CF07_0}
  \end{align}
  For each $\tau_T>0$,  set
  \begin{align}
   \boldsymbol{\mu}_0 \coloneqq \boldsymbol{\mu} \left( \Omega_0, g(0), \tau_0 \right), \quad \boldsymbol{\mu}_T \coloneqq \boldsymbol{\mu}(\Omega, g(T), \tau_T). 
  \label{eqn:CF04_7}  
  \end{align}
  Then we have 
    \begin{align}
   \boldsymbol{\mu}_0-\boldsymbol{\mu}_{T}
   \leq \frac{\int_{\Omega_0 \backslash \Omega_0'}  \left\{ \tau_0 |\nabla h|^2 h^{-1} - h \log h\right\} u}{\int_{\Omega_0} uh}
   \leq \left( 4\tau_0 C_h + e^{-1}\right)  \cdot \frac{\int_{\Omega_0 \backslash \Omega_0'} u}{\int_{\Omega_0'} u}. 
  \label{eqn:CF04_3} 
  \end{align}
 \label{thm:CF04_1} 
 \end{theorem}

 \begin{proof}  
  Since $-x \log x \leq e^{-1}$ for each positive $x$, it is obvious that the second inequality of (\ref{eqn:CF04_3}) follows from the first inequality of (\ref{eqn:CF04_3}). 
  Therefore, we only need to prove the first part of (\ref{eqn:CF04_3}), which will be discussed in details in the next paragraph.  
  
 Without loss of generality, we assume $T=1$.  
 For simplicity of notations,  we define
 \begin{align}
     S \coloneqq \left. \int_{\Omega_0} u h \right|_{t=0} \leq \left. \int_{\Omega_0} u  \right|_{t=0} \leq \left. \int_{M} u  \right|_{t=0}=1.    \label{eqn:CF04_2}
 \end{align}
 Let $v=\left\{\tau(2\Delta f-|\nabla f|^2 +R) +f-m-\boldsymbol{\mu}_1 \right\}u$ as in Theorem~\ref{thm:CA02_2}. 
  At time $t=0$,  let $\tilde{u}$ be $\frac{uh}{S}$. Then $\tilde{u}$ satisfies the normalization condition(c.f. (\ref{eqn:MJ17_3})) $\int_{M} \tilde{u}=1$ and is supported on $\Omega_0$.
  Accordingly, we define
  \begin{align*}
    \tilde{f} \coloneqq -\log \tilde{u}-\frac{m}{2} \log (4\pi \tau_0)=-\log u-\frac{m}{2} \log (4\pi \tau_0)-\log h+\log S=f-\log h + \log S,
  \end{align*}
  where $\tau_0=\tau_1+1$. 
  Plugging $\sqrt{\tilde{u}}$ into the formula (\ref{eqn:MJ16_b}), we obtain
  \begin{align*}
     &\quad \boldsymbol{\mu}_0+m\\
     &\leq \mathcal{W}^{(R)}\left( \Omega_0, g(0), \sqrt{\tilde{u}}, \tau_0 \right)+m=\int_{\Omega_0} \left\{ \tau_0(R+ 2\Delta \tilde{f}-|\nabla \tilde{f}|^2) + \tilde{f}\right\}  \tilde{u}\\
     &=\frac{1}{S}\int_{\Omega_0} \left\{ \tau_0 \left(R+ 2\Delta f -2\Delta \log h-|\nabla f |^2-|\nabla \log h|^2 +2 \langle \nabla f, \nabla \log h \rangle \right) +f -\log h + \log S \right\} uh.
  \end{align*}
  Putting the expression of $v$ into the above inequality, we arrive at
  \begin{align}
   \boldsymbol{\mu}_0 &\leq \log S + \int_{\Omega_0}  \left(\frac{v}{u}+\boldsymbol{\mu}_1 \right) \tilde{u} + \frac{1}{S} \int_{\Omega_0} \left\{ \tau_0 \left( -2\Delta \log h-|\nabla \log h|^2 +2 \langle \nabla f, \nabla \log h\rangle \right) -\log h\right\} uh \notag\\
     &=\log S + \boldsymbol{\mu}_1 + \frac{1}{S} \int_{\Omega_0} vh +  \frac{1}{S} \int_{\Omega_0} \left\{ \tau_0 |\nabla \log h|^2   -\log h\right\} uh \notag\\
     &=\boldsymbol{\mu}_1+\left\{ \log S + \frac{1}{S} \int_{\Omega_0} vh \right\} +\frac{1}{S} \int_{\Omega_0} \left\{ 4\tau_0 \left|\nabla \sqrt{h} \right|^2   - h\log h\right\} u. 
  \label{eqn:CF07_5}   
  \end{align}
  At time $t=0$, note that both $\left|\nabla \sqrt{h} \right|$ and $h \log h$ are supported on $\Omega_0 \backslash \Omega_0'$.  
  Recall the fact that $v \leq 0$ and $0<S \leq 1$(c.f. Theorem~\ref{thm:CA02_2} and inequality (\ref{eqn:CF04_2})). Then it follows from the above inequality that
  \begin{align}
   \boldsymbol{\mu}_0-\boldsymbol{\mu}_1 &\leq \left\{ \log S + \frac{1}{S} \int_{\Omega_0} vh \right\} + \frac{1}{S}\int_{\Omega_0 \backslash \Omega_0'}  \left\{ 4\tau_0 \left|\nabla \sqrt{h} \right|^2 - h \log h\right\} u \notag\\
    &\leq \frac{1}{S}\int_{\Omega_0 \backslash \Omega_0'}  \left\{ 4\tau_0 \left|\nabla \sqrt{h} \right|^2 - h \log h\right\} u,    \label{eqn:CF18_2}
  \end{align}
  whence we obtain (\ref{eqn:CF04_3}) by the choice of $\boldsymbol{\mu}_0$, $\boldsymbol{\mu}_1$, $\tau_0$ and $S$(c.f. (\ref{eqn:CF04_7}) and (\ref{eqn:CF04_2})). 
 \end{proof}

 \begin{theorem}[\textbf{Effective monotonicity of local functionals}]
 Same conditions and notations as in Theorem~\ref{thm:CA02_2}. 
 For each $\lambda \in [0, \tau_T]$, let $\varphi_T^{(\lambda)}$ be the minimizer function of $\boldsymbol{\mu}(\Omega, g(T), \lambda)$. 
 Let $u^{(\lambda)}$ be the conjugate heat solution starting from $\left( \varphi_T^{(\lambda)} \right)^2$. Set
 \begin{align}
      c_u \coloneqq \inf_{\lambda \in (0, \tau_T]} \int_{\Omega_0'} u^{(\lambda)}.   \label{eqn:CF07_1} 
 \end{align} 
 Define
  \begin{align}
   \boldsymbol{\nu}_0 \coloneqq \boldsymbol{\nu} \left( \Omega_0, g(0), \tau_0 \right), \quad \boldsymbol{\nu}_T \coloneqq \boldsymbol{\nu}(\Omega, g(T), \tau_T). 
  \label{eqn:CF07_3}  
  \end{align}
 Then we have
  \begin{align}
     \boldsymbol{\nu}_0-\boldsymbol{\nu}_{T}
    \leq \inf_{\tau \in [T, \tau_T+ T]} \boldsymbol{\mu} \left( \Omega_0, g(0), \tau \right)-\boldsymbol{\nu}(\Omega, g(T), \tau_T) 
    \leq \left( 4\tau_0 C_h + e^{-1}\right)  \cdot \left\{ c_u^{-1}-1 \right\}. 
  \label{eqn:CF04_5}   
  \end{align}
 Consequently, we have
 \begin{align}
   \max \left\{\boldsymbol{\mu}_0-\boldsymbol{\mu}_{T}, \boldsymbol{\nu}_0-\boldsymbol{\nu}_{T} \right\} \leq \left( 4\tau_0 C_h + e^{-1}\right) \cdot \left\{c_u^{-1}-1 \right\}.
 \label{eqn:CF05_1}  
 \end{align}
 \label{thm:CF07_1}
 \end{theorem}

 \begin{proof}
  It is clear that the first inequality of (\ref{eqn:CF04_5}) follows directly from the  definition of $\boldsymbol{\nu}$(c.f. equation (\ref{eqn:MJ16_d}) and (\ref{eqn:MJ16_1})). 
  According to inequality (\ref{eqn:CF04_3}), we have
 \begin{align*}
     \boldsymbol{\mu} \left( \Omega_0, g(0), \tau_T'+T \right)  \leq \left( 4(\tau_T'+T) C_h + e^{-1}\right)  \cdot \frac{\int_{\Omega_0 \backslash \Omega_0'} u}{\int_{\Omega_0'} u} 
     +\boldsymbol{\mu} \left(\Omega, g(T), \tau_T' \right), 
 \end{align*}
 where $u=u^{(\tau_T')}$ is clearly a positive function on $M \times [0,T)$.  Consequently, we have
  \begin{align*}
     \frac{\int_{\Omega_0 \backslash \Omega_0'} u}{\int_{\Omega_0'} u} \leq \frac{\int_{M \backslash \Omega_0'} u}{\int_{\Omega_0'} u}=\frac{\int_M u-\int_{\Omega_0'} u}{\int_{\Omega_0'} u}=\frac{1}{\int_{\Omega_0'} u}-1
     \leq c_u^{-1}-1
 \end{align*}
 where we used the definition (\ref{eqn:CF07_1}) in the last step.  Combining the previous two steps, we arrive at
 \begin{align}
    \boldsymbol{\mu} \left( \Omega_0, g(0), \tau_T'+T \right)  \leq \left( 4(\tau_T'+T) C_h + e^{-1}\right)  \cdot \left\{ c_u^{-1} -1 \right\}
     +\boldsymbol{\mu} \left(\Omega, g(T), \tau_T' \right)
 \label{eqn:CF07_4}    
 \end{align}
  for each $\tau_T' \in (0, \tau_T]$.  Letting $\tau_T'$ run over $[0, \tau_T]$ and taking infimum of the right hand side of the above inequality, we obtain the second inequality of (\ref{eqn:CF04_5}). 
 So we finish the proof of (\ref{eqn:CF04_5}).  
 
The inequality (\ref{eqn:CF05_1}) follows from the combination of (\ref{eqn:CF04_5}) and (\ref{eqn:CF07_4}) by setting $\tau_T'=\tau_T$.  
 \end{proof}

 A  particular case of  Theorem~\ref{thm:CF07_1} is to choose $\Omega=\Omega_0=\Omega_0'=M$.    
 Then we  obtain the monotonicity of the global functionals $\boldsymbol{\mu}$ and $\boldsymbol{\nu}$ of Perelman(c.f. Section 3 of~\cite{Pe1}). 
 
\begin{theorem}[\textbf{Monotonicity of $\boldsymbol{\mu}$ and $\boldsymbol{\nu}$-functionals, Perelman}]
Let $\{(M^m, g(t)), 0 \leq t \leq T\}$ be a Ricci flow solution on a closed manifold. 
Then we have
 \begin{align}
    &\boldsymbol{\mu}(M, g(T), \tau_T) -\boldsymbol{\mu}(M, g(0), \tau_{T}+T) \geq 0,     \label{eqn:MJ17_12}\\
    &\boldsymbol{\nu}(M, g(T), \tau_T) -\boldsymbol{\nu}(M, g(0), \tau_{T}+T) \geq 0 ,   \label{eqn:CF06_10}
  \end{align}
  for every $\tau_T>0$.  Moreover, if equality in (\ref{eqn:MJ17_12}) or (\ref{eqn:CF06_10}) holds, then the flow is induced by a gradient shrinking soliton metric.
\label{thm:CA05_1}  
\end{theorem} 

\begin{proof}
Without loss of generality, we assume $T=1$

Since $\Omega_0'=M$, it is clear that $\int_{\Omega_0'} u^{(\lambda)} \equiv 1$ for each $\lambda \in (0, \tau_1]$.  It follows from (\ref{eqn:CF07_1}) that $c_u=1$.
Therefore, both (\ref{eqn:MJ17_12}) and (\ref{eqn:CF06_10}) follows from (\ref{eqn:CF05_1}) of Theorem~\ref{thm:CF07_1}. 

We now focus on the equality case of (\ref{eqn:MJ17_12}).
We follow the notations in the proof of Theorem~\ref{thm:CF04_1}. 
It follows from (\ref{eqn:CF07_5}) that
\begin{align*}
 -\left\{ \log S + \frac{1}{S} \int_{\Omega_0} vh \right\}  \leq  (\boldsymbol{\mu}_1-\boldsymbol{\mu}_0)-\frac{1}{S} \int_{\Omega_0} \left\{ 4\tau_0 \left|\nabla \sqrt{h} \right|^2   - h\log h\right\} u.
\end{align*}
In the current situation, we have $h \equiv 1$ on $\Omega_0=M$ and $\boldsymbol{\mu}_1-\boldsymbol{\mu}_0=0$.  So the right hand side of the above inequality vanishes. 
Recall that the left hand side of the above term is nonnegative(c.f. (\ref{eqn:CF04_2}) and  (\ref{eqn:CF06_4}) in Theorem~\ref{thm:CA02_2}).  Therefore, it is forced to be zero and we obtain $S=1$ and $v \equiv 0$ on $M$ at $t=0$. 
In light of the second part of Theorem~\ref{thm:CA02_2}, we obtain that the flow is induced by a gradient shrinking Ricci soliton metric. 

If (\ref{eqn:CF06_10}) becomes an equality, we shall show the flow is also induced by a soliton metric.
Actually, it follows from Proposition~\ref{prn:CA05_1} that $\boldsymbol{\nu}(M, g(1), \tau_1)<0$ as $M$ is closed. 
On the other hand, from the proof of Proposition~\ref{eqn:CA01_3}, it is clear that 
\begin{align*}
\lim_{s \to 0^+} \boldsymbol{\nu}(M, g(1), s)=0.
\end{align*}
Therefore, there exists some $\tau_1' \in (0, \tau_1]$ such that 
\begin{align*}
  \boldsymbol{\nu}(M, g(1), \tau_1)=\inf_{\tau \in (0, \tau_1]}  \boldsymbol{\mu}(M, g(1), \tau)=\boldsymbol{\mu}(M, g(1), \tau_1'). 
\end{align*}
Using (\ref{eqn:MJ17_12}), we then have
\begin{align*}
    \boldsymbol{\nu}(M, g(1), \tau_1)=\boldsymbol{\mu}(M, g(1), \tau_1') \geq \boldsymbol{\mu}(M, g(0), 1+\tau_1') \geq \boldsymbol{\nu}(M, g(0), 1+\tau_1). 
\end{align*}
Since now (\ref{eqn:CF06_10}) is an equality, we know all the inequalities in the above line become equalities. In particular, we obtain
\begin{align*}
  \boldsymbol{\mu}(M, g(1), \tau_1')=\boldsymbol{\mu}(M, g(0), 1+\tau_1').
\end{align*}
Therefore we return to the equality case of (\ref{eqn:MJ17_12}) and obtain that the flow is induced by a gradient shrinking Ricci soliton metric. 
\end{proof}

 In light of Theorem~\ref{thm:CF07_1}, the estimate of the difference of the local functionals 
 is reduced to the estimate of two numbers: the upper bound of $C_h$(c.f. (\ref{eqn:CF07_0})) and the lower bound of $c_u$ (c.f. (\ref{eqn:CF07_1})).
 If we assume $\Omega_0$ and $\Omega_0'$ to be concentric geodesic balls, say $\Omega_0=B(x_0,2r)$ and $\Omega_0'=B(x_0,r)$, then there is a natural way to estimate $C_h$.
 Actually, we can choose $\sqrt{h}$ as a cutoff function such that $\left|\nabla \sqrt{h} \right|<2r^{-1}$.  Then it follows from (\ref{eqn:CF07_0}) that 
 \begin{align}
      C_h<4r^{-2}.    \label{eqn:CF07_6}
 \end{align}
 The difficult step is to estimate the lower bound of $c_u$. 
 In the remaining part of this section, we shall estimate $c_u$ in the case that $\Omega_0'$ is very large compared to $\Omega$.
 However, the really hard case is that $\Omega_0'$ is  very small compared to $\Omega$, which situation will be discussed in Section~\ref{sec:alter}.

\begin{theorem}[\textbf{Almost monotonicity of local-$\boldsymbol{\mu}$-functional}]
Let $A \geq 1000m$ be a large constant. 
Let $\{(M^m, g(t)), 0 \leq t \leq T\}$ be a Ricci flow solution satisfying
\begin{align}
    t \cdot Rc(x,t) \leq (m-1)A, \quad \forall \; x \in B_{g(t)} \left(x_0, \sqrt{t} \right), \; t \in (0, T].   \label{eqn:MJ17_11}
\end{align}
Then we have
\begin{align}
  \boldsymbol{\mu}(\Omega_T', g(T), \tau_T) -\boldsymbol{\mu}(\Omega_0, g(0), \tau_{T}+T) \geq -A^{-2},  \quad \forall \;  \tau_T \in (0, A^2 T),  \label{eqn:MJ17_9}
\end{align} 
where $\Omega_T'=B_{g(T)}\left(x_0,  8A \sqrt{T} \right)$ and $\Omega_0=B_{g(0)}\left(x_0, 20A\sqrt{T} \right)$.   In particular, we have
\begin{align}
  \boldsymbol{\nu}(\Omega_T', g(T), \tau_T) \geq -A^{-2} +\inf_{\tau \in [T, \tau_T+T]} \boldsymbol{\mu}(\Omega_0, g(0), \tau) \geq -A^{-2} + \boldsymbol{\nu}(\Omega_0, g(0), \tau_{T}+T). 
  \label{eqn:MJ17_10}
\end{align}
\label{thm:CA02_3}
\end{theorem}

Note that the almost monotonicity inequality (\ref{eqn:MJ17_9}) in Theorem~\ref{thm:CA02_3} can also be regarded as a generalization of the monotonicity inequality (\ref{eqn:MJ17_12})
 in Theorem~\ref{thm:CA05_1}.  Actually, for a given Ricci flow $M \times [0,T]$, a given point $x_0 \in M$ and a given scale $\tau_T>0$, we can always choose a very large $A$ such that (\ref{eqn:MJ17_11}) holds. 
  Then applying (\ref{eqn:MJ17_9}) and letting $A \to \infty$,  in light of (\ref{eqn:CA02_6}), we obtain (\ref{eqn:MJ17_12}).
  Similarly,  one can obtain (\ref{eqn:CF06_10}) by (\ref{eqn:MJ17_10}). 
 
  A key step of the proof of Theorem~\ref{thm:CA02_3} is to apply the condition (\ref{eqn:MJ17_11}) to obtain the $c_u$ lower bound in Theorem~\ref{thm:CF07_1}. 
  For the convenience of the readers, we recall the following distance estimates which will be used repeatedly in this paper.   
  Their detailed proofs can be found in Section 26 of Kleiner-Lott~\cite{KL}. 
 
 \begin{lemma}[\textbf{Distance estimate}, c.f. Lemma 8.3 of Perelman~\cite{Pe1} and Section 17 of Hamilton~\cite{Ha95}]
 Suppose $\{(M, g(t)), 0 \leq t \leq T\}$ is a Ricci flow solution on a complete manifold $M$ and $t_0 \in [0, T]$. 
 
 (a). If $Rc(x,t_0) \leq (m-1)K$ in the ball $B_{g(t_0)}(x_0, r_0)$, then we have
    \begin{align}
      \square d=  \left( \partial_t -\Delta \right) d \geq -(m-1) \left( \frac{2}{3}Kr_0 + r_0^{-1}\right)    \label{eqn:ML14_1}
    \end{align}
    whenever $t=t_0$ and $d=d_{g(t)}(\cdot, x_0)>r_0$. 
    
  (b). If $Rc(x, t_0) \leq (m-1)K$ in the union of balls $B_{g(t_0)}(x_0,r_0) \cup B_{g(t_0)}(y_0,r_0)$, then we have
    \begin{align}
       \frac{d}{dt} d_{g(t)}(x_0, y_0) \geq -2(m-1) \left( \frac{2}{3}Kr_0 + r_0^{-1}\right)     \label{eqn:ML14_2}
    \end{align}  
    at time $t=t_0$. 
 \label{lma:ML14_1}   
 \end{lemma}

 Combining Lemma~\ref{lma:ML14_1} with the condition (\ref{eqn:MJ17_11}), we have the following Lemma. 
 
 \begin{lemma}
 Let $\{(M^m, g(t)), 0 \leq t \leq T\}$ be a Ricci flow solution satisfying
\begin{align*}
    t \cdot Rc(x,t) \leq (m-1)A, \quad \forall \; x \in B_{g(t)} \left(x_0, \sqrt{t} \right), \; t \in (0, T].   
\end{align*}
Let $d(x,t)=d_{g(t)}(x,x_0)$.  Then we have
\begin{align}
   \square \left( d+2A\sqrt{t}\right) \geq 0.  \label{eqn:CC22_1}
\end{align}
\label{lma:CC22_1}
\end{lemma}

\begin{proof}
At time $t \in (0,T]$, let $r_0=\sqrt{t}$ and $K=\frac{A}{(m-1)t}$.
  Then we have $Rc(\cdot, t) \leq (m-1)K$ in $B_{g(t)}(x_0,r_0)$. Consequently, we can apply (\ref{eqn:ML14_1}) to obtain
  \begin{align}
    \square d \geq -(m-1) \cdot \left( \frac{2A}{3(m-1)} + 1\right) \cdot t^{-\frac{1}{2}}=-\left( \frac{2A}{3} + (m-1)\right) t^{-\frac{1}{2}} \geq  -\frac{A}{\sqrt{t}},   \label{eqn:MJ17_1}
  \end{align}
which is equivalent to (\ref{eqn:CC22_1}). 
\end{proof}

Now we are ready for the proof of Theorem~\ref{thm:CA02_3}. 

\begin{proof}[Proof of Theorem~\ref{thm:CA02_3}]
  Without loss of generality, we assume $T=1$. Same as Lemma~\ref{lma:CC22_1},  we set $d$ be the function $d_{g(t)}(\cdot, x_0)$. 
    Let $\psi$ be a cutoff function such that $\psi \equiv 1$ on $(-\infty, 1)$,  $\psi \equiv 0$ on $(2, \infty)$ and $-10 \leq \psi' \leq 0$ everywhere.
  Moreover, $\psi$ satisfies
    \begin{align}
    \psi'' \geq -10 \psi, \quad (\psi')^{2} \leq 10 \psi.   \label{eqn:MJ17_2}
  \end{align}
  The choice of $\psi$ is inspired by the proof of pseudo-locality theorem of Perelman~\cite{Pe1}. 
  The proof of the existence of such $\psi$ can be found at the beginning of Section 34 of Kleiner-Lott~\cite{KL}.
  See Figure~\ref{fig:boundedcutoff} for a graph of $\psi$. 
  
 \begin{figure}[H]
 \begin{center}
 \psfrag{A}[c][c]{$\color{brown}{s=2}$}
 \psfrag{B}[c][c]{$1$}
 \psfrag{C}[c][c]{$y$}
 \psfrag{D}[c][c]{$s$}
 \psfrag{E}[c][c]{$s=1$}
 \psfrag{F}[c][c]{$y=\psi(s)$}
 \psfrag{G}[c][c]{$0$}
 \includegraphics[width=0.3 \columnwidth]{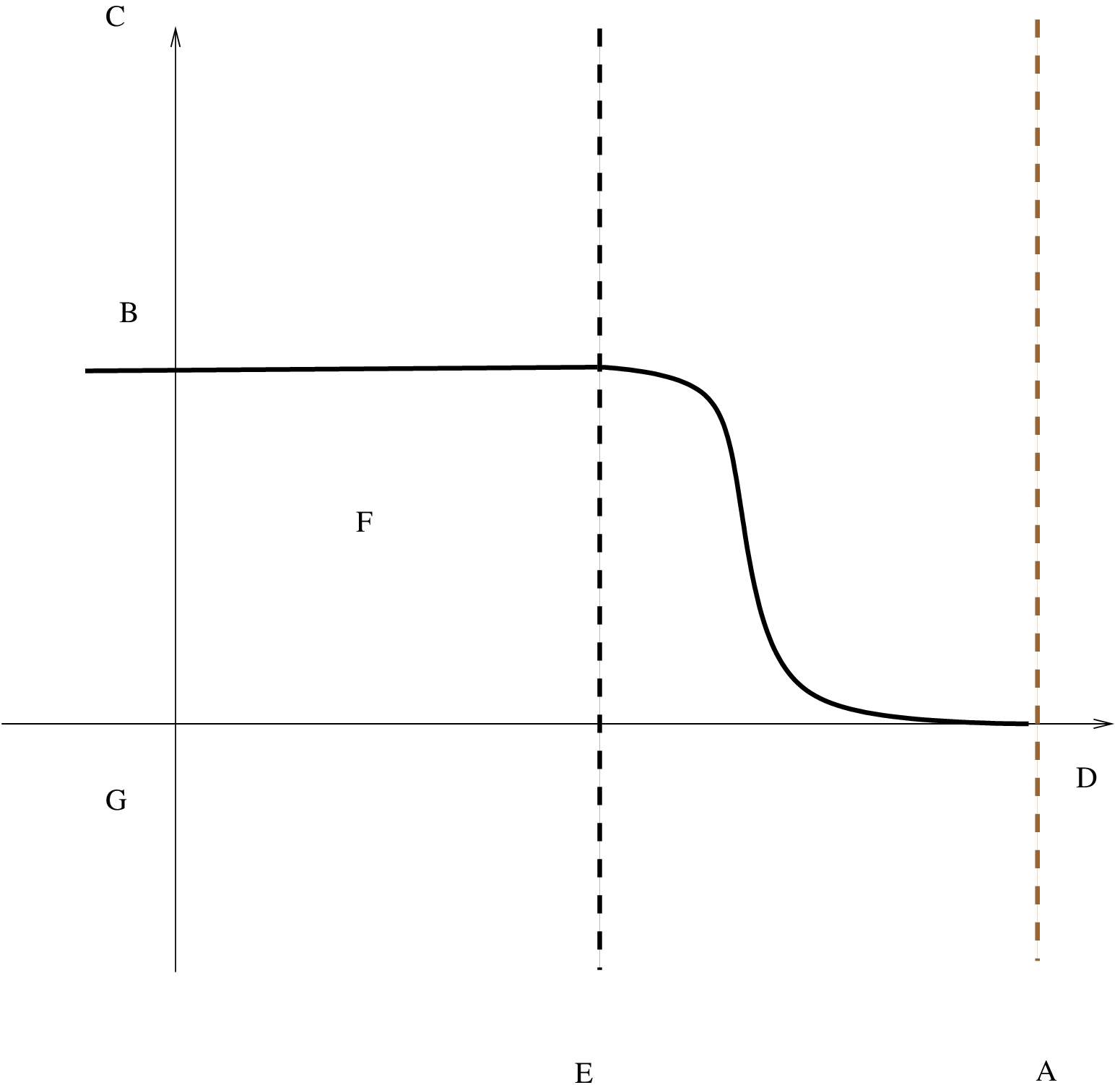}
 \caption{The choice of bounded cutoff function $\psi$}
 \label{fig:boundedcutoff}
 \end{center}
 \end{figure}
  
  \noindent
  By abusing of notations and setting $\psi= \psi\left( \frac{d+2A\sqrt{t}}{10A} \right)$, we can regard $\psi$ as a function on space-time.  
  For each $t \in [0, 1]$, we define
  \begin{align}
     \Omega_t \coloneqq B_{g(t)}\left(x_0, 20A-2A\sqrt{t} \right), \quad \Omega_t'  \coloneqq B_{g(t)}\left( x_0, 10A-2A\sqrt{t} \right).  \label{eqn:CA03_1}
  \end{align}
  It follows from the definition that
  \begin{align}
     \psi (x, t)=
  \begin{cases}
  1, \quad \forall \; x \in \Omega_t';\\
  0, \quad \forall \; x \in M \backslash \Omega_t.
  \end{cases}   
  \label{eqn:CF04_6}
  \end{align}
  We shall study the behavior of the conjugate heat solutions on the space-time support set of $\psi$.  The different domains are illustrated in Figure~\ref{fig:semilocal}. 
  
 \begin{figure}[H]
 \begin{center}
 \psfrag{A}[c][c]{$M$}
 \psfrag{B}[c][c]{$t=1$}
 \psfrag{BBB}[c][c]{$t=0$}
 \psfrag{C}[c][c]{$t$}
 \psfrag{E}[c][c]{$\color{green}{\partial B_{g(t)}\left(x_0, \sqrt{t} \right)}$}
 \psfrag{F}[c][c]{$\partial \Omega_t'$}
 \psfrag{H}[c][c]{$(x_0,0)$}
 \psfrag{I}[c][c]{$\color{brown}{\partial \Omega_t}$}
 \includegraphics[width=0.6 \columnwidth]{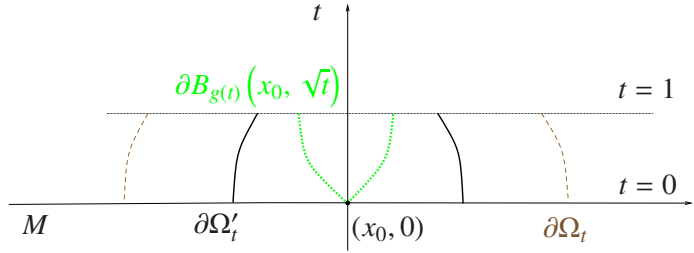}
 \caption{Different domains for almost monotonicity}
 \label{fig:semilocal}
 \end{center}
 \end{figure}
 
  Applying (\ref{eqn:MJ17_1}) and (\ref{eqn:MJ17_2}), we have
  \begin{align*}
    \square \psi=\square  \psi\left( \frac{d+2A\sqrt{t}}{10A} \right) =\frac{1}{10 A} \left(\square  d +\frac{A}{\sqrt{t}} \right) \psi' -\frac{1}{(10A)^2} \psi'' \leq \frac{\psi}{10A^2}. 
  \end{align*}
  Let $h=e^{-\frac{t}{10A^2}}\psi$. Then we have $\square h \leq 0$. 
  At time $t=1$, the support set of $h$ is $\Omega_1=B(x_0, 18A)$ and $h \equiv 1$ on $\Omega_1'=B(x_0, 8A)$.
  Let $\varphi_1^{(\tau_1)}$ be a minimizer for $\boldsymbol{\mu}_1=\boldsymbol{\mu}(\Omega_1', g(1), \tau_1)$ for some positive number $\tau_1$.
  Starting from $u_1=\left( \varphi_1^{(\tau_1)} \right)^{2}$, we solve the equation $\square^* u=0$. Recall that $u>0$ on $M \times [0,1)$.
  Thus, we have
  \begin{align*}
     \frac{d}{dt} \int_{M} uh = \int_{M} \left\{ u(\square h)-h\square^* u \right\}  \leq 0.
  \end{align*}
  Integrating the above inequality yields that
  \begin{align*}
    \left. \int_{M} uh \right|_{t=0} \geq \left. \int_{M} uh \right|_{t=1}= \left. \int_{\Omega} uh \right|_{t=1}=e^{-\frac{1}{10A^2}}  \left. \int_{\Omega} u \right|_{t=1}
    =e^{-\frac{1}{10A^2}}.  
  \end{align*}
  At time $t=0$, $h=\psi$ by definition. Therefore, $h \equiv 1$ on $\Omega_0'$ and $h \equiv 0$ outside $\Omega_0'$. It follows that
  \begin{align}
  1 \geq  \left. \int_{\Omega_0} u \right|_{t=0} \geq   \left. \int_{M} uh \right|_{t=0} \geq e^{-\frac{1}{10A^2}}.   \label{eqn:MJ17_3}
  \end{align}
  Note that $\Omega_0'=B_{g(0)}(x_0,10A)$ and $10A=20 \cdot \frac{A}{2}$. Similar to (\ref{eqn:MJ17_3}), we have
   \begin{align}
    1 \geq  \left.  \int_{\Omega_0'} u \right|_{t=0} \geq e^{-\frac{2}{5A^2}}.  
   \label{eqn:MJ17_5}
  \end{align}
  Also, the gradient estimate of $h$ at $t=0$ is explicit because of (\ref{eqn:MJ17_2}):
  \begin{align*}
    \left| \nabla \sqrt{h}\right|^2=\frac{|\nabla h|^2}{4h}=\frac{(\psi')^2}{400A^2\psi} \leq \frac{1}{40A^2}. 
  \end{align*}
  Applying (\ref{eqn:CF05_1}) in Theorem~\ref{thm:CF07_1} and setting $C_h=\frac{1}{40A^2}$, we obtain
  \begin{align}
   \boldsymbol{\mu}_0-\boldsymbol{\mu}_1  \leq \left\{ \frac{\tau_0}{10 A^2} + e^{-1}\right\} \cdot \left\{ \frac{1}{\int_{\Omega_0'} u}-1 \right\}. 
   \label{eqn:MJ17_6} 
  \end{align}
  Plugging (\ref{eqn:MJ17_5}) into the above inequality, we obtain
  \begin{align}
    \boldsymbol{\mu}_0-\boldsymbol{\mu}_1 \leq \left\{ \frac{\tau_0}{10A^2} + e^{-1}\right\} \cdot  \left( e^{\frac{2}{5A^2}} -1\right). 
  \end{align}
  Since $\tau_0=1+\tau_1 \in (1, 1+A^2)$ and $A \geq 1000m$, the right hand side of the above inequality can be bounded by  $A^{-2}$. 
  Therefore, we finish the proof of (\ref{eqn:MJ17_9}). 
\end{proof}

  The almost monotonicity inequality (\ref{eqn:MJ17_10})  implies the following no-local-collapsing Theorem.

  \begin{theorem}[\textbf{A local version of no-local-collapsing}]
  Suppose $\{(M^m, g(t)), 0 \leq t \leq T\}$ is a Ricci flow solution satisfying (\ref{eqn:MJ17_11}) for some $A \geq 1000m$. 
  Then for each $r \in \left(0, \sqrt{T} \right]$ and geodesic ball $B_{g(T)}(x,r) \subset B_{g(T)}\left(x_0, 8A\sqrt{T} \right)$ where $R(\cdot, T) \leq r^{-2}$, 
  we have
  \begin{align}
      \frac{|B_{g(T)}(x_0, r)|}{\omega_m r^m} \geq \kappa,     \label{eqn:CA02_7}
  \end{align}
  where $\kappa$ is a positive constant and can be chosen as $e^{-2^{m+7}-2+\underset{\tau \in [T, 2T]}{\inf} \boldsymbol{\mu}(\Omega_0, g(0), \tau)}$. 
  \label{thm:CA02_4}
  \end{theorem}
  
  \begin{proof}
    It follows from (\ref{eqn:MJ17_10}) and Proposition~\ref{prn:CA01_2} that
  \begin{align*}
     \boldsymbol{\nu} \left( B_{g(T)}(x,r), g(T), r^2 \right) &\geq  \boldsymbol{\nu} \left(B_{g(T)} \left(x_0, 8A\sqrt{T} \right), g(T), r^2 \right) \\
     &\geq \inf_{\tau \in [T, T+r^2]} \boldsymbol{\mu} \left(B_{g(0)} \left(x_0, 20A\sqrt{T} \right), g(0), \tau+T\right)-A^{-2}\\
     &\geq \inf_{\tau \in [T, 2T]} \boldsymbol{\mu} \left(B_{g(0)} \left(x_0, 20A\sqrt{T} \right), g(0), \tau \right)-1.
  \end{align*}
  Note that $R(\cdot, T) \leq r^{-2}$ in the ball $B_{g(T)}(x,r)$.  Therefore, (\ref{eqn:CA02_7}) and the choice of $\kappa$ follows from (\ref{eqn:GC28_2}) in Theorem~\ref{thm:CF21_3}. 
  \end{proof}

\begin{remark}
 Note that the constant $\kappa$ in (\ref{eqn:CA02_7}) does not depend on $A$ whenever $A \geq 1000m$. 
 Therefore,  Perelman's no-local-collapsing theorem, i.e., Theorem 4.1 of~\cite{Pe1},  is implied by Theorem~\ref{thm:CA02_4} by choosing $A$ large enough. 
 In the paper~\cite{Pe1}, to obtain the volume lower bound, one need the Riemannian curvature bound in the ball $B_{g(T)}(x,r)$. 
 This condition of Riemannian curvature bound is then replaced by scalar curvature bound by Perelman later, using a contradiction argument(c.f. Remark 13.13 of Kleiner-Lott~\cite{KL}).  
 Theorem~\ref{thm:CA02_4} has the advantage that the lower bound of volume ratio is calculated explicitly, due to estimate (\ref{eqn:GC28_2}) in Theorem~\ref{thm:CF21_3}. 
\label{rmk:MJ16_2} 
\end{remark}

  In Theorem~\ref{thm:CA02_4},  $\kappa$ depends only on the initial condition $\displaystyle \inf_{\tau \in [T, 2T]} \boldsymbol{\mu} \left(\Omega_0, g(0), \tau \right)$, which is totally determined
  by the local geometry of $\Omega_0=B_{g(0)} \left(x_0, 20A\sqrt{T} \right)$.
  As $A$ is a big constant, the geodesic ball $\Omega_0$ seems to be large. 
  Therefore, the requirement of  $\displaystyle \inf_{\tau \in [T, 2T]} \boldsymbol{\mu} \left(\Omega_0, g(0), \tau \right)$ at the initial time seems to be strong. 
  Such condition can be replaced by a weaker one(c.f. second part of (\ref{eqn:ML27_3})) in Theorem~\ref{thmin:ML14_1}, with the help of stronger conditions in space-time(c.f.  first part of (\ref{eqn:ML27_3})).

\section{General reduced distance and volume density function}
\label{sec:reduced}

In this section, we shall study the reduced distance and reduced volume density function, starting from a probability measure at some fixed time slice $T$.
Such concepts are natural generalization of the ones(c.f. Definition~\ref{dfn:CF21_1} and Remark~\ref{rmk:CE26_1}) defined by Perelman in~\cite{Pe1}. 

We first recall the concept and basic properties of the reduced distance and reduced volume of Perelman~\cite{Pe1}. 
Suppose $\left\{ (M^m, g(t)), 0 \leq t \leq T \right\}$ is a Ricci flow solution.  Fix $y$ as a base point. Let $\boldsymbol{\gamma}$ be a spacetime curve parametrized by $\tau=T-t$ satisfying 
\begin{align*}
\boldsymbol{\gamma}(\tau)=(\gamma(\tau), T-\tau) \in M \times [0, T], \quad \tau \in [0, \bar{\tau}]. 
\end{align*}
For each such $\boldsymbol{\gamma}$, we can calculate the Lagrangian $\mathcal{L}$ by
\begin{align*}
   \mathcal{L}(\boldsymbol{\gamma})=\int_{0}^{\bar{\tau}} \sqrt{\tau} \left( R + |\dot{\gamma}|^2\right)_{g(T-\tau)} d\tau. 
\end{align*}
Note that the above definition coincides with (\ref{eqn:CF06_3}) by letting $\tau_T \to 0$. 
Fix $(x, T-\bar{\tau})$. Among all such $\boldsymbol{\gamma}$ connecting $(y,T)$ and $(x, T-\bar{\tau})$, 
there exists at least one curve $\boldsymbol{\alpha}$ such that $\mathcal{L}$ is minimized.  Such $\boldsymbol{\alpha}$ is called a shortest reduced geodesic.
The reduced distance between $(y,T)$ and $(x, T-\bar{\tau})$ is defined as 
\begin{align}
  l((y,T), (x, T-\bar{\tau}))=\frac{1}{2\sqrt{\bar{\tau}}} \mathcal{L}(\boldsymbol{\alpha}).   \label{eqn:CE29_1}
\end{align}
Fix $y \in M$, then we denote
\begin{align*}
    l(x,t)=l((y,T), (x,t)).
\end{align*}
By this notation, $l$ is a function defined on the space-time $M \times [0,T)$.  It was proved by Perelman(c.f. Corollary 9.5 of~\cite{Pe1} and Theorem 2.23 of~\cite{Ye1}) that 
\begin{align}
     \square^*  \left\{ \left\{ 4\pi[T-t] \right\}^{-\frac{m}{2}} e^{-l} \right\}  \leq 0         \label{eqn:CE19_1}
\end{align}
in the distribution sense.  Similarly, we have
\begin{align}
      \square \left\{4(T-t)\left(l -\frac{m}{2} \right) \right\} \geq 0.   \label{eqn:CC19_3}
\end{align}
As pointed out by Perelman in section 7 of~\cite{Pe1},  the maximum principle implies that $R \geq -\frac{m}{2(T-t)}$ on $M \times [0, T)$. 
Then direct calculation yields that  $l+\frac{m}{2}$ is a nonnegative function on $M \times [0, T)$.
In other words, we have
\begin{align}
      l((y,T),(x,t)) + \frac{m}{2} \geq 0, \quad \forall \; x,y \in M.   \label{eqn:CC19_4}
\end{align}
Recall that 
\begin{align*}
\displaystyle \lim_{t \to T^{-}} 4(T-t)l((y,T),(\cdot, t))=d_{g(T)}^2(\cdot, y), 
\end{align*}
which is a large deviation formula dates back to Varadhan~\cite{Vara}. 
Therefore, the minimum value of $4(T-t)\left(l -\frac{m}{2} \right)$ is nearly zero as $t \to T^{-}$. 
Consequently, applying the maximum principle on (\ref{eqn:CC19_3}) implies that the minimum value of $4(T-t)\left(l -\frac{m}{2} \right)$ is always non-positive for each $t \in [0,T)$.     
Therefore, we have
\begin{align}
  0 \leq   \min_{x \in M}   \left\{ l((y,T),(x,t)) + \frac{m}{2} \right\}  \leq m.   \label{eqn:CE21_2}
\end{align}

We now generalize the concept of reduced distance and reduced volume density. 

\begin{definition}
  Suppose $p dv_{g(T)}$ is a continuous probability measure of the Riemannian manifold $(M, g(T))$, i.e., $p$ is a nonnegative continuous function satisfying $\int_{M} p(y) dv_{g(T)}(y)=1$. 
  For each $(x,t) \in M \times [0, T)$, we define
  \begin{align*}
     &l(x,t) \coloneqq \int_{M} p(y) l((y,T), (x,t)) dv_{g(T)}(y), \\
     &w(x,t) \coloneqq \int_{M} p(y) \left\{ 4\pi[T-t] \right\}^{-\frac{m}{2}} e^{-l((y,T), (x,t))} dv_{g(T)}(y).
  \end{align*}
  We call $l$ as the reduced distance, $w$ as the reduced volume density function, with respect to the probability measure $p dv_{g(T)}$ at time $T$. 
\label{dfn:CF21_1}  
\end{definition}

 \begin{remark}
 If we choose the base probability measure $p$ as the Dirac measure at time $t=T$, our reduced distance function and reduced volume density function coincide with the ones defined by
 Perelman in~\cite{Pe1}.  For simplicity of notations, we shall also use $l$ and $w$ for our reduced distance  and reduced volume density with respect to probability measures.
 However, it will be clear what is the base measure $p$ in the context.
 \label{rmk:CE26_1}
 \end{remark}

 \begin{proposition}[\textbf{Subsolution and supersolution}]
   Suppose $l$ is the reduced distance function, $w$ is the reduced volume density function, with respect to a probability measure $u_T dv_{g(T)}$ at time $T$. 
   Then we have 
   \begin{align}
      &\square \left\{ 4(T-t)\left( l-\frac{m}{2}\right) \right\} \geq 0,     \label{eqn:CE19_2}\\
      &\square^* w \leq 0    \label{eqn:CC14_4}
   \end{align}
   in the distribution sense.    In particular, on the space-time $M \times [0, T)$, we have
   \begin{align}
      & l + \frac{m}{2} \geq 0, \label{eqn:CE21_1}\\
      & w-u \leq 0,   \label{eqn:CC14_3}
   \end{align}
   where $u$ is the conjugate heat solution $\square^* u=0$ satisfying the initial condition $u=p$ at time $t=T$. 
 \label{prn:CC23_2}  
 \end{proposition}
 
 \begin{proof}
Actually, (\ref{eqn:CE19_2}) and (\ref{eqn:CC14_4}) follow from (\ref{eqn:CC19_3}) and (\ref{eqn:CE19_1}) respectively, due to the fact that $\square$ and $\square^*$ are linear operators.  
(\ref{eqn:CE21_1}) follows from (\ref{eqn:CC19_4}), by integration with respect to the probability measure $p(y)dv_{g(T)}(y)$. 
Note that $w-u$ is a continuous function starting from $0$ at time $t=T$, and it satisfies $\square^*(w-u) \leq 0$. Therefore, (\ref{eqn:CC14_3}) follows from the standard maximum principle for parabolic sub-solutions.  
 \end{proof}

\begin{lemma}[\textbf{Construction of ``cutoff" function}]
 For each $A >1000m$,   there exists a non-decreasing positive $C^2$-function 
 $$\psi: (-\infty, 0.2) \to [1, \infty)$$
 such that $\psi \equiv 1$ on $(-\infty, 0.1)$ and increase to $\infty$ on $(0.1, 0.2)$. Moreover, $\psi$ satisfies
  \begin{align}
    2\frac{\left( \psi' \right)^2}{\psi} -\psi'' - 4A \psi' \geq - F_0(A) \psi,   \label{eqn:MJ12_7}
  \end{align}
  where $F_0(A)$ is a constant depending only on $A$ and can be chosen as $1000A^2$.  
\label{lma:CE30_1}  
\end{lemma}

\begin{proof}
Solving the second order ODE 
$$y''-2\frac{(y')^2}{y}+4Ay'=0$$
with the condition that $y$ blows up at $s=0.2$, it is not hard to see the general solution is $\frac{C}{e^{4A(0.2-s)}-1}$. 
For simplicity, we choose $C=2(e^4-1)$ and denote $\frac{2(e^4-1)}{e^{4A(0.2-s)}-1}$ by $y$. 
At $s=s_1=0.2-\frac{1}{A}$, we have
\begin{align*}
  y(s_1)=2>1, \quad
  y'(s_1)=8A \cdot \frac{e^{4}}{(e^{4}-1)}, \quad
  y''(s_1)=32A^2 \cdot \frac{e^{4}(e^{4}+1)}{(e^{4}-1)^2}.
\end{align*}
For simplicity of notation, we denote
\begin{align}
   a_1 \coloneqq \frac{8e^{4}}{(e^{4}-1)} \sim 8.149, \quad a_2 \coloneqq \frac{32e^{4}(e^{4}+1)}{(e^{4}-1)^2} \sim 33.813.   \label{eqn:CE31_5}
\end{align}
We observe that the following interpolation holds.
\begin{claim}
 For every $k \in (2,A)$, there is an increasing positive function $f$ on $[0, \frac{k}{A}]$ such that
 \begin{align}
 \begin{cases}
  &f(0)=f'(0)=f''(0)=0, \\
  &f\left(\frac{k}{A} \right)=1, \quad f' \left(\frac{k}{A} \right)=a_1 A, \quad f'' \left(\frac{k}{A} \right)=a_2 A^2, 
 \end{cases} 
 \label{eqn:CE31_2}
 \end{align}
 where $a_1$ and $a_2$ are defined in (\ref{eqn:CE31_5}).  
\label{clm:CE31_1} 
\end{claim}

This is done by an elementary interpolation. On $[0, \frac{k}{A}]$, let $h$ be defined as follows
\begin{align}
   h \coloneqq A^2\left\{ c_1 \left( \frac{At}{k}\right) + c_2 \left( \frac{At}{k}\right)^2 + c_3\left( \frac{At}{k}\right)^3 \right\} .     \label{eqn:CE31_3}
\end{align}
Let $H$ be the anti-derivative of $h$ with $H(0)=0$, $\mathcal{H}$ be the anti-derivative of $H$ with $\mathcal{H}(0)=0$.  Then we have
\begin{align*}
    &H=A^2 \left\{ \frac{c_1t}{2} \left( \frac{At}{k}\right)  + \frac{c_2 t}{3} \left( \frac{At}{k}\right)^2 + \frac{c_3 t}{4} \left( \frac{At}{k}\right)^3 \right\}, \\
    &\mathcal{H}=A^2 \left\{ \frac{c_1t^2}{6} \left( \frac{At}{k}\right)  + \frac{c_2 t^2}{12} \left( \frac{At}{k}\right)^2 + \frac{c_3 t^2}{20} \left( \frac{At}{k}\right)^3 \right\}.
\end{align*}
We want to figure out $c_1, c_2, c_3$ such that $h=f''$, $H=f'$ and $\mathcal{H}=f$. For this purpose, we need 
\begin{align*}
\begin{cases}
  &1=f\left(\frac{k}{A} \right)=\mathcal{H}\left(\frac{k}{A} \right)=k^2\left\{ \frac{c_1}{6} + \frac{c_2}{12} + \frac{c_3}{20}\right\}, \\
  &a_1A =f'\left(\frac{k}{A} \right)=H\left(\frac{k}{A} \right)=Ak \left\{ \frac{c_1}{2} + \frac{c_2}{3} + \frac{c_3}{4} \right\}, \\
  &a_2 A^2 =f''\left(\frac{k}{A} \right)=h\left(\frac{k}{A} \right)= A^2\left\{ c_1+c_2+c_3 \right\}. 
\end{cases}  
\end{align*}
Solving this equation, we obtain
  \begin{align}
  \begin{cases}
    &c_1=3a_2-24 a_1k^{-1} + 60 k^{-2},\\
    &c_2=-12 a_2+84a_1 k^{-1}-180 k^{-2}, \\
    &c_3=10a_2-60a_1k^{-1} +120k^{-2}.
  \end{cases}  
  \label{eqn:CE31_4} 
  \end{align}
  Let $\theta=\frac{At}{k}$.  Then we have
  \begin{align}
     f'(t)=kA \theta^2 \left\{ \frac{c_1}{2}  + \frac{c_2}{3} \theta + \frac{c_3}{4} \theta^2  \right\}.     \label{eqn:CE31_1}
  \end{align}
  Note that $c_3>0$ and
  \begin{align*}
     -\frac{c_2}{3c_1}=\frac{12 a_2-84a_1 k^{-1}+180 k^{-2}}{3(3a_2-24 a_1k^{-1} + 60 k^{-2})}=\frac{12 -84a_1a_2^{-1} k^{-1}+180 a_2^{-1}k^{-2}}{9-72 a_1a_2^{-1}k^{-1} + 180 a_2^{-1}k^{-2}}. 
  \end{align*}
 Since $k>2$, it follows from a direct calculation that 
  \begin{align*}
    \frac{12 -84a_1a_2^{-1} k^{-1}}{9-72 a_1a_2^{-1}k^{-1} }>1.
  \end{align*}
  Consequently, we obtain that 
  \begin{align*}
    \min_{\theta \in [0,1]} \left\{ \frac{c_1}{2}  + \frac{c_2}{3} \theta + \frac{c_3}{4} \theta^2  \right\}=\min \left\{ \frac{c_1}{2},  \frac{c_1}{2}  + \frac{c_2}{3} + \frac{c_3}{4} \right\}
    =\min \left\{ \frac{c_1}{2},  \frac{a_1}{k}\right\}>0. 
  \end{align*}
  It follows from (\ref{eqn:CE31_1}) that $f'$ is a positive function on $(0, \frac{k}{A}]$.  The condition (\ref{eqn:CE31_2}) holds naturally by the construction of $f$.
  Therefore, the proof of  Claim~\ref{clm:CE31_1} is complete. \\

  Recall that $A$ is very large.    Let $k=\sqrt{A}$ and define
  \begin{align*}
     \psi(s) \coloneqq
  \begin{cases}   
     &\frac{1}{e^{4A(0.2-s)}-1},   \quad s \in [0.2-\frac{1}{A}, 0.2);\\
     &1+f(s-0.2+\frac{k+1}{A}), \quad s \in [0.2-\frac{k+1}{A}, 0.2-\frac{1}{A}];\\
     &1, \quad s \in (-\infty, 0.2-\frac{k+1}{A}].
  \end{cases}   
  \end{align*}
  Here $f$ is the positive increasing function defined in Claim~\ref{clm:CE31_1}.   Then $\psi$ obviously satisfies all the required conditions except the differential inequality (\ref{eqn:MJ12_7}),
  which will be verified in the following steps. 
  It is clear that $\psi$ is an increasing $C^2$-function.  On the interval $(-\infty, 0.2-\frac{k+1}{A}] \cup [0.2-\frac{1}{A}, 0.2]$, we have 
  $2\frac{\left( \psi' \right)^2}{\psi} -\psi'' -4A \psi' =0$ and (\ref{eqn:MJ12_7}) holds trivially. 
  Therefore, we focus our attention on the intermediate interval $[0.2-\frac{k+1}{A}, 0.2-\frac{1}{A}]$, where we have $1 \leq \psi \leq 2$.  Thus, we have
  \begin{align*}
    2\frac{\left( \psi' \right)^2}{\psi} -\psi'' -4A \psi' &\geq \frac{\left( \psi' \right)^2}{\psi} -\psi'' -4A \psi' =    \left( f' \right)^2 -f'' - 4A f'  \geq -4A^2 -f''. 
  \end{align*}
  However, it follows from the construction(c.f. (\ref{eqn:CE31_3})) of $f$ that
  \begin{align*}
     \max_{s \in [0, \frac{k}{A}]} f'' = \max_{s \in [0, \frac{k}{A}]} h(s)=\max_{\theta \in [0,1]} \left\{ c_1 \theta + c_2 \theta^2 + c_3 \theta^3 \right\} \leq |c_1|+|c_2|+|c_3|<950, 
  \end{align*}
  where we used the explicit value of $c_1, c_2$ and $c_3$ in the last step(c.f. (\ref{eqn:CE31_4}) and (\ref{eqn:CE31_5})). 
  Combining the previous two inequalities, we have
  \begin{align*}
     2\frac{\left( \psi' \right)^2}{\psi} -\psi'' -4A \psi' \geq -1000A^2 \geq -1000 A^2 \psi
  \end{align*}
  on the interval $[0.2-\frac{k+1}{A}, 0.2-\frac{1}{A}]$. 
  Therefore, (\ref{eqn:MJ12_7}) holds on all $(-\infty, 0.2)$. 
  The proof of the lemma is complete.  
\end{proof}

The existence of $\psi$ in Lemma~\ref{lma:CE30_1} was pointed out by Perelman~\cite{Pe1}.  Its construction was described in Theorem 28.2 of Kleiner-Lott~\cite{KL}. 
In Lemma~\ref{lma:CE30_1}, we provide the full details of the construction and calculate the explicit value of $F_0(A)$, which will be used in our forthcoming discussions(c.f. Remark~\ref{rmk:CF18_1}).  
For intuition,  the graph of $\psi$ is shown in Figure~\ref{fig:unboundedcutoff}.

 \begin{figure}[H]
 \begin{center}
 \psfrag{A}[c][c]{$\color{brown}{s=\frac{1}{5}}$}
 \psfrag{B}[c][c]{$y=\psi(s)$}
 \psfrag{C}[c][c]{$y$}
 \psfrag{D}[c][c]{$s$}
 \psfrag{E}[c][c]{$s=\frac{1}{20}$}
 \psfrag{F}[c][c]{$1$}
 \psfrag{G}[c][c]{$0$}
 \includegraphics[width=0.5 \columnwidth]{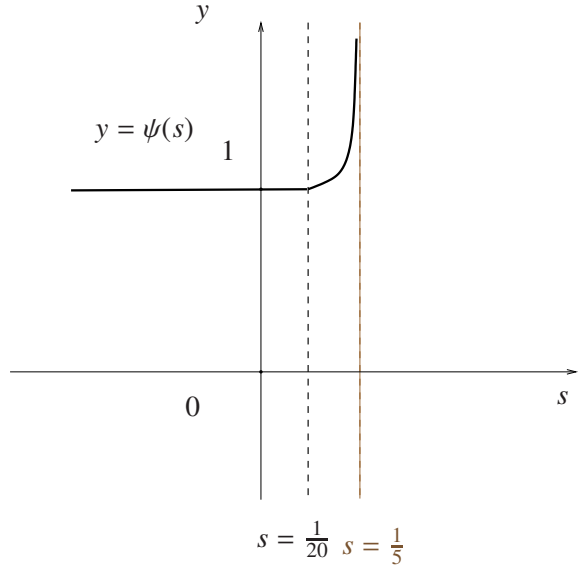}
 \caption{The choice of unbounded cutoff function $\psi$}
 \label{fig:unboundedcutoff}
 \end{center}
 \end{figure}

The following property is inspired by the proof of Theorem 8.2 of Perelman~\cite{Pe2}. 

\begin{proposition}[\textbf{Harnack type Estimate of positive super solution}]
Suppose $\{(M^m, g(t)), 0 \leq t \leq T\}$ is a Ricci flow solution satisfying
\begin{align}
    t \cdot Rc(x,t) \leq (m-1)A, \quad \forall \; x \in B_{g(t)} \left(x_0, \sqrt{t} \right), \; t \in (0, T].   \label{eqn:CC21_1}
\end{align}
Suppose $\square H \geq 0$ and $H>0$  on $M \times [t_0, T]$ for some $t_0 \in [0.5T, T]$.
Then we have
\begin{align}
    \min_{z \in \bar{B}_{g(t_0)} \left(x_0, 0.2\sqrt{T} \right)} H(z,t) \leq e^{F_0(1-\frac{t_0}{T})}  \min_{z \in \bar{B}_{g(1)} \left(x_0, A\sqrt{T} \right)} H(z,1),   \label{eqn:CE24_1}
\end{align}
where $F_0=F_0(A)$ is a large positive number satisfying (\ref{eqn:MJ12_7}). 
\label{prn:CE24_1}
\end{proposition}

\begin{proof}
Let $\psi$ be the cutoff function defined in Lemma~\ref{lma:CE30_1} and define auxiliary functions 
 \begin{align}
   &\tilde{d}(x,t) \coloneqq d(x,t)-\frac{A(t-t_0)}{1-t_0}, \label{eqn:CC23_3}\\
   & h \coloneqq H \psi\left( \tilde{d}(x,t) \right).  \label{eqn:CC23_2}
 \end{align}
 By the positivity assumption of $H$,  it is clear that $h(\cdot, t)$ is positive and finite in $B_{g(t)}\left(x_0, \frac{A(t-t_0)}{1-t_0}+0.2 \right)$ for each $t \in [t_0, 1]$. 
 Moreover, $h(\cdot, t)=H(\cdot, t)$ in $B_{g(t)}\left(x_0, \frac{A(t-t_0)}{1-t_0}+0.05 \right)$. 
 Also,  $h(x,t) \to \infty$ whenever $x \to \partial B_{g(t)}\left(x_0, \frac{A(t-t_0)}{1-t_0}+0.2 \right)$. 
 In particular, at time $t=1$, we have $h(\cdot, 1)=H(\cdot, 1)$ in $B_{g(1)}(x_0, A)$. 
 The different domains are illustrated in Figure~\ref{fig:conjugateheat}.

 \begin{figure}[H]
 \begin{center}
 \psfrag{A}[c][c]{$M$}
 \psfrag{B}[c][c]{$t=1$}
 \psfrag{BB}[c][c]{$t=t_0$}
 \psfrag{BBB}[c][c]{$t=0$}
 \psfrag{C}[c][c]{$t$}
 \psfrag{D}[c][c]{$\color{blue}{\partial B_{g(t)}\left(x_0, \frac{A(t-t_0)}{1-t_0}+0.05 \right)}$}
 \psfrag{E}[c][c]{$\color{green}{\partial B_{g(t)}\left(x_0, \sqrt{t} \right)}$}
 \psfrag{F}[c][c]{$\color{blue}{B_{g(0)}(x_0,1) \times [0,1]}$}
 \psfrag{G}[c][c]{$\color{red}{B_{g(1)}(x,r)}$}
 \psfrag{H}[c][c]{$x_0$}
 \psfrag{I}[c][c]{$\color{brown}{\partial B_{g(t)}\left(x_0, \frac{A(t-t_0)}{1-t_0}+0.2 \right)}$}
 \includegraphics[width=0.6 \columnwidth]{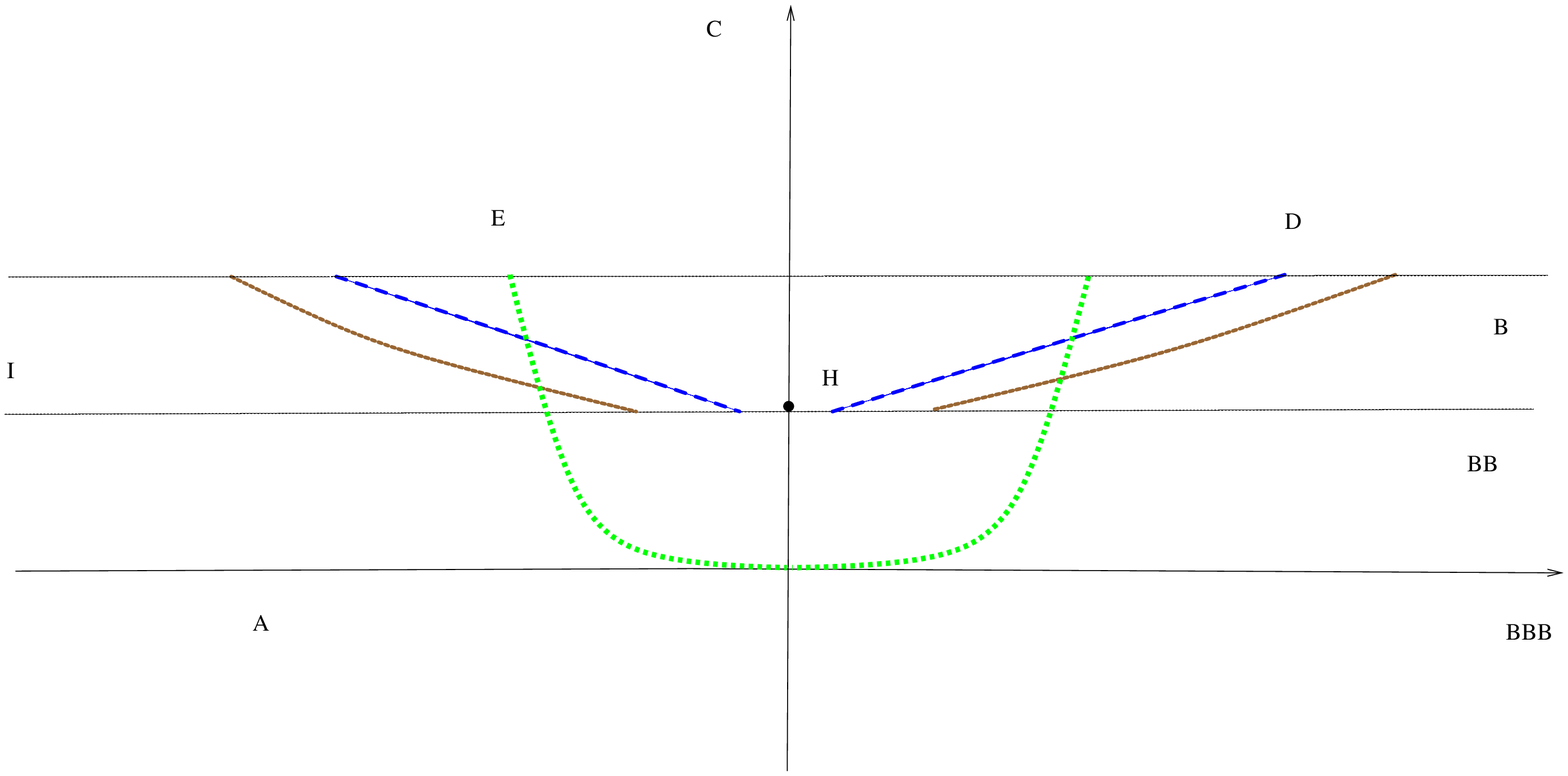}
 \caption{Different domains for applying conjugate heat equations}
 \label{fig:conjugateheat}
 \end{center}
 \end{figure}
 
Note that
\begin{align*}
  \square h=H \square \psi + \psi \square H
    -2 \left\langle \nabla \psi, \nabla H \right\rangle, 
  \quad
  \nabla h=H \nabla \psi + \psi \nabla H.
\end{align*}
At space minimum points of $h$, we have $\nabla h=0$, which implies that $\nabla H=-\psi^{-1}H \nabla \psi$. 
Consequently, by applying the condition $\square H \geq 0$,  on such minimum points, we have
\begin{align}
  \square h=H \left( \square \psi +\frac{2}{\psi} |\nabla \psi|^2 \right) + \psi \square H \geq H \left( \square \psi +\frac{2}{\psi} |\nabla \psi|^2 \right) .
\label{eqn:MJ13_4}  
\end{align}  
However, it is clear that
\begin{align}
 \square \psi=\psi' \left(\square d-\frac{A}{1-t_0}\right) -\psi''. 
 \label{eqn:MJ13_5}
\end{align}
By our assumption (\ref{eqn:CC21_1}), it follows from Lemma~\ref{lma:CC22_1} and the fact $t_0 \in [0.5, 1]$ that
\begin{align}
    \square d-\frac{A}{1-t_0} \geq -\frac{A}{\sqrt{t}}-\frac{A}{1-t_0} \geq -4A, \quad \forall\;   t \in [t_0, 1].   \label{eqn:MJ13_6}
\end{align}
Combining (\ref{eqn:MJ13_4}), (\ref{eqn:MJ13_5}) and (\ref{eqn:MJ13_6}), we obtain  
\begin{align*}  
 \square h &\geq H \left( \psi' \left(\square d -\frac{A}{1-t_0}\right) + \frac{2(\psi')^2}{\psi} -\psi'' \right)
 \geq H \left( -4A\psi' + \frac{2(\psi')^2}{\psi} -\psi'' \right).
 \end{align*}
Plugging (\ref{eqn:MJ12_7}) into the above inequality, at each space minimum point of $h(\cdot, t)$, we have
\begin{align*}
  \square h \geq -F_0 H  \psi \geq -F_0 h, \quad \Rightarrow \quad \square \left( e^{F_0 t} h \right) \geq 0.
\end{align*}
Hence, it follows from the maximum principle that 
\begin{align}
  h_{min}(t) \leq e^{F_0(1-t)} h_{min}(1), \quad \forall \; t \in [t_0,1].   \label{eqn:MJ12_1}
\end{align}
Suppose $h_{min}(t)$ is achieved at point $z_t$.   In other words, $h_{min}(t)=h(z_t,t)$. 
Clearly,  $h(z_t,t) >0$ and $z_t \in B_{g(t)}\left(x_0, \frac{A(t-t_0)}{1-t_0} +0.2 \right)$. 
Using the fact that $\psi \geq 1$ and inequality (\ref{eqn:MJ12_1}), then we have
\begin{align}
  H(z_t,t)=\frac{h(z_t,t)}{\psi(d(z,t))}=\frac{h_{min}(t)}{\psi(d(z,t))} \leq h_{min}(t) \leq h_{min}(1) e^{F_0(1-t)}.   \label{eqn:MJ12_3}
\end{align}
Note that
\begin{align*}
  &\min_{z \in \bar{B}_{g(t_0)}(x_0, 0.2)} H(z,t_0) \leq H(z_{t_0}, t_0), \\
  &h_{min}(1) \leq  \min_{z \in \bar{B}_{g(1)}(x_0, A)} h(z,1) =\min_{z \in \bar{B}_{g(1)}(x_0, A)} H(z,1).
\end{align*}
Combining the previous inequalities, we obtain
\begin{align*}
   \min_{z \in \bar{B}_{g(t_0)}(x_0, 0.2)} H(z,t_0) \leq e^{F_0(1-t_0)}  \min_{z \in \bar{B}_{g(1)}(x_0, A)} H(z,1),
\end{align*}
which is equivalent to (\ref{eqn:CE24_1}) after rescaling $1$ to $T$. 
\end{proof}

\begin{proposition}[\textbf{Estimate of reduced distance}]
Let $\{(M^m, g(t)), 0 \leq t \leq T\}$ be a Ricci flow solution satisfying
\begin{align}
    t \cdot |Rc|(x,t) \leq (m-1)A, \quad \forall \; x \in B_{g(t)} \left(x_0, \sqrt{t} \right), \; t \in (0, T].   \label{eqn:CE26_1}
\end{align}
Then for each point $y \in B_{g(T)}\left(x_0, A\sqrt{T} \right)$ and $x \in B_{g(0.5T)}(x_0, 0.1)$,  we have
\begin{align}
  l((y,T), (x, 0.5T)) \leq 4m e^{0.5 F_0}        \label{eqn:CC21_12}
\end{align}
where $F_0=F_0(A)$ is a large positive number satisfying (\ref{eqn:MJ12_7}). 
\label{prn:CC21_1}
\end{proposition}

\begin{proof}
 In light of the scaling invariance of the reduced distance,  we may assume that $T=1$ without loss of generality. 
 Fix $y \in B_{g(1)}(x_0, A)$, $s \in [0.5, 0.6]$ and let $l(x,t)=l((y,1),(x,t))$.  We define an auxiliary function
 \begin{align}
     \tilde{L} \coloneqq 4(1-t)\left(l-\frac{m}{2} \right)+2m+1.    \label{eqn:CE25_3}
 \end{align}
 It follows directly from (\ref{eqn:CC19_3}) that 
 \begin{align}
    \square \tilde{L}  \geq 0.   \label{eqn:CC20_2} 
 \end{align}
 On the other hand, one can apply (\ref{eqn:CC19_4}) to obtain
 \begin{align}
   \tilde{L}=4(1-t)\left(l+\frac{m}{2} \right)+2m+1-4m(1-t) \geq 4mt-2m+1 \geq 1    \label{eqn:CE25_1}
 \end{align}
 whenever $t \in [0.5, 1]$.  In light of (\ref{eqn:CC20_2}) and (\ref{eqn:CE25_1}), we can apply Proposition~\ref{prn:CE24_1} for $\tilde{L}$.
 Then (\ref{eqn:CE24_1}) reads as
 \begin{align}
    \min_{z \in \bar{B}_{g(s)} \left(x_0, 0.2 \right)} \tilde{L}(z,t) \leq e^{F_0(1-s)}  \min_{z \in \bar{B}_{g(1)} \left(x_0, A\right)} \tilde{L}(z,1) \leq e^{0.5 F_0} \min_{z \in \bar{B}_{g(1)} \left(x_0, A\right)} \tilde{L}(z,1).  
\label{eqn:CE25_2}    
\end{align}
However, for each $z \in M$, we have
\begin{align*}
   \tilde{L}(z,1)=\lim_{t \to 1^{-}} \tilde{L}(z,t)=2m+1+d_{g(1)}^2(z, y).
\end{align*}
Since $y \in B_{g(1)} \left(x_0, A\right)$, it is clear that $\displaystyle  \min_{z \in B_{g(1)} \left(x_0, A\right)} \tilde{L}(z,1)=\tilde{L}(y,1)=2m+1$.  Threfore, it follows from (\ref{eqn:CE25_2}) that
\begin{align*}
    \min_{z \in \bar{B}_{g(s)} \left(x_0, 0.2 \right)} \tilde{L}(z,t) \leq (2m+1) e^{0.5 F_0}. 
\end{align*} 
Plugging the above inequality into the defining equation (\ref{eqn:CE25_3}), we arrive at
\begin{align}
  \min_{z \in \bar{B}_{g(s)} \left(x_0, 0.2 \right)} l(z,s) \leq \frac{m}{2} + \frac{2m+1}{4(1-s)} \left( e^{0.5F_0}-1\right) < 2m e^{0.5 F_0},   \label{eqn:CE25_4}
\end{align}
as $s \in [0.5, 0.6]$.   Let $z_0$ be a point such that
\begin{align}
   l(z_0,s)=\min_{z \in \bar{B}_{g(s)} \left(x_0, 0.2 \right)} l(z,s).        \label{eqn:CE25_5}
\end{align}

\begin{claim}
For each $t \in \left[s, s+\epsilon_A\right]$, we have
\begin{align}
 B_{g(s)}(x_0, 0.1) \subset B_{g(s+\epsilon_A)}(x_0, 0.2) \subset B_{g(t)}\left(x_0, \sqrt{t}\right).  \label{eqn:CC21_3}
\end{align}
where $\epsilon_A$ is a small constant defined by
 \begin{align}
    \epsilon_A  \coloneqq \frac{1}{1000mA}.      \label{eqn:CC23_1}
 \end{align}
\label{clm:CC21_1} 
\end{claim}

The proof of Claim~\ref{clm:CC21_1} follows from a standard application of the maximum principle of ball containment.  
For simplicity, we shall only prove the second inequality of (\ref{eqn:CC21_3}). The proof of the first inequality is almost the same and will be left to interested readers. 
Actually, since $0.2<\sqrt{s+\epsilon_A}$, it is clear that the second inequality of (\ref{eqn:CC21_3}) holds strictly whenever $t=s+\epsilon_A$.
Suppose it fails for some $t \in \left[s, s+\epsilon_A\right]$.
Then we can assume $t'$ to be the smallest $t$ such that (\ref{eqn:CC21_3}) starts to fail. 
In other words,  the second inequality of (\ref{eqn:CC21_3}) holds on $[t', s+\epsilon_A]$ and 
\begin{align}
  \partial B_{g(s+\epsilon_A)}(x_0, 0.2) \cap \partial B_{g(t')}\left(x_0, \sqrt{t'}\right) \neq \emptyset.  \label{eqn:CC21_5}
\end{align}
Therefore, we can find a point $q$ in the above intersection set and a shortest geodesic $\sigma$ connecting $q$ to $x_0$, with respect to metric $g(s+\epsilon_A)$. 
According to the choice of $\sigma$, it is clear that
\begin{align}
  |\sigma|_{g(s+\epsilon_A)}=0.2.   \label{eqn:CC21_4}
\end{align}
Also, as $\sigma$ is contained in the ball $B_{g(t)}\left(x_0, \sqrt{t}\right)$ for each $t \in [t', s+\epsilon_A]$, we can apply our assumption (\ref{eqn:CC21_1}) to obtain
\begin{align*}
 \frac{|\sigma|_{g(s+\epsilon_A)}}{|\sigma|_{g(t')} } \geq e^{-\int_{t'}^{s+\epsilon_A} \frac{(m-1)A}{t}dt} \geq \left\{ \frac{t'}{s+\epsilon_A}\right\}^{(m-1)A} \geq \left\{ 1-\frac{1}{500mA}\right\}^{(m-1)A} \sim e^{-\frac{m-1}{500m}}>0.5,
\end{align*}
where we used the explicit value of $\epsilon_A$ in (\ref{eqn:CC23_1}). 
Combining (\ref{eqn:CC21_4}) with the above inequality, we obtain that
\begin{align}
  |\sigma|_{g(t')}  < 2 |\sigma|_{g(s+\epsilon_A)} <0.5.    \label{eqn:CC21_6}
\end{align}
However, in view of the facts that  $q \in \partial B_{g(t')}\left(x_0, \sqrt{t'}\right)$ and that $\sigma$ connects $q$ and $x_0$, we have
\begin{align*}
  |\sigma|_{g(t')} \geq d_{g(t')}(x_0, q)=\sqrt{t'} >\sqrt{0.5}>0.5,
\end{align*}
which contradicts (\ref{eqn:CC21_6}).  This contradiction establishes the proof of the second inequality of (\ref{eqn:CC21_3}). 
The first inequality of (\ref{eqn:CC21_3}) can be proved similarly and is left to the interested readers.\\

Combining (\ref{eqn:CE25_4}) and (\ref{eqn:CC21_3}), we have
\begin{align}
\begin{cases}
 &\min_{z \in \bar{B}_{g(0.5+\epsilon_A)} \left(x_0, 0.2 \right)} l(z,s) < 2m e^{0.5 F_0}, \\
 &B_{g(0.5)}(x_0, 0.1) \subset B_{g(0.5+\epsilon_A)}(x_0, 0.2) \subset B_{g(t)}\left(x_0, \sqrt{t}\right), \quad \forall \; t \in [0.5, 0.5+\epsilon_A]. 
 \end{cases}
\label{eqn:CE26_2} 
\end{align}
Fix an arbitrary point $x \in B_{g(0.5)}(x_0, 0.1)$. 
Based on Claim~\ref{clm:CC21_1}, we shall construct a space-time curve $\boldsymbol{\beta}$ connecting $(z_0, 0.5+\epsilon_A)$ and $(x, 0.5)$.
Its construction is described in details as follows. 
With respect to the metric $g(0.5+\epsilon_A)$, we can find a piecewise smooth geodesic $\beta$ connecting $z_0$ and $x$. 
For example, $\beta$ can be chosen as the concatenation of the shortest geodesic connecting $z_0$ to $x_0$ and $x_0$ to $x$. 
The length of $\beta$ is less than $0.4$. 
Furthermore, we can parametrize $\beta$ by  $\tau=1-t \in [0.5-\epsilon_A, 0.5]$ such that
\begin{align*}
   \beta(0.5-\epsilon_A)=z_0, \quad \beta\left( 0.5 \right)=x. 
\end{align*}
Clearly, with respect to $g(0.5+\epsilon_A)$, we have
\begin{align}
  |\beta'|_{g(0.5+\epsilon_A)} \equiv \frac{d_{g(0.5+\epsilon_A)}(z_0,x)}{\epsilon_A}=1000mA d_{g(0.5+\epsilon_A)}(z_0,x)<400 mA,   \label{eqn:CC21_7}
\end{align}
where we used the fact that the length of $\beta$ is bounded by $0.4$ in the last step. 
Since $\beta \subset B_{g(0.5+\epsilon_A)}(x_0, 0.2)$, it follows from Claim~\ref{clm:CC21_1} that a Ricci upper bound holds for each $\beta(\tau)$ whenever $\tau \in [0.5-\epsilon_A, 0.5]$. 
Consequently, as $A$ is large, it follows from (\ref{eqn:CC21_7}) that
\begin{align}
   |\beta'|_{g(1-\tau)} \leq e^{\int_{1-\tau}^{0.5+\epsilon_A} \frac{(m-1)A}{t}dt} |\beta'|_{g(0.5+\epsilon_A)}<2|\beta'|_{g(0.5+\epsilon_A)}<1000mA.    \label{eqn:CC21_80}
\end{align}
Hence, by (\ref{eqn:CC21_80}) and the Ricci upper bound guaranteed by Claim~\ref{clm:CC21_1},  we obtain
\begin{align}
  \int_{0.5-\epsilon_A}^{0.5} \sqrt{\tau} \left( R + |\beta'|^2 \right) d\tau &<  \int_{0.5-\epsilon_A}^{0.5} \left( 4m(m-1)A + |\beta'|^2 \right) d\tau <2000mA.  \label{eqn:CC21_8}
\end{align}
In conclusion, we have a space-time curve $\boldsymbol{\gamma}(\tau)=(\gamma(\tau), 1-\tau)$ connecting $(y,1)$ to $(x,0.5)$, parametrized by $\tau=1-t$ such that
\begin{align*}
   \gamma(\tau)=
\begin{cases}
 \alpha(\tau),     & \textrm{if} \; \tau \in [0, 0.5-\epsilon_A];\\
 \beta(\tau), &\textrm{if} \; \tau \in [0.5-\epsilon_A, 0.5].
\end{cases}   
\end{align*}
Here $\alpha$ is the space projection of a reduced geodesic $\boldsymbol{\alpha}$ connecting $(y,1)$ and $(z_0, 0.5+\epsilon_A)$.  
It follows from (\ref{eqn:CE25_4}) and (\ref{eqn:CE25_5}) that
\begin{align}
   \mathcal{L}(\boldsymbol{\alpha})&=4 \cdot (0.5-\epsilon_A) \cdot l((y,1),(z_0,0.5+\epsilon_A)) \notag\\
   &\leq 2 \max\{ l((y,1),(z_0,0.5+\epsilon_A)), 0\} \leq 4m e^{0.5 F_0}.   \label{eqn:CC21_9}
\end{align}
The inequality (\ref{eqn:CC21_8}) can be understood as
\begin{align}
   \mathcal{L}(\boldsymbol{\beta})<2000mA < 2A^2<F_0.  \label{eqn:CC21_10}
\end{align}
Combining (\ref{eqn:CC21_9}) and (\ref{eqn:CC21_10}), we obtain
\begin{align*}
   \mathcal{L}(\boldsymbol{\gamma})=\mathcal{L}(\boldsymbol{\alpha})+\mathcal{L}(\boldsymbol{\beta})< 4m e^{0.5 F_0} + F_0.
\end{align*}
As $\bar{\tau}=0.5$, we have $2\sqrt{\bar{\tau}}=\sqrt{2}$.
It follows from the definition(c.f. (\ref{eqn:CE29_1})) of reduced distance $l$ and the above inequality that
\begin{align*}
 l < \frac{ 4m e^{0.5 F_0} + F_0}{\sqrt{2}}< 4me^{0.5 F_0}, 
\end{align*}
whence the inequality (\ref{eqn:CC21_12}) holds. 
\end{proof}

\begin{theorem}[\textbf{Lower bound of reduced volume density and conjugate heat solution}]
Let $\{(M^m, g(t)), 0 \leq t \leq T\}$ be a Ricci flow solution satisfying (\ref{eqn:CE26_1}).
Suppose $\varphi_T$ is a nonnegative function supported on $B_{g(T)}\left(x_0, A\sqrt{T} \right)$ satisfying $\int_{M} \varphi_T^2 dv_{g(T)}=1$. 
Let $w$ be the reduced volume density function with respect to the probability measure $\varphi_T^2 dv_{g(T)}$. 
Let $u$ be the conjugate heat solution starting from $\varphi_T^2$. 
In other words, $\square^*u=(\partial_{\tau}-\Delta +R)u=0$ on $M \times [0,T]$ and $u(\cdot, T)=\varphi_T^2$.  
Then for every $x \in B_{g(0.5)}(x_0, 0.2)$,  we have
\begin{align}
   u(x, 0.5T) \geq  w(x, 0.5T)   \geq  (2\pi )^{-\frac{m}{2}}  e^{-4me^{0.5 F_0}} T^{-\frac{m}{2}}   \label{eqn:CC23_5}     
\end{align}
where $F_0=F_0(A)$ is a large positive number satisfying (\ref{eqn:MJ12_7}). 
\label{thm:CF21_1}
\end{theorem}

\begin{proof}
The first inequality of (\ref{eqn:CC23_5}) follows directly from (\ref{eqn:CC14_3}).  
We focus on the proof of the second inequality of (\ref{eqn:CC23_5}). By (\ref{eqn:CC21_12}), direct calculation implies that
\begin{align*}
   w(x_0, 0.5T)=\int_{M} (2\pi T)^{-\frac{m}{2}} e^{-l((y,T),(x_0, 0.5T))} \varphi_T^2(y) dv_{g(T)}(y)
     \geq (2\pi T)^{-\frac{m}{2}} e^{-4me^{0.5F_0}} \int_M \varphi_T^2(y) dv_{g(T)}(y)
\end{align*}
which yields (\ref{eqn:CC23_5}) since $\int_M \varphi_T^2(y) dv_{g(T)}(y)=1$. 
\end{proof}

As discussed in the paragraph before Theorem~\ref{thm:CA02_3},  to relate the local functionals of different time slices, it is a key step to
estimate the uniform lower bound of $c_u$.   Now Theorem~\ref{thm:CF21_1} provides such a lower bound, by comparing $u$ with the reduced volume density function $w$ starting from
a probability measure(c.f. Definition~\ref{dfn:CF21_1}). 
Consequently, we are now ready to compare the local functionals of $\Omega$ with the local functionals of $\Omega_0$, 
even when $\Omega_0$ is very small compared to $\Omega$, in terms the notation
of effective monotonicity theorem, i.e., Theorem~\ref{thm:CF07_1}.

\section{A generalization of the no-local-collapsing theorem}
\label{sec:alter}

The purpose of this section is to generalize Section 8 of Perelman's paper~\cite{Pe1}. 
Let us first recall Theorem 8.2 of Perelman(\cite{Pe1}).

\begin{theorem}[Theorem 8.2 of Perelman~\cite{Pe1}]
  For any $A>1$ there exists $\kappa=\kappa(m,A)>0$ with the following property. 
  If $g_{ij}(t)$ is a smooth solution to the Ricci flow
  \begin{align*}
    \D{}{t} g_{ij}=-2R_{ij}, \quad   0 \leq t \leq r_0^2,
  \end{align*}
  which has $|Rm|(x,t) \leq m^{-1}r_0^{-2}$ for all $(x,t)$ satisfying $d_{g(0)}(x,x_0)<r_0$, and 
  the volume of the metric ball $B_{g(0)}(x_0, r_0)$ is at least $A^{-1} r_0^{m}$, then 
  $g_{ij}(t)$ cannot be $\kappa$-collapsed on the scales less than $r_0$
  at a point $(x,r_0^2)$ with $d_{g(r_0^2)}(x,x_0) \leq Ar_0$.
  \label{thm:Pe8.2}
\end{theorem}

We have modified the statement of Theorem 8.2 of Perelman(\cite{Pe1}) slightly following Kleiner-Lott~\cite{KL}.
One can check Theorem 27.2 and Remark 27.3 of Kleiner-Lott~\cite{KL} for more details.
Clearly, such modifications do not affect its application at all. 
Note that the ``$\kappa$-collapsed" in Theorem~\ref{thm:Pe8.2} means that 
\begin{align*}
  |B_{g(t_0)}(x_0, r)|_{g(t_0)} r^{-m}<\kappa
\end{align*}
whenever $|Rm|(x,t) \leq r^{-2}$ for all $(x,t)$ satisfying $d_{g(t_0)}(x,x_0)<r$ and $t_0-r^2 \leq t \leq t_0$. 
This seems to be a strong condition and sometimes hard to obtain, especially in the high dimensional case. 
In the following theorem, we shall show that the Riemannian curvature bound in a space-time neighborhood of $(x,t_0)$ is not needed, 
a scalar curvature bound in a space neighborhood of $(x, t_0)$ is sufficient to draw similar non-collapsing conclusion. 

Our starting point is the relationships between volume ratios and the local $\boldsymbol{\nu}$-functionals, expressed explicitly in 
Theorem~\ref{thm:CF21_3} and Theorem~\ref{thm:CF21_2}.   
Based on the effective monotonicity(c.f. Theorem~\ref{thm:CF07_1}) and the uniform lower bound of $c_u$(c.f. Theorem~\ref{thm:CF21_1}), we can first set up an  estimate of the propagation of the local $\boldsymbol{\nu}$-functionals.   Then we use the equivalent relationships to transform it to an estimate of the  propagation of the volume ratios.

\begin{theorem}[\textbf{Propagation estimate of local $\boldsymbol{\nu}$-functional}]
  Suppose $\{(M^{m}, g(t)), 0 \leq t \leq  T\}$ is a Ricci flow solution.
  Suppose $0<r<\sqrt{T}$ and $B_{g(T)}(y_0,r)$ is a geodesic ball where the scalar curvature $R(\cdot, T)<r^{-2}$. 
  Suppose there is a big constant $A>1000m$ such that
  \begin{align}
  \begin{cases}
   & t|Rc|(x,t)<(m-1)A, \quad \textrm{for each} \; x \in B_{g(t)}\left(x_0, \sqrt{t} \right), \quad   0<t<T;\\
   &B_{g(T)}(y_0,r) \subset B_{g(T)}\left(x_0, A\sqrt{T} \right). 
  \end{cases}
  \label{eqn:CF25_2}
  \end{align}
 Define
 \begin{align}
 \begin{cases}
  &\boldsymbol{\nu}_{a} \coloneqq \boldsymbol{\nu} \left(B_{g(0.5T)}\left(x_0,0.1\sqrt{T} \right), g(0.5T), 1.5T \right),\\
  &\boldsymbol{\nu}_{b} \coloneqq \boldsymbol{\nu} \left(B_{g(T)}\left(x_0, A \sqrt{T} \right), g(T), T \right),\\
  &\boldsymbol{\nu}_{c} \coloneqq \boldsymbol{\nu} \left(B_{g(T)}\left(y_0, r\right), g(T), r^2 \right).
\end{cases}
\label{eqn:CF10_4}  
\end{align}
Then the following inequality holds:
\begin{align}
\boldsymbol{\nu}_c \geq \boldsymbol{\nu}_b \geq \boldsymbol{\nu}_{a}- e^{-\boldsymbol{\nu}_{a}+G_0},      \label{eqn:CE24_2}
\end{align}
where $G_0$ is a large constant chosen as
\begin{align}
  G_0 \coloneqq 5m e^{500A^2}.    
\label{eqn:CF09_5}  
\end{align}
\label{thm:CE26_1}
\end{theorem}

\begin{proof}

The first part of (\ref{eqn:CE24_2}) follows directly from the definition (\ref{eqn:MJ16_d}) and Proposition~\ref{prn:CA01_2}. 
We focus on the proof of the second part of (\ref{eqn:CE24_2}). 

Without loss of generality, we assume $T=1$.   For simplicity of notations,  we define(c.f. Figure~\ref{fig:functionaldomain}): 
\begin{align}
\Omega_a' \coloneqq B_{g(0.5)}(x_0, 0.05), \quad
\Omega_a \coloneqq B_{g(0.5)}(x_0, 0.1), \quad
\Omega_b \coloneqq B_{g(1)}(x_0, A), \quad \Omega_c \coloneqq B_{g(1)}(y_0, r).
\label{eqn:CE26_4}
\end{align}

 \begin{figure}[H]
 \begin{center}
 \psfrag{A}[c][c]{$M$}
 \psfrag{B}[c][c]{$t=1$}
 \psfrag{BB}[c][c]{$t=0.5$}
 \psfrag{BBB}[c][c]{$t=0$}
 \psfrag{C}[c][c]{$t$}
 \psfrag{D}[c][c]{$\color{red}{\Omega_c}$}
 \psfrag{E}[c][c]{$\color{green}{\partial B_{g(t)}\left(x_0, \sqrt{t} \right)}$}
 \psfrag{F}[c][c]{$\color{blue}{B_{g(0)}(x_0,1) \times [0,1]}$}
 \psfrag{G}[c][c]{$\color{red}{B_{g(1)}(x,r)}$}
 \psfrag{H}[c][c]{$\color{magenta}{\Omega_a'}$}
 \psfrag{I}[c][c]{$\color{brown}{\Omega_a}$}
 \psfrag{J}[c][c]{$\color{blue}{\Omega_b}$}
 \psfrag{K}[c][c]{$x_0$}
 \psfrag{L}[c][c]{$y_0$}
 \includegraphics[width=0.6 \columnwidth]{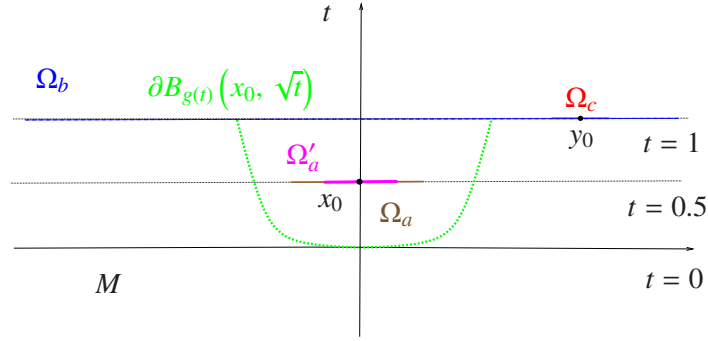}
 \caption{Different domains for comparing $\boldsymbol{\nu}$-functionals}
 \label{fig:functionaldomain}
 \end{center}
 \end{figure}

We shall apply Theorem~\ref{thm:CF07_1}, regarding $\Omega_a$ as $\Omega_0$, $\Omega_a'=\Omega_0'$ and $\Omega_b=\Omega$ respectively. 
Naturally, the default metric for the discussion with respect to $\Omega_a'$ and $\Omega_a$ is $g(0.5)$,  the default metric for $\Omega_b$ and $\Omega_c$ is $g(1)$.

Using the terminology in Theorem~\ref{thm:CF07_1}, it follows from (\ref{eqn:CC23_5}) in Theorem~\ref{thm:CF21_1}, (\ref{eqn:CF07_1}) and (\ref{eqn:CF07_6}) that
\begin{align}
 C_h \leq 1600,\quad
 c_u \geq  (2\pi)^{-\frac{m}{2}} e^{-4me^{0.5F_0}}|\Omega_a'|.  \label{eqn:CF20_5}
\end{align}
As mentioned before, the lower bound of $c_u$ is a key. 
Applying (\ref{eqn:CF04_5}) of Theorem~\ref{thm:CF07_1}, it follows from the above inequality and the fact that $\tau_{0.5} \leq 1.5$ that
\begin{align}
\boldsymbol{\nu}_a-\boldsymbol{\nu}_b \leq \frac{6400 \tau_{0.5}+1}{c_u} < \frac{10000 (2 \pi)^{\frac{m}{2}} e^{4me^{0.5F_0}}}{|\Omega_a'|}.
\label{eqn:CE03_1}  
\end{align}
Recall that(c.f. (\ref{eqn:CF25_2}) and (\ref{eqn:CE26_4})) we have $Rc \geq -2(m-1)A=\bar{\Lambda}$ in $B_{g(0.5)} \left(x_0, \sqrt{0.5} \right)$,
which contains $5\Omega_a$ and $\Omega_a'$.   Using Gromov-Bishop volume comparison, similar to the discussion around (\ref{eqn:CF19_3}), 
we can apply the fact $1<\frac{\sinh t}{t}<e^{t}$ to obtain
\begin{align}
    \frac{|\Omega_a'|}{|\Omega_a|} \geq  \frac{\int_{0}^{0.05} \left( \frac{\sinh \sqrt{2A} r}{\sqrt{2A}}\right)^{m-1}dr}{\int_{0}^{0.1} \left( \frac{\sinh \sqrt{2A} r}{\sqrt{2A}}\right)^{m-1}dr}
    \geq  \frac{m^{-1} (0.05)^m}{e^{0.1(m-1)\sqrt{2A}} \cdot m^{-1} \cdot \left( 0.1 \right)^{m}}
    =2^{-m}e^{-0.1(m-1)\sqrt{2A}}. 
\label{eqn:CF20_6}    
\end{align}
Now we combine (\ref{eqn:CE03_1}) and (\ref{eqn:CF20_6}). 
Because $F_0=1000A^2$(c.f. the choice of $F_0$ in Lemma~\ref{lma:CE30_1}) is very large, we can  absorb the extra constants and obtain that
\begin{align}
\boldsymbol{\nu}_a-\boldsymbol{\nu}_b < \frac{e^{4.2me^{0.5F_0}}}{|\Omega_a|}. 
\label{eqn:CF21_1}
\end{align}
However, in light of the scalar curvature upper bound $R \leq 2m(m-1)A$ in $\Omega_a$, we can bound $|\Omega_a|$ from below
\begin{align}
   \frac{|\Omega_a|}{\omega_m (0.1)^m} \geq e^{\boldsymbol{\nu}-2^{m+7}-2m(m-1)A \cdot (0.1)^2},
\label{eqn:CF20_7}   
\end{align}
where $\boldsymbol{\nu}=\boldsymbol{\nu}(\Omega_a, 0.01)$.   Since $0.01<1.5$, it is clear(c.f. the first inequality of (\ref{eqn:CF10_3}) in Proposition~\ref{prn:CF10_2}) that
\begin{align*}
 \boldsymbol{\nu}= \boldsymbol{\nu}(\Omega_a, 0.01) \geq \boldsymbol{\nu}(\Omega_a, 1.5)=\boldsymbol{\nu}_a.
\end{align*}
Thus, putting the above inequality into (\ref{eqn:CF20_7}) yields that
\begin{align}
   |\Omega_a| \geq e^{\boldsymbol{\nu}_a-2^{m+7}-0.02m(m-1)A-m\log 10} \omega_m \geq e^{\boldsymbol{\nu}_a-2^{m+7}-0.1m^2A} \omega_m> e^{\boldsymbol{\nu}_a-2^{m+7}-m^2A-\log \Gamma(\frac{m}{2}+1)},
\label{eqn:CF20_8}   
\end{align}
where we used the volume formula (\ref{eqn:CF19_9}).  
The combination of (\ref{eqn:CF21_1}) and (\ref{eqn:CF20_8}) yields that
\begin{align*}
   \boldsymbol{\nu}_a-\boldsymbol{\nu}_b \leq  e^{-\boldsymbol{\nu}_a}  \cdot e^{4.2me^{0.5F_0} +2^{m+7}+m^2A +\log \Gamma(\frac{m}{2}+1)}<e^{-\boldsymbol{\nu}_a+5me^{0.5F_0}}, 
\end{align*}
where we used again the fact that $F_0=1000A^2$ is very large. Plugging $F_0=1000A^2$ into the above inequality, we arrive at the second inequality of (\ref{eqn:CE24_2}).  
The proof of the Theorem is complete.  
\end{proof}

From the Ricci flow point of view, the inequality (\ref{eqn:CE24_2}) in Theorem~\ref{thm:CE26_1} is natural. 
However,  many geometers are not acquainted with the local functional $\boldsymbol{\nu}$.
Therefore, for the convenience of the readers, we also provide a \textit{volume ratio} version of inequality (\ref{eqn:CE24_2}) in the following Theorem. 
The basic idea is first to transform the volume ratio information into the local functional information in the domain with bounded Ricci curvature, by Theorem~\ref{thm:CF21_2}.
Then we apply inequality (\ref{eqn:CE24_2}) in Theorem~\ref{thm:CE26_1} to obtain an estimate of the local $\boldsymbol{\nu}$-functional in a destination domain far away.
Finally, we transform the local functional information back into volume ratio information at the destination domain, by Theorem~\ref{thm:CF21_3}. 
In this way, we obtain a propagation estimate (\ref{eqn:CF09_4}) of the volume ratio. 

\begin{theorem}[\textbf{Propagation of non-collapsing constant}]
Same conditions as in Theorem~\ref{thm:CE26_1}.  
Set
  \begin{align}
     \boldsymbol{\rho}_a \coloneqq  \frac{\left|B_{g\left(0.5\sqrt{T} \right)}\left(x_0, 0.1 \sqrt{T} \right) \right|}{\omega_m \left(0.1\sqrt{T} \right)^{m}}, 
     \quad  \boldsymbol{\rho}_c \coloneqq  \frac{|B_{g(T)}(y_0,r)|}{\omega_m r^m}. 
  \label{eqn:CF10_5}   
  \end{align}
  Then we have
  \begin{align}
   \boldsymbol{\rho}_c \geq  \boldsymbol{\rho}_a^{m+1} e^{-\frac{G_0}{\boldsymbol{\rho}_a}}
  \label{eqn:CF09_4} 
  \end{align}
  where $G_0=5m e^{500A^2}$, a large constant as defined in (\ref{eqn:CF09_5}). 
  
\label{thm:CF03_1}  
\end{theorem}

\begin{proof}
Without loss of generality, we assume $T=1$. 
It follows from (\ref{eqn:CF21_1}) that
   \begin{align}
     \boldsymbol{\nu}_a -\boldsymbol{\nu}_c \leq \boldsymbol{\nu}_a -\boldsymbol{\nu}_b \leq  \boldsymbol{\rho}_a^{-1} \cdot e^{4.5me^{0.5F_0}}  .  \label{eqn:CF09_3}
   \end{align}
   Recall that $\boldsymbol{\nu}_a=\boldsymbol{\nu}(B_{g(0.5)}(x_0, 0.1), g(0.5), 1.5)$. 
   Since $|R| \leq 2m(m-1)A$ on $\Omega_a=B_{g(0.5)}(x_0, 0.1)$, we can apply (\ref{eqn:CF10_3}) in Proposition~\ref{prn:CF10_2} so that
   \begin{align*}
    \boldsymbol{\nu}(\Omega_a, 0.01) \leq \boldsymbol{\nu}_a  + \frac{m}{2} \log 150 + 1.51 \cdot 2m(m-1)A \leq \boldsymbol{\nu}_a +4m^2A. 
   \end{align*}
   Plugging the above inequality into (\ref{eqn:CF09_3}), we arrive at
   \begin{align}
     -4m^2A + \boldsymbol{\nu}(\Omega_a, 0.01)- \frac{e^{4.5me^{0.5F_0}}}{\boldsymbol{\rho}_a} \leq \boldsymbol{v} (\Omega_c, r^2) =\boldsymbol{v}_c.
   \label{eqn:CF20_9}  
   \end{align}
   Now we need to transform all $\boldsymbol{\nu}$'s into the volume ratios.   Notice that (\ref{eqn:CF20_3}) in Theorem~\ref{thm:CF21_2} can be applied by choosing $K=\sqrt{2A}$ and $\underline{\Lambda}=2m(m-1)A$.
   Using the fact that $A$ is large, we have
   \begin{align}
     -2m^2A + (m+1)\log \boldsymbol{\rho}_a \leq  \boldsymbol{\nu}(\Omega_a, 0.01).    \label{eqn:CF20_10}
   \end{align}
   Similarly, as $R \leq r^{-2}$ in $\Omega_c$ at time $t=1$, we can apply (\ref{eqn:CF20_3}) again by choosing $\bar{\Lambda}=r^{-2}$.   So we have
   \begin{align}
    \boldsymbol{\nu}_c=\boldsymbol{\nu}(\Omega_c, r^2) \leq \log \boldsymbol{\rho}_c + \left\{ 2^{m+7}+1\right\}. 
   \label{eqn:CF20_11}
   \end{align}
   Putting (\ref{eqn:CF20_9}), (\ref{eqn:CF20_10}) and (\ref{eqn:CF20_11}) together, we obtain
   \begin{align*}
     -6m^2A +(m+1) \log  \boldsymbol{\rho}_a -  \frac{e^{4.5me^{0.5F_0}}}{\boldsymbol{\rho}_a} \leq  \log \boldsymbol{\rho}_c + 2^{m+8}.
   \end{align*}
   Absorbing the extra terms into $e^{F_0}$, we have
   \begin{align}
    \log \boldsymbol{\rho}_c \geq (m+1) \log  \boldsymbol{\rho}_a - \frac{e^{4.5me^{0.5F_0}}}{\boldsymbol{\rho}_a} - e^{F_0}
    =(m+1) \log  \boldsymbol{\rho}_a - \frac{\left\{e^{4.5me^{0.5F_0}} + e^{F_0} \boldsymbol{\rho}_a \right\}}{\boldsymbol{\rho}_a}.
   \label{eqn:CF21_2} 
   \end{align}
   However, it follows from Gromov-Bishop volume comparison that
\begin{align*}
   \boldsymbol{\rho}_a \leq 10^m \cdot m\int_{0}^{0.1} \left( \frac{\sinh \sqrt{2A}r}{\sqrt{2A}}\right)^{m-1} dr \leq e^{0.1\sqrt{2A}(m-1)}.
\end{align*}
Thus we obtain
\begin{align}
  e^{F_0} \boldsymbol{\rho}_a \leq e^{F_0 + 0.1\sqrt{2A}(m-1)} \leq e^{2F_0}<<e^{4.5m e^{0.5F_0}}. 
\label{eqn:CF21_3}   
\end{align}
Combining (\ref{eqn:CF21_2}) with (\ref{eqn:CF21_3}), noting that $2e^{4.5m e^{0.5F_0}} << e^{5me^{0.5F_0}}$, we thus arrive at
\begin{align*}
   \log \boldsymbol{\rho}_c \geq (m+1) \log  \boldsymbol{\rho}_a - \frac{e^{5me^{0.5F_0}}}{\boldsymbol{\rho}_a},
\end{align*}
which is equivalent to (\ref{eqn:CF09_4}), since $F_0=1000A^2$ as chosen in Lemma~\ref{lma:CE30_1}. 
The proof of the Theorem is complete. 
\end{proof}

\begin{remark}
 Suppose the Ricci flow in Theorem~\ref{thm:CF03_1} is Ricci flat. Then one can use standard ball containment argument and Bishop-Gromov volume comparison to obtain that 
 \begin{align}
    \boldsymbol{\rho}_c \geq \frac{\left|B_{g(T)}\left(y_0, 3A\sqrt{T} \right) \right|}{\omega_m \left(3A \sqrt{T} \right)^{m}} 
    \geq \frac{\left|B_{g(T)}\left(x_0, 0.1 \sqrt{T} \right) \right|}{\omega_m \left(3A \sqrt{T} \right)^{m}}
    =\boldsymbol{\rho}_a \cdot (30 A)^{-m}. 
 \label{eqn:CF18_1}   
 \end{align}
 Comparing the above inequality with (\ref{eqn:CF09_4}) in Theorem~\ref{thm:CF03_1}, it is clear that the above inequality is much stronger. 
 This suggests that the estimate (\ref{eqn:CF09_4}) should not be sharp. 
  It will be interesting to ask whether (\ref{eqn:CF09_4}) can be improved to an inequality similar to (\ref{eqn:CF18_1}).
  One key step for such an improvement will be the construction of cutoff function with better $F_0(A)$ in Lemma~\ref{lma:CE30_1}.
  This is partly the reason why we explicitly estimate $F_0(A)$ in details there.  
\label{rmk:CF18_1} 
\end{remark}

Now we are ready to finish the proof of Theorem~\ref{thmin:ML14_1}, which is nothing but a corollary of Theorem~\ref{thm:CF03_1}. 
For the convenience of the readers, we rewrite Theorem~\ref{thmin:ML14_1} in the following slightly more general way. 

\begin{theorem}[\textbf{Improved version of no-local-collapsing}] 
  For every $A>1$ there exists $\kappa=\kappa(m,A)>0$ with the following property. 
  Suppose $\{(M^{m}, g(t)), 0 \leq t \leq  r_0^2\}$ is a Ricci flow solution of the type
  \begin{align}
     \partial_t g=-k \left\{ -Rc + \lambda(t) g\right\}    \label{eqn:CF16_1}
  \end{align}
  where $1 \leq k \leq 2$ and $|\lambda(t)| \leq 1$.  Suppose
  \begin{align}
   r_0^2 |Rm|(x,t) \leq m^{-1}, \quad \forall \; x \in B_{g(0)}(x_0, r_0),  \; 0 \leq t \leq r_0^2;   \quad   r_0^{-m} \left| B_{g(0)}(x_0, r_0) \right|_{dv_{g(0)}} \geq A^{-1}.  \label{eqn:CE26_7}
  \end{align}
  Then we have
  \begin{align}
    r^{-m} \left|B_{g(t_0)}(x,r) \right|_{dv_{g(t_0)}} \geq \kappa      \label{eqn:CE26_8}
  \end{align}
  whenever $A^{-1}r_0^2 \leq t_0 \leq r_0^2$, $0<r \leq r_0$, and $B_{g(t_0)}(x,r) \subset B_{g(t_0)}(x_0,Ar_0)$ is a geodesic ball satisfying $r^2R(\cdot, t_0) \leq 1$. 
\label{thm:CE26_2}
\end{theorem}

Comparing Theorem~\ref{thm:CE26_2} with Theorem~\ref{thm:Pe8.2}, the conclusion of Theorem~\ref{thm:CE26_2} is much stronger. 
We obtain uniform non-collapsing estimates only assuming the scalar curvature's local upper bound in a given time-slice, rather than the Riemannian curvature bound in a space-time neighborhood. 
The statement of Theorem~\ref{thm:CE26_2} seems to be slightly more general than Theorem~\ref{thmin:ML14_1}. 
However, by a standard parabolic rescaling and the fact that both $\log k$ and $\lambda(t)$ are uniformly bounded,  it is clear that 
Theorem~\ref{thm:CE26_2} and Theorem~\ref{thmin:ML14_1} are equivalent.

We close this section by providing the proof of Theorem~\ref{thmin:ML14_1}.  The proof is basically to use Theorem~\ref{thm:CF03_1} on a very short period of Ricci flow before the time $t=t_0$
to estimate the lower bound of $r^{-m} \left|B_{g(t_0)}(x,r) \right|_{dv_{g(t_0)}}$ by the volume ratio of a small ball in the domain with bounded geometry.  
By Bishop-Gromov volume comparison, a lower bound of the volume ratio of the small ball can be obtained by the non-collapsing condition of the relatively bigger ball $B_{g(0)}(x_0, r_0)$.
The full details are given below.

\begin{figure}[H]
 \begin{center}
 \psfrag{A}[c][c]{$M$}
 \psfrag{B0}[c][c]{$0$}
 \psfrag{B1}[l][l]{$t_0-A^{-1}$}
 \psfrag{B2}[c][c]{$t_a$}
 \psfrag{B3}[c][c]{$t_0$}
 \psfrag{B4}[c][c]{$1$}
 \psfrag{C}[c][c]{$t$}
 \psfrag{D}[c][c]{$\color{red}{B_{g(t_0)}(x,r)}$}
 \psfrag{E}[c][c]{$\color{blue}{\partial \left\{B_{g(0)}(x_0,1) \times [0,1] \right\}}$}
 \psfrag{F}[c][c]{$\color{green}{\partial B_{g(t)}\left(x_0, \sqrt{t-t_0+A^{-1}} \right)}$}
  \psfrag{G}[c][c]{$\color{red}{B_{g(1)}(x,r)}$}
 \psfrag{I}[c][c]{$\color{brown}{\Omega_a}$}
 \psfrag{K}[c][c]{$x_0$}
 \psfrag{L}[c][c]{$x$}
 \includegraphics[width=0.8 \columnwidth]{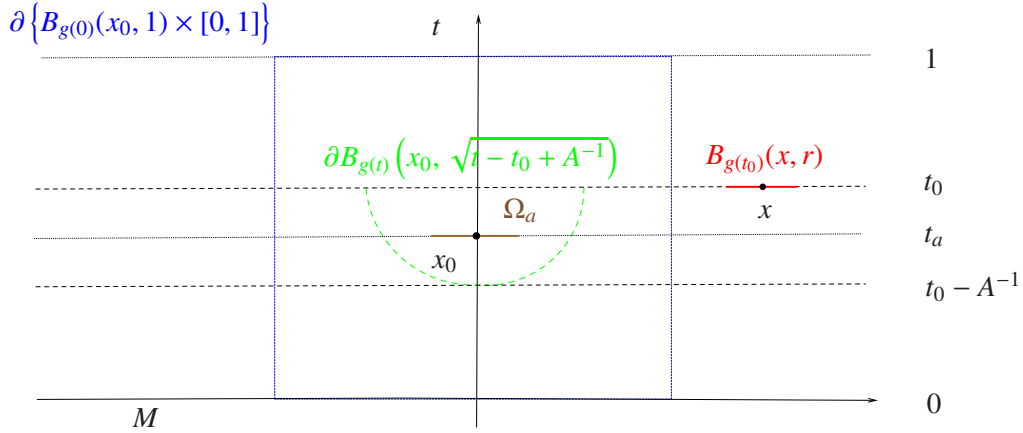}
 \caption{A short period of Ricci flow}
 \label{fig:domainpiece}
 \end{center}
 \end{figure}

 \begin{proof}[Proof of Theorem~\ref{thmin:ML14_1}]
 
 Without loss of generality, we assume $r_0=1$. 
 Fix a time $t=t_0 \in [A^{-1}, 1]$, it suffices to show that
 \begin{align}
     r^{-m} |B_{g(t_0)}(x,r)|_{dv_{g(t_0)}} \geq \kappa.       \label{eqn:CF17_1}
 \end{align}
  We focus attention on the Ricci flow space-time $\{(M^m, g(t)), t_0-A^{-1} \leq t \leq t_0\}$. 
  Let $\bar{A}=A^2$ and $s=t-t_0+A^{-1}$.   By the first part of (\ref{eqn:CE26_7}), we have 
  \begin{align*}
   s |Rc|(x,t) \leq (m-1)A \leq (m-1) \bar{A}, \quad \forall \; x \in B_{g(t)}\left(x_0, \sqrt{s} \right), \;  t_0-A^{-1} \leq t \leq t_0.
  \end{align*}
  Note that 
  \begin{align*}
    B_{g(t_0)}(x,r) \subset B_{g(t_0)}(x_0, A) \subset B_{g(t_0)}\left(x_0, \bar{A} \sqrt{A^{-1}} \right). 
  \end{align*}
  Therefore, up to a time shifting, we can apply Theorem~\ref{thm:CF03_1} with constant $\bar{A}=A^2$ and $T=A^{-1}$.
  The upshot is
  \begin{align}
     \omega_m^{-1}r^{-m} |B_{g(t_0)}(x,r)|_{dv_{g(t_0)}} \geq \boldsymbol{\rho}_a^{m+1} e^{-\frac{500A^4}{\boldsymbol{\rho}_a}},       \label{eqn:CF17_3}
  \end{align}
  where $\boldsymbol{\rho}_a$ is the volume ratio(c.f. the first part of (\ref{eqn:CF22_3})) of the ball $B\left(x_0, 0.1 A^{-\frac12} \right)$ at the time $t=t_0-0.5A^{-1}$.
  For simplicity, we denote this ball as $\Omega_a$ and denote the time $t_0-0.5A^{-1}$ by $t_a$. 
  Note that $|t_a-t_0|=0.5A^{-1}$, which is a very small number. 
  However,  from the first part of (\ref{eqn:CE26_7}) and the Ricci flow equation, we can compare small geodesic balls at time $t_a$ and time $t_0$.
  In particular, the following relationship holds:
  \begin{align*}
            B_{g(0)}\left(x_0, 0.03 A^{-\frac12} \right)   \subset         \Omega_a \subset B_{g(0)} \left(x_0, 0.3 A^{-\frac12} \right). 
  \end{align*}
  Recall that the volume element evolution is dominated by $-R$.  In the very short time period, the volume element is almost fixed. 
  In other words,  we can apply the first part of (\ref{eqn:CE26_7}) again and obtain that
    \begin{align*}
    |\Omega_a|_{dv_{g(t_a)}} \geq  \left| B_{g(0)}\left(x_0, 0.03 A^{-\frac12} \right) \right|_{dv_{g(t_a)}} \geq 0.1  \left| B_{g(0)}\left(x_0, 0.03 A^{-\frac12} \right) \right|_{dv_{g(0)}}.
  \end{align*}
  Applying Bishop-Gromov volume comparison at time $t=0$, as $|Rc| \leq 1 \leq (m-1) \cdot 1^2$, we have
  \begin{align*}
       \frac{\left| B_{g(0)}\left(x_0, 0.03 A^{-\frac12} \right) \right|_{dv_{g(0)}}}{\left| B_{g(0)}(x_0, 1)\right|_{dv_{g(0)}}} \geq \frac{\int_{0}^{0.03 A^{-\frac12} } \sinh^{m-1} r dr}{\int_{0}^{1} \sinh^{m-1} r dr}
       \geq \frac{\left( 0.03 A^{-\frac12} \right)^{m}}{e^{m-1} \cdot 1^m} \geq  10^{-2m} A^{-\frac{m}{2}},
  \end{align*}
  where we again used the fact that $1<\frac{\sinh s}{s}<e^{s}$ for each positive $s$, as done in (\ref{eqn:CF19_3}). 
  Recall that $\left| B_{g(0)}(x_0, 1)\right|_{dv_{g(0)}} \geq A^{-1}$ according to our assumption.  
  Combining the previous steps, we obtain
  \begin{align*}
    |\Omega_a|_{dv_{g(t_a)}} \geq  0.1 \left| B_{g(0)}\left(x_0, 0.03 A^{-\frac12} \right) \right|_{dv_{g(0)}} >10^{-2m-1} A^{-\frac{m}{2}-1}.
  \end{align*}
  In other words, it means that
  \begin{align}
      \boldsymbol{\rho}_a =\omega_m^{-1} \left(0.1 A^{-\frac{1}{2}} \right)^{-m}  |\Omega_a|_{dv_{g(t_a)}} \geq \omega_m^{-1} 10^{-m-1} A^{-1}.     \label{eqn:CF22_3}
  \end{align}
  Plugging this into (\ref{eqn:CF17_3}), we obtain (\ref{eqn:CF17_1}) for some positive $\kappa=\kappa(m,A)$, which can be chosen as
  \begin{align*}
    \omega_m \cdot \left( 10^{m+1} \omega_m A\right)^{-(m+1)} \cdot e^{-500 \cdot 10^{m+1} \cdot \omega_m A^{5}}.
  \end{align*}
  The proof of the Theorem is complete.     
  \end{proof}

\section{K\"ahler Ricci flow on minimal models of general type}
\label{sec:krf}

In this section, we briefly discuss one application of our general theory developed in previous sections. 
Let $X$ be a projective manifold of complex dimension $n$.
$X$ is called a minimal model of general type if the canonical bundle $K_X$ is big ($K_{X}^n \neq 0$) and nef(numerically effective).
Starting from an initial K\"ahler metric $g_0$, we run the Ricci flow
\begin{align}
   \frac{\partial}{\partial t} g=-Ric -g.   \label{eqn:ML26_4}
\end{align}
Let $\omega_0$ be the metric form corresponding to $g_0$, i.e., $\omega_0(\cdot, \cdot)=g_0(J\cdot, \cdot)$ for the complex structure $J$. 
Then the K\"ahler condition $\nabla J \equiv 0$ is preserved by the flow (\ref{eqn:ML26_4}). 
Let $\omega(t)$ be the metric form compatible to both $g(t)$ and $J$. Then (\ref{eqn:ML26_4}) can be rewritten as the following evolution equation:
\begin{align}
\frac{\partial \omega}{\partial t}  = -Rc(\omega) - \omega. 
\label{eqn:ML26_5}
\end{align}
With some efforts(c.f. Tsuji~\cite{Ts}, Tian-Zhang~\cite{TiZha}), the above equation (\ref{eqn:ML26_5}) can be simplified as an evolution equation of scalar functions. 
Actually,  we can choose a smooth volume form $\Omega$ on $X$ and denote $\sqrt{-1} \partial \bar{\partial} \log \Omega$ by $\chi$.
It is clear that $[\chi]=c_1(K_X)=-c_1(X)$.  Then we have
\begin{align}
[\omega(t)]=e^{-t}[\omega_0] + (1-e^{-t}) [\chi].   \label{eqn:CF22_1}
\end{align}
Therefore, up to an additive constant, $\omega(t)$ can be uniquely determined as $\chi + e^{-t}(\omega_0 - \chi) + \sqrt{-1} \partial \bar{\partial}u$
for some smooth function $u=u(\cdot, t)$.    In terms of $u$, the equation (\ref{eqn:ML26_5}) can be translated as the following complex Monge-Amp\`{e}re equation
\begin{align}
\frac{\partial u}{\partial t} &= \log \frac{\left(\chi + e^{-t}(\omega_0 - \chi) + \sqrt{-1} \partial \bar{\partial} u \right)^n}{\Omega} -u, 
\label{eqn:ML26_6}
\end{align}
starting from $u(\cdot, 0) \equiv 0$.  By the work of Tsuji~\cite{Ts} and Tian-Zhang~\cite{TiZha},  the equation (\ref{eqn:ML26_6}) has long time existence. 
It was also shown there that (\ref{eqn:ML26_6}) converges in the distribution sense to a K\"ahler-Einstein current. 
Moreover,  the convergence can be improved to be in the smooth topology outside the exceptional locus $\mathcal{B}$, which is canonically determined and will be explained in the next paragraph. 
A natural question is whether one have the ``global" convergence of (\ref{eqn:ML26_5}) or (\ref{eqn:ML26_6}), including the behavior of the exceptional locus.
The following conjecture is well-known.

\begin{conjecture}[c.f. Conjecture 4.1 of~\cite{S1} and Conjecture 6.2 of~\cite{ST3}]
Suppose $X$ is a smooth minimal model of general type. 
The normalized K\"ahler-Ricci flow (\ref{eqn:ML26_5}) converges to the unique (possibly singular) K\"ahler-Einstein metric $\omega_{KE}$ on $X_{can}$ in the Gromov-Hausdorff topology as $t \rightarrow \infty$.
\label{cje:GB03_1}
\end{conjecture}

Let us briefly explain the meaning and history of the above conjecture.  
The condition that $X$ is a minimal model of general type is equivalent to the fact that $K_X$ is big and nef. 
Hence one can use sections of the pluri-canonical line bundle $K_X^{\nu}$ to define a map $\iota$ from $X$ to $\CP^N$, for $\nu$ sufficiently large. 
This map is an embedding map on $X \backslash \mathcal{B}$ for some exceptional set $\mathcal{B}$.
However, $\iota$ fails to be an embedding on $\mathcal{B}$, which is at least complex co-dimension $1$ and does not depend on $\nu$. 
If the exceptional set $\mathcal{B}=\emptyset$, i.e., $K_X$ is ample or $c_1(X)<0$, then the above conjecture holds automatically by the  classical result of H.D. Cao~\cite{HDC}. 
In general, $\mathcal{B} \neq \emptyset$, $\iota(X)$ is isomorphic to $X_{can}$, the canonical model of $X$. 
There is a \textit{unique} K\"ahler Einstein current $\omega_{KE}$ on $X$ such that $[\omega_{KE}]=[c_1(K_X)]$ and its restriction on $X \backslash \mathcal{B}$ 
is a genuine smooth K\"ahler Einstein metric(c.f.~\cite{EGZ} and the references therein).   The metric completion of $(X \backslash \mathcal{B}, \omega_{KE})$ can be regarded as a canonical metric
on the singular variety $X_{can}=\iota(X)$.   By abusing of notations, we denote this metric completion by $(X_{can}, \omega_{KE})$, which is clearly also \textit{unique}. 
To answer Conjecture~\ref{cje:GB03_1},  it is important to understand the degeneration of metrics $\omega(t)$ 
along the exceptional set $\mathcal{B}$.
Global estimates along the flow need to be developed along the flow. 
The last decade witnessed many important progresses along this direction.
First,  it was confirmed by J. Song and B. Weinkove(c.f.~\cite{SW1} and~\cite{SW2}) that Conjecture~\ref{cje:GB03_1} holds for many 2-dimensional manifolds. 
Then, B. Guo~\cite{BGuo} proved Conjecture~\ref{cje:GB03_1} under the assumption that Ricci curvature is uniformly bounded from below. 
Recently, Conjecture~\ref{cje:GB03_1} was solved completely in low dimension, by Guo-Song-Weinkove~\cite{GSW} in dimension 2 and Tian-Zhang~\cite{TiZhL} in dimension 3. 
We confirm this conjecture for general dimension by Theorem~\ref{thmin:ML28_1}.
In our solution, the uniform scalar curvature bound by Z. Zhang~\cite{Zh} and the diameter bound of $(X_{can}, \omega_{KE})$ by J. Song~\cite{S1} will play important roles. 
For the convenience of the readers, we rewrite Theorem~\ref{thmin:ML28_1} as follows.

\begin{theorem}
Let $X$ be a projective manifold with $K_X$ big and nef. 
Then the K\"ahler Ricci flow (\ref{eqn:ML26_5}) has uniformly bounded diameter and converges to 
the unique singular K\"ahler-Einstein metric $\omega_{KE}$ on $X_{can}$ in the sense of Gromov-Hausdorff as $t \rightarrow \infty$.
\label{thm:GD05_1}
\end{theorem}

\begin{proof}[Proof of Theorem~\ref{thm:GD05_1}]
It is known (c.f. Tsuji~\cite{Ts} and Tian-Zhang~\cite{TiZha}) that $\omega_{t}$ converges smoothly to  $\omega_{KE}$ on
$X \backslash \mathcal{B}$ as time $t \to \infty$.   
Fix $x_0 \in X \backslash \mathcal{B}$ and choose $r_0$ small enough such that the $2r_0$-neighborhood (with respect to the metric $\omega_{KE}$)
of $x_0$ locates in $X \backslash \mathcal{B}$.
Recall that $(X_{can}, \omega_{KE})$ is  the metric completion of $(X \backslash \mathcal{B}, \omega_{KE})$.   
In light of the results of J. Song(c.f. Theorem 4.1 of~\cite{S1}), we can assume $L<\infty$ to be the diameter of $(X_{can}, \omega_{KE})$. 
We claim that
\begin{align}
  \lim_{t \to \infty}  \diam(X, g(t)) \leq L.
  \label{eqn:GC23_2}  
\end{align}
For otherwise, we can find $t_i \to \infty$ and $y_i \in X$ such that  $\displaystyle \lim_{t \to \infty} d_{g(t_i)}(x_0, y_i) \geq L +2\epsilon$ for some positive number $\epsilon \in (0, 1)$. 
By adjusting $y_i$ if necessary, we can further assume that
\begin{align}
   \lim_{t \to \infty} d_{g(t_i)}(x_0, y_i) =L +2\epsilon<L+2.   \label{eqn:ML26_7}
\end{align}
In light of (\ref{eqn:CF22_1}), we have the global volume estimate
\begin{align}
    \lim_{t \to \infty}  \left| X\right|_{\omega_t^n}= \left| X \backslash \mathcal{B}\right|_{\omega_{KE}^n}.    \label{eqn:CF22_2}
\end{align}
For each compact set $K \Subset X \backslash \mathcal{B}$, it follows from the choice of $L$ and (\ref{eqn:ML26_7}) that $B_{g(t_i)}(y_i,\epsilon) \cap K=\emptyset$ for large $i$.  Therefore the ``volume-squeezing" implies
\begin{align}
  \lim_{i \to \infty} \left| B_{g(t_i)}(y_i, \epsilon) \right|_{\omega_{t_i}^{n}} \leq \inf_{K \Subset X \backslash \mathcal{B}} \limsup_{i \to \infty} \left| X \backslash K\right|_{\omega_{t_i}^n}
  =\inf_{K \Subset X \backslash \mathcal{B}}  \left| \left\{ X \backslash \mathcal{B} \right\} \backslash K\right|_{\omega_{KE}^n}=0. 
  \label{eqn:GC23_1}  
\end{align}
Let $g_i(t)=g(t+t_i-r_0^2)$. By the currents convergence, we know $B_{g_i(0)}(x_0, r_0) \times [-r_0^2, r_0^2]$ smoothly converges to
$B_{g_{KE}}(x_0, r_0) \times [-r_0^2, r_0^2]$.  Namely,  the regularity theory of complex Monge-Amp\`{e}re equations(c.f. Tsuji~\cite{Ts},  Kolodziej~\cite{Kol} and Tian-Zhang~\cite{TiZha})
 implies that $B_{g_i(0)}(x_0, r_0) \times [-r_0^2, r_0^2]$ has uniformly bounded geometry. 
Recall that $g_i$ is a solution of (\ref{eqn:ML26_4}), $y_i$ has bounded distance to $x_0$ by (\ref{eqn:ML26_7}),  scalar curvature is uniformly bounded by $\Lambda$ along (\ref{eqn:ML26_4}) in view of the estimate of Z. Zhang~\cite{Zh}.
Consequently, up to a parabolic rescaling, we can apply Theorem~\ref{thmin:ML14_1}(or apply Theorem~\ref{thm:CE26_2} directly for $k=1$ and $\lambda(t) \equiv 1$) to the flow $g_i$ and obtain that 
\begin{align}
  \epsilon^{-2n} \left| B_{g(t_i)}(y_i, \epsilon) \right|_{\omega_{t_i}^{n}}= \epsilon^{-2n} \left| B_{g_i(r_0^2)}(y_i, \epsilon) \right|_{dv_{g_i(0)}} \geq \kappa  \label{eqn:ML26_8}
\end{align}
uniformly for some positive $\kappa=\kappa(\omega_{KE}, x_0, L, n, \Lambda)$. However, the above inequality contradicts (\ref{eqn:GC23_1}).  
This contradiction establishes the proof of (\ref{eqn:GC23_2}).

The equation (\ref{eqn:GC23_2}) implies that $(X, \omega(t))$ has uniformly bounded diameter along the flow (\ref{eqn:ML26_5}).
Similar to the proof of (\ref{eqn:ML26_8}), we can apply Theorem~\ref{thmin:ML14_1} again to obtain uniform non-collapsing.  In other words, there is a constant $\kappa>0$ such that
\begin{align}
 r^{-m} |B(y,r)|_{dv_{g(t)}}> \kappa    \label{eqn:MM03_1}
\end{align}
uniformly for every $(y, t)$ in the flow space-time and $r \in (0,1)$.
Based on (\ref{eqn:MM03_1}) and (\ref{eqn:CF22_2}), it is not hard to see(c.f. Guo~\cite{BGuo}, Tian-Zhang~\cite{TiZhL}) that the identity map restricted on $X \backslash \mathcal{B}$ is an $\epsilon$-almost isometry from $(X, \omega(t))$ to $(X_{can}, \omega_{KE})$, for every $\epsilon$ and correspondingly large $t$.  
It follows that
\begin{align*}
   \lim_{t \to \infty}  d_{GH}\left( \left( X, \omega(t) \right),  (X_{can}, \omega_{KE}) \right) \leq \epsilon, 
\end{align*}
whence we finish the proof of convergence by letting $\epsilon \to 0$.  
\end{proof}

\begin{remark}
The uniform $\kappa$-non-collapsing (\ref{eqn:MM03_1}) and uniform diameter bound (\ref{eqn:GC23_2}) are the key difficulties for proving Theorem~\ref{thm:GD05_1}.  
The correspondent estimates in the Fano K\"ahler Ricci flow were discovered by Perelman and refined by Sesum-Tian~\cite{SeT}.
\label{rmk:CF25_1}
\end{remark}

A natural idea to show (\ref{eqn:MM03_1}) is to apply Perelman's $\boldsymbol{\nu}$-functional.  
Let $\Omega$ be $B_{g(t)}(x,r)$.  Then we have the following formal inequalities:
\begin{align}
    \boldsymbol{\nu}(\Omega, g(t), r^2) \geq  \boldsymbol{\nu}(M, g(t), r^2) \geq  \boldsymbol{\nu}(M, g(0), t+r^2) \geq   \boldsymbol{\nu}(M, g(0)).  \label{eqn:ML26_9}
\end{align}
Consequently, we can apply the scalar curvature bound and Theorem~\ref{thm:CF21_3} to obtain the volume ratio lower bound. 
However, a pitfall is that $\boldsymbol{\nu}(M, g(0))$ could be $-\infty$, which makes the final inequality trivial 
and prevent us from extracting  useful information from the inequalities. 
In the proof of Theorem~\ref{thm:GD05_1}, we use  Theorem~\ref{thmin:ML14_1}, where the inequalities (\ref{eqn:ML26_9}) were localized and the aforementioned pitfall was avoided.

\begin{remark}
Under the conditions of Theorem~\ref{thm:GD05_1}, beyond (\ref{eqn:MM03_1}) and (\ref{eqn:GC23_2}), many other uniform estimates hold along the flow.
For example, there exists uniform non-inflation bound, dual to the Fano K\"ahler Ricci flow case(c.f. Q. Zhang~\cite{Zhq3} and Chen-Wang~\cite{CW1}).
The limit length space $(X_{can}, \omega_{KE})$ has a regular-singular decomposition $\mathcal{R} \cup \mathcal{S}$ such that $\mathcal{R}$ is a geodesic convex Einstein manifold 
and $\mathcal{S}$ has Hausdorff codimension at least $4$(c.f. Tian-Wang~\cite{TiWa} and Song~\cite{S1}). 
Furthermore,  the convergence topology of Theorem~\ref{thm:GD05_1} could be better and one can discuss the ``space-time" convergence, 
in the so called $\hat{C}^{\infty}$-Cheeger-Gromov topology, of the K\"ahler Ricci flow on general-type minimal projective manifolds, 
which mirrors the picture of the K\"ahler Ricci flow on Fano manifolds(c.f. Chen-Wang~\cite{CW2},~\cite{CW3}).
The full details will be provided in a separate paper~\cite{Wang}. 
\label{rmk:ML26_1}
\end{remark}

\vspace{0.5in}

Bing  Wang, Department of Mathematics, University of Wisconsin-Madison,
Madison, WI, 53706, USA;  bwang@math.wisc.edu.\\

\end{document}